\documentclass[11pt,letterpaper]{article}
\usepackage{amsmath,amssymb,amsthm}
\DeclareMathOperator*{\argmin}{arg\,min}

\newcommand{\SingleSpacedXI}{\linespread{1}\selectfont}
\RequirePackage{tgtermes}
\RequirePackage{newtxtext}
\RequirePackage{newtxmath}
\RequirePackage{bm}
\RequirePackage{endnotes}

\SingleSpacedXI
\usepackage{algorithm}
\usepackage{algpseudocode}
\AddToHook{env/algorithmic/begin}{\fontsize{11}{18.99}\selectfont}

\usepackage{tikz}

\usepackage{graphicx} 

\usepackage{amsmath,amssymb}
\usepackage{caption}
\DeclareCaptionFont{singlecaption}{\SingleSpacedXI\footnotesize}
\usepackage{subcaption}
\usepackage{xr-hyper}
\usepackage[hidelinks,bookmarksopen=true,bookmarksnumbered=true,pdfpagemode=UseOutlines]{hyperref}
\usepackage{bookmark}
\bookmarksetup{open,numbered,depth=3}
\usepackage[letterpaper,margin=1in,headheight=12pt,headsep=14pt]{geometry}

\usepackage{algpseudocode}
\usetikzlibrary{positioning}
\usetikzlibrary{shapes.geometric}
\usetikzlibrary{patterns}
\usepackage{xcolor}
\usepackage{subcaption}

\newcommand{\mb}{\mathbf}
\newcommand{\R}{\mathbb{R}}
\newcommand{\E}{\mathbb{E}}
\newcommand{\p}{\mathbb{P}}

\newcommand{\bx}{\boldsymbol{\xi}}
\newcommand{\bz}{\boldsymbol{\zeta}}

\newcommand{\internal}{\textsf{int}}

\newcommand{\gaps}{\textsf{GapS}}
\newcommand{\gapu}{\textsf{GapU}}
\newcommand{\lips}{\textsf{Lip}}

\newcommand{\memo}{\textsf{MEMORY}}

\newcommand{\rr}[1]{\textsf{read}(#1)}
\newcommand{\ww}[1]{\textsf{write}(#1)}
\newcommand{\del}[1]{\textsf{del}(#1)}

\newcommand{\msup}{$\mathcal{P}_U$}
\newcommand{\mssp}{$\mathcal{P}_S$}

\usepackage{natbib}
 \bibpunct[, ]{(}{)}{,}{a}{}{,}

\theoremstyle{plain}
\newtheorem{theorem}{Theorem}
\newtheorem{lemma}{Lemma}
\newtheorem{corollary}{Corollary}
\newtheorem{assumption}{Assumption}
\theoremstyle{definition}
\newtheorem{definition}{Definition}
\renewenvironment{proof}[1]{\par\smallskip\noindent\textit{#1}\enspace\ignorespaces}{\par\smallskip}
\newenvironment{APPENDICES}{\appendix}{}
\hypersetup{pdftitle={Efficient Online Mirror Descent Stochastic Approximation for Multi-Stage Stochastic Programming},pdfauthor={Junhui Zhang and Patrick Jaillet}}

\begin{document}
\title{Efficient Online Mirror Descent Stochastic Approximation for Multi-Stage Stochastic Programming}
\author{Junhui Zhang\thanks{Operations Research Center, MIT. Email: \texttt{junhuiz@mit.edu}.}\and Patrick Jaillet\thanks{Operations Research Center \& Department of Electrical Engineering and Computer Science, MIT. Email: \texttt{jaillet@mit.edu}.}}
\date{}
\maketitle
\begin{abstract}
We study the unconstrained and the minimax saddle point variants of the convex multi-stage stochastic programming problem, where consecutive decisions are coupled through the objective functions, rather than through the constraints. We approach the problems from the infinite-dimensional policy perspective, but consider an online setting where only the policies corresponding to the actual realization of the underlying stochastic process are needed. This leads to a tractable formulation, where the dimension of the output is linear in the number of stages $T$. 

We propose hypothetical Mirror Descent Stochastic Approximation (MDSA) for the infinite-dimensional policies using stochastic conditional gradients. By taking advantage of the decomposability of the updates across stages and realizations of the underlying stochastic process, we show that the proposed MDSA algorithms admit efficient online implementation. In the moderate-accuracy regime, where the allowed total error scales linearly with $T$, the oracle complexity is linear in $T$ and the extra space needed is independent of $T$.
\end{abstract}
\medskip

\section{Introduction}

Sequential decision making has found applications in a variety of real-life problems: from classical ones such as power system management \citep{Bhattacharya2018, Lu2020,pereira_multi-stage_1991} and inventory control \citep{LanShapiro2024,Shapiro2021Lectures}, to more modern ones such as data center management \citep{ChenComdenLiuGandhiWierman2016,Lin2012,LinWiermanAndrewThereska2011} and online resource allocation \citep{BalseiroLuMirrokni2023,Vera2021Bayesian,Xie2024Benefits}. In this work, we study \textit{sequential} decision making in \textit{stochastic} environments in an \textit{online} framework. 

To model the sequential revelation of information and the underlying stochasticity, we adopt a Multi-Stage Stochastic Programming (MSSP) formulation \citep{Shapiro2021Lectures}. Informally, MSSP divides the decision-making process into stages ($t=1,\ldots,T$): at stage $t$, the decision maker observes the realization of a random variable $\bx_t=\xi_t$, then needs to decide $x_t$ based only on the information known by stage $t$ (that is, $\bx_{1:t}$, but not $\bx_{(t+1):T}$). The overall decision process is the following:
\begin{align*}
    \text{observation}(\xi_1) \rightsquigarrow \text{decision}(x_1) \rightsquigarrow \text{observation}(\xi_2) \rightsquigarrow \\
    \cdots \rightsquigarrow \text{observation}(\xi_T) \rightsquigarrow \text{decision}(x_T). 
\end{align*}

The goal is to minimize the costs subject to coupled constraints, both of which depend on $x_{1:T}$ and $\xi_{1:T}$. Due to its wide applications, MSSP has been studied in many prior works \citep{SDDP-review,pereira_multi-stage_1991,Shapiro2021Lectures}. However, existing algorithms usually suffer from restrictive assumptions such as stage-wise independent randomness \citep{SDDP-review,LanShapiro2024} and have exponential (in $T$) complexity \citep{lan_dynamic_2021,SHAPIRO20061complexity}.

We study the unconstrained and the minimax saddle point variants of the convex multi-stage stochastic programming problems, where consecutive decisions are coupled through the objective functions, rather than through the constraints. We approach these problems from a \textit{policy} perspective, where the decision variables are mappings from realizations of $\bx_{1:T}$ to actual decisions. The resulting formulations become infinite-dimensional. However, instead of solving for policies for all realizations of $\bx_{1:T}$, we aim at solving only for policies corresponding to $\xi_{1:T}$, \textit{the relevant (actual) realization} of $\bx_{1:T}$, which is revealed sequentially. Thus, the problems remain tractable, as the output dimension is finite, and in fact linear in $T$. 

In order to find a decision $x_t$ for the relevant $\xi_{1:t}$, we first propose \textit{hypothetical} mirror descent updates applied to policies for each realization of $\bx_{1:T}$, and propose using \textit{stochastic conditional gradients} as approximate gradients, which generalize the stochastic gradient oracles in the classical stochastic programming literature \citep{LanBook,NemirovskiRobustSA2009}. We provide one sampling approach to construct such stochastic conditional gradients, and prove the convergence of the resulting Mirror Descent Stochastic Approximation (MDSA) algorithms. Then, we take advantage of the decomposability of the mirror descent updates across stages and realizations $\xi_{1:T}$, and show that, in the moderate-accuracy regime where the allowed total error scales linearly with $T$, finding the relevant components of the policies requires only $O(T)$ gradient oracle calls, and the extra space needed is independent of $T$.

We demonstrate the effectiveness and robustness of our mirror descent stochastic approximations by applying them to a tracking problem and a revenue management problem (Appendix \ref{sec:num_exp}).

\subsection{Existing algorithms and comparison}
In the classical setup of Stochastic Programming (SP), the objective function is $F(x) = \E[f(x,\bx(w))]$, where $w\in \Omega$ and $(\Omega,\mathcal F,\p)$ is the underlying probability space, and $\bx$ is a random vector. Multi-stage stochastic programming (with $T$ stages) is a generalization of SP, driven by an underlying stochastic process $\bx = (\bx_1,\ldots,\bx_T)$: at stage $t$, $\bx_t$ is revealed and the decision maker needs to decide $\mb x_t$ in a non-anticipative manner such that $\mb x_t$ cannot depend on $\bx_{(t+1):T}$. The goal is to minimize $\sum_{t=1}^T h_t(\mb x_t,\bx_t)$ subject to the constraints $ g_t(\mb x_{t-1},\mb x_t,\bx_t)\leq \mb 0$ for all $t$.

As a preview, our MDSA algorithms for the unconstrained and the saddle point variants of the multi-stage stochastic programming problems can find $O(\epsilon T)$ suboptimal policies for all $T$ stages in an online fashion, with gradient oracle complexity $O(T\min(e^{C/\epsilon^2},(\frac{C}{\epsilon})^{2T})$ and extra space needed $O(\epsilon^{-4})$ (Corollary \ref{cor:overall}). Here, we assume that the parameters (Lipschitz constants of the objectives and domain sizes) for all stages are $O(1)$. In this regime, $C=O(1)$ as $T\to \infty$. In particular, at fixed per-stage accuracy $\epsilon=\Theta(1)$, the oracle complexity is $O(T)$ and the extra space needed is $O(1)$. Below, we compare our algorithms with existing approaches for multi-stage stochastic programming problems. 

\textbf{Progressive hedging.} The \textit{non-anticipativity constraint} on $\mb x_t$ can be formulated as the linear constraint that $\E[\mb x_t|\bx_{1:t}] = \mb x_t$. \citet{Rockafellar1976Nonanticipativity,RockafellarWets1991Scenarios,RockafellarWets2017,bareilles_randomized_2020,Eckstein2025} dualize this constraint and apply augmented Lagrangian methods. However, this requires the number of scenarios $|\Omega|<\infty$, and the complexities of these algorithms scale linearly with $|\Omega|$, which is often exponential in $T$. 

\textbf{Stochastic approximation.} Mirror descent stochastic approximation has been well studied for stochastic programming problems where the objective functions are of the (simple) form $F(x) = \E[f(x,\bx)]$ \citep{JuditskyNemirovskiTauvel2011,lan_optimal_2012,LanBook,NemirovskiRobustSA2009,Shapiro2021Lectures}. In the multi-stage setting, to apply stochastic approximation type of algorithms, one needs to have access to (potentially stochastic, biased) first order oracles for the cost-to-go functions.
For stage $t$, the cost-to-go function is \begin{align*}
    V_t(x_t) := &\inf_{\mb x_{(t+1):T}}\E[\sum_{s=t+1}^Th_s(\mb x_s,\bx_s)]\\
    &~s.t.~ g_{t+1}(x_t,\mb x_{t+1},\bx_{t+1})\leq \mb 0\\
    &~\quad \quad  g_{s}(\mb x_{s-1},\mb x_{s},\bx_{s})\leq \mb 0,\quad s = t+2,\ldots,T\\
    &~\quad\quad \E[\mb x_{s}|\bx_{1:s}] = \mb x_s,s=t+1,\ldots,T
\end{align*}

\citep{lan_dynamic_2021} proposes the Dynamic Stochastic Approximation (DSA) algorithm, which solves backwardly for approximate subgradients of the cost-to-go functions at the query points, and then applies inexact primal-dual updates. To find an $O(\epsilon T)$ suboptimal \textit{first stage} solution, DSA requires $O((C'/\epsilon)^{2T-1})$ proximal updates for general convex objectives and $O((C''/\epsilon)^{T-1/2})$ for strongly convex objectives (We point it out that Theorem 22 in \citet{lan_dynamic_2021} shows that to find an $O(\epsilon)$ suboptimal solution for first stage, the complexity is $O((T/\epsilon)^{2T-1})$, and here we rescale the suboptimality with $T$). The space complexity of DSA is $O(T)$. 

Table \ref{tab:dsa-mdsa-complexity} compares the complexity bounds for DSA and online MDSA. The DSA row uses accuracy $O(\epsilon)$ for each realized-stage solve, rather than the first-stage accuracy $O(T\epsilon)$ used in the preceding paragraph. The proximal-to-gradient conversion and storage conventions are specified in the caption.

\begin{table}[htbp]
\centering
\caption{Complexity bounds for an $O(T\epsilon)$ total-suboptimality target, with $0<\epsilon\leq1$. Constants denoted by $C$ are independent of $T$ and $\epsilon$, assuming uniformly controlled problem parameters. The DSA row assumes rerunning DSA with accuracy $O(\epsilon)$ at each realized stage and additive accumulation of stagewise losses. Its gradient/subgradient count additionally assumes each proximal subproblem can be solved sufficiently accurately in $O(\epsilon^{-1})$ calls, with uniformly controlled constants and inner errors compatible with the stated stagewise accuracy; the factor $T$ is a conservative bound for the repeated solves. MDSA assumes exact Bregman updates and conditional sampling whose bias contributes at most $O(T\epsilon)$. Storage is measured in stage records. The storage column reports active-stage storage for DSA and additional evaluation storage for MDSA. Storage excludes model, sampling/history, oracle-workspace, and archived-output costs; MDSA's additional evaluation storage also excludes initialization and indexed seeds.}
\label{tab:dsa-mdsa-complexity}
\begingroup
\footnotesize
\setlength{\tabcolsep}{3pt}
\renewcommand{\arraystretch}{1.65}
\begin{tabular}{l|l|l|l}
\hline
\textbf{Algorithm / regime} & \textbf{Proximal-update count} & \textbf{Gradient/subgradient calls} & \textbf{Working storage} \\
\hline
DSA, rerun at each stage
& $O\!\left(T(CT/\epsilon)^{2T-1}\right)$
& $O\!\left(\frac{T}{\epsilon}(CT/\epsilon)^{2T-1}\right)$
& $O(T)$ \\
Online MDSA, general $\epsilon$
& $O\!\left(T\min\left\{e^{C/\epsilon^2},(C/\epsilon)^{2T}\right\}\right)$
& Same order
& $O(\epsilon^{-4})$ \\
Online MDSA, $\epsilon=\Theta(1)$
& $O(T)$
& $O(T)$
& $O(1)$ \\
Online MDSA, $\epsilon=\Theta(T^{-\alpha})$, $\alpha\geq0$
& $O\!\left(T\min\left\{e^{CT^{2\alpha}},(CT^\alpha)^{2T}\right\}\right)$
& Same order
& $O(T^{4\alpha})$ \\
\hline
\end{tabular}
\endgroup
\end{table}

\textbf{Cutting plane methods.} Another well known algorithm for MSSP is the stochastic dual dynamic programming (SDDP) \citep{python-package, ju2023dualdynamicprogrammingstochastic,lan_complexity_2022,pereira_multi-stage_1991}, which is a cutting-plane based algorithm designed for MSSP where $\bx_1,\ldots,\bx_T$ are independent. \citet{lan_complexity_2022,lan_correction_2022} show that if $\bx_t$ has at most $d$ possible realizations and $x_t\in \R^{n_t}$, then to find a $O(\epsilon T)$ suboptimal solution for all stages, SDDP needs $O(Td^T\epsilon^{-\max_{t=1,\ldots,T} n_t})$ forward-backward iterations, where each iteration involves (approximately) solving $T$ convex optimization problems (of dimension $O(\max_{t=1,\ldots,T} n_t)$). More recently, \citet{leclere2026sddp} establishes an $O(T\epsilon^{-(n+1)})$ bound on the expected number of iterations whose policy gap exceeds $O(T\epsilon)$, where $n$ is the state dimension, under uniformly bounded Lipschitz constants, state diameters, and per-stage loss ranges. Thus, exponential dependence on the horizon in earlier upper bounds is not an intrinsic limitation of SDDP. As a comparison, our algorithms do not require stage-wise independent randomness, can be applied to the general case where $d=\infty$, and the complexity has no explicit dependence on the problem dimensions $n_t$.  

\textbf{Other related works.} \citet{shapiro_inference_2003,SHAPIRO20061complexity,Shapiro2005} show that for sample average approximation type of algorithms for MSSP, the sample complexity has an \textit{upper bound} which is exponential in $T$. In fact, \citep{dyer_computational_2006,hanasusanto_comment_2016} show that even approximating the solution of $2$-stage stochastic programs is $\#P$-hard for a sufficiently high accuracy. 

Independent and concurrent to our work, \citet{cetin2025onlinestochasticpackinggeneral} also proposes implementing stochastic gradient descent on-the-fly for the network revenue management problem and the online matching problem, and their algorithms achieve similar $\exp(\text{poly}(1/\epsilon))$ complexity.

\textbf{Online optimization.} Sequential decision making is also studied through the perspective of online optimization, where the unknown part (i.e., future) of the objectives could be potentially chosen by an adversary, and the performance is compared to the optimal \textit{offline} solution. Well-known algorithms include online mirror descent \citep{Hazan2022Online,Zinkevich2003}, online restarted gradient descent \citep{BesbesGurZeevi2015}, and online primal-dual approach \citep{BuchbinderCompetitive}, to name a few. In this work, we consider the semi-online case, where the solution is constructed in an online fashion and is compared against the optimal \textit{online} solution. Our efficient online implementation is motivated by a recent line of research on smoothed online convex optimization \citep{bansal_et_al,goel19a,LiChenLi2019,LiLi2020,LiQuLi2021,Lin2012,ZhangLiLi2021}, where the cost at stage $t$ is the sum of a stage cost which depends only on $x_t$, and a switching cost $\|x_t-x_{t-1}\|$ or $\|x_t-x_{t-1}\|^2$. As a comparison, our algorithms can be applied to (convex or saddle point) problems with general couplings between consecutive decision variables, and do not require strong convexity (as required in \citet{LiLi2020,LiQuLi2021}). 

\textbf{Markov Decision Process and online planning.} Multi-stage stochastic programming can also be viewed as Markov Decision Process (MDP), where the state space is $\cup_{t=1}^T \Xi_{1:t}$, with state transition $P(\xi_{1:s},\xi'_{s:t}) = \p[\bx_{t} = \xi'_{t}|\bx_{1:s} = \xi_{1:s}]$ if $t=s+1$ and $\xi'_{1:s} = \xi_{1:s}$ and $P(\xi_{1:s},\xi'_{s:t})=0$ otherwise. From this perspective, our proposed online MDSA can be viewed as online planning algorithms targeted for ``convex'' MDP. Other well known algorithms for MDP include Model Predictive Control, Monte Carlo Tree Search, policy improvement by lookahead, to name a few. We leave detailed comparisons and potential connection between two fields to future work. 

\subsection{Contributions}

We make the following contributions to the multi-stage stochastic programming literature.

\textbf{Tractable online formulation.} We propose an online framework for the multi-stage stochastic programming problems (Section \ref{sec:setup}), where the decision for stage $t$ is needed at stage $t$, and is needed only for the corresponding realization of $\bx_{1:t}$. This reduces the output from the infinite-dimensional policies to the much more \textit{tractable finite dimensional} decisions for one specific realization of the stochastic process -- in fact the dimension of the output is $O(T)$. Moreover, our framework can model \textit{randomized algorithms}, where the decision maker has additional internal random variables. 

\textbf{Hypothetical mirror descent stochastic approximation.} We propose hypothetical mirror descent updates applied to the (infinite-dimensional) policies for each possible realization of $\bx_{1:T}$ (Section \ref{sec:pathwise}). The convergence behavior suggests using stochastic conditional gradients as gradient estimates. Building upon this, we provide a sampling approach to construct stochastic conditional gradients (Section \ref{sec:def-grad-cvx}), and prove the convergence of the resulting mirror descent stochastic approximation algorithms (Section \ref{sec:MDSA}). An accelerated extension for smooth objectives, including strongly convex objectives, is developed in Appendix \ref{appendix:acc}. Compared to existing algorithms for MSSP, our algorithms do not assume stage-wise independent randomness and solve for decisions for all $T$ stages.

\textbf{Efficient online implementation.} Taking advantage of the decomposability of the updates across stages and realizations $\bx_{1:T}$, we show that in the online framework proposed, the hypothetical MDSA admits an efficient online implementation (Section \ref{sec:online-implementation}). To find $O(\epsilon T)$ suboptimal solutions for all stages, our MDSA achieves a gradient oracle complexity of $O(T\min(e^{C/\epsilon^2},(\frac{C}{\epsilon})^{2T})$ and additional space $O(\epsilon^{-4})$, where $C = O(1)$ as $T\to \infty$. In particular, at fixed per-stage accuracy $\epsilon=\Theta(1)$, the oracle complexity is $O(T)$ and the additional evaluation storage is $O(1)$, excluding the storage described in Section \ref{sec:eval_oracle}. Compared to DSA, if $\epsilon=\Omega(T^{-\alpha})$ for some $\alpha\geq 0$, our MDSA has better gradient oracle complexity when $\alpha<1/2$, while its additional evaluation storage is $O(T^{4\alpha})$ under the access model in Section \ref{sec:eval_oracle}.

\section{Setup}\label{sec:setup}
We consider the unconstrained and the saddle point variant of the $T$-stage stochastic programming problems on the probability space $(\Omega,\mathcal F, \p)$, driven by the stochastic process $\bx = (\bx_1,\bx_2,\ldots,\bx_T)$, where $\bx_t:\Omega\to \Xi_t\subset \R^{k_t}$ is a random vector, representing the randomness in stage $t$. We use $\xi_t\in \Xi_t$ to represent the value taken by the random vector $\bx_t$, and we assume that $\bx_1$ is deterministic, i.e. there exists $\xi_1\in \R^{k_1}$ such that $\Xi_1  =\{\xi_1\}$ and $\bx_1(w) = \xi_1$ for all $w\in \Omega$. 

For convenience, we denote $\mathcal F_t = \sigma(\bx_{1:t})$, the sub-$\sigma$-algebra representing the information available at stage $t$. For example, $\mathcal F_1 = \{\emptyset,\Omega\}$. 

Since the decision $x_t$ at stage $t$ is non-anticipative, i.e. depends on $\bx_{1:t}$ but not on $\bx_{(t+1):T}$, it is a random variable $\mb x_t:\Omega\to \R^{n_t}$ which is measurable w.r.t. $\mathcal F_t$. Equivalently, there exists a \textit{policy} $X_t: \Xi_{1:t}\to \R^{n_t}$ such that $\mb x_t(w) = X_t(\bx_1(w),\ldots,\bx_t(w))$ for all $w\in \Omega$. Thus, the problems become infinite-dimensional, where the decision variables are policies $X = (X_1,\ldots,X_T)$ in the policy space $ \prod_{t=1}^T\mathcal P_t(\R^{n_t})$. Here, for any subset of a finite dimensional Euclidean space $S$, we define $\mathcal P_t(S) = \{X_t:\Xi_{1:t}\to S,~X_t \text{ is measurable}\}$.  (Here $X_t$ is measurable w.r.t. the Borel $\sigma$-algebra on $\Xi_{1:t}$ and $S$. As a consequence, $\mb x_t$ is measurable w.r.t. $\mathcal F_t$)

\textbf{Notations.} To simplify notations, in the rest of the paper, we use $X_t:\Xi_{1:t} \to \R^{n_t}$ to denote the policy for stage $t$, $\mb x_t:\Omega\to \R^{n_t}$ defined as $\mb x_t(w) = X_t(\bx_{1:t}(w))$ to denote the induced random variable, and $X$ and $\mb x$ without subscripts are shorthands for $X_{1:T}$ and $\mb x_{1:T}$, respectively. Similar definitions extend to other policies/variables ($Y_t$ for the dual variable, $Z_t =(X_t,Y_t)$ for the primal-dual pair, and $G_t$ for the approximate conditional gradient). 

In this section, we address the following three concerns in terms of problem formulation.
\begin{itemize}
    \item What are the objectives and the constraints? (Section \ref{sec:obj_cons})
    \item What are the desired outputs? (Section \ref{sec:seq_output})
    \item How to model randomized policies? (Section \ref{sec:product_space})
\end{itemize}
Then in Section \ref{sec:outline}, we provide the road map for the rest of the work.

\subsection{Objectives and constraints}\label{sec:obj_cons}

Motivated by control problems such as tracking (see Appendix \ref{sec:exp-soco-gen} for an example), we consider the following \textit{multi-stage unconstrained programming} (\msup) problem: 
\begin{align}\label{eq:obj_unconstrained-online}\tag{\msup}
   & \inf_{X_t\in \mathcal P_t(\mathcal X_t),~t=1,\ldots,T} \E[f(\mb x,\bx)],\quad \mb x_t = X_t(\bx_{1:t}),~t=1,\ldots,T\\
   & f(x_1,\ldots,x_T,\xi):= f_1(x_1,\xi_1) + \sum_{t=2}^T f_{t}(x_{t-1},x_t,\xi_t), \nonumber
\end{align}
where $f_t:\R^{n_{t-1}}\times \R^{n_t}\times\Xi_t\to \R$ ($n_0 = 0$), $\mathcal X_t\subset \R^{n_t}$ for all $t=1,\ldots,T$. The expectation is taken over the randomness in $\bx$. We make the following assumption.  

\begin{assumption}\label{assu:convex}
    For each $t=1,2,\ldots,T$, 
    \begin{enumerate}
        \item $f_t(\cdot,\cdot,\xi_t)$ is convex and differentiable in $(x_{t-1},x_t)$ for all $\xi_t\in \Xi_t$; 
        \item $\mathcal X_t\subset \R^{n_t}$ is a nonempty compact convex subset;
        \item $f_t$, $\frac{\partial }{\partial x_{t-1}}f_t$, and $\frac{\partial }{\partial x_{t}}f_t$ are measurable w.r.t. the Borel $\sigma$-algebras on $\R^{n_{t-1}}\times \R^{n_t}\times\Xi_t$, and $\R, \R^{n_{t-1}},\R^{n_t}$, respectively.
    \end{enumerate}
    We assume that there is a fixed reference point $x^{\rm ref}\in\prod_{t=1}^T\mathcal X_t$ such that $\E[|f_t(x_{t-1}^{\rm ref},x_t^{\rm ref},\bx_t)|]<\infty$ for every $t$, with the first argument omitted when $t=1$.\footnote{Together with the uniform gradient bounds used in the main theorems and compact decision sets, this ensures absolute integrability of the objective at every admissible policy. The same reasoning applies to the saddle-point formulation below.}
    In addition, \eqref{eq:obj_unconstrained-online} has a solution $X^*$ (and equivalently, an induced solution $\mb x^* $ where $\mb x_t^* = X_t^*(\bx_{1:t})$) for all $t$. 
\end{assumption}

To measure the quality of a solution $\mb x $, we consider suboptimality in terms of the objective value $\gapu(\mb x):=\E[f(\mb x,\bx)-f(\mb x^*,\bx)]$. We say that $\mb x$ is an $\epsilon$-suboptimal solution if $\gapu(\mb x)\leq \epsilon$. For a policy $X$, we denote $\gapu(X)=\gapu(\mb x)$, where $\mb x_t = X_t(\bx_{1:t})$ is the induced solution. 

In addition, motivated by penalty version of the classical MSSP problems (Appendix \ref{sec:lagrangian}), we consider the following \textit{multi-stage saddle point} (\mssp) problem: 
 \begin{align}\label{eq:minmax-online}\tag{\mssp}
        &\inf_{X_t\in \mathcal P_t(\mathcal X_t),~t=1,\ldots,T} \E[\tilde{\phi}(\mb x,\bx)],\quad \mb x_t = X_t(\bx_{1:t}),~t=1,\ldots,T\\
        &\tilde{\phi}(x_1,\ldots,x_T,\xi):= \tilde{\phi}_1(x_1,\xi_1) + \sum_{t=2}^T \tilde{\phi}_{t}(x_{t-1},x_t,\xi_t)\nonumber\\
        &\tilde{\phi}_{t}(x_{t-1},x_t,\xi_t):= \sup_{y_t\in \mathcal Y_t}\phi_{t}(x_{t-1},x_t,y_t,\xi_t),\quad t = 1,\ldots,T\nonumber
\end{align}
where $\phi_t:\R^{n_{t-1}}\times\R^{n_t}\times \R^{m_t}\times \Xi_t\to \R$ ($n_0 =0$), $\mathcal X_t \subseteq \R^{n_t}$ and $\mathcal Y_t\subseteq \R^{m_t}$ for all $t$. Let $m = \sum_{t=1}^T m_t$ and denote
\begin{equation*}
    {\phi}(x_1,y_1,\ldots,x_T,y_T,\xi):= {\phi}_1(x_1,y_1,\xi_1) + \sum_{t=2}^T {\phi}_{t}(x_{t-1},x_t,y_t,\xi_t).
\end{equation*}

Similar to the above setting for the unconstrained problem, we make the following assumption.
\begin{assumption}\label{assu:convex_saddle}
For each $t=1,2,\ldots,T$, 
    \begin{enumerate}
        \item $\phi_t(\cdot,\cdot,\cdot,\xi_t)$ is differentiable, convex in $(x_{t-1},x_t)$ and concave in $y_t$ for all $\xi_t\in \Xi_t$; 
        \item $\mathcal X_t\subset \R^{n_t}$ and $\mathcal Y_t\subset \R^{m_t}$ are nonempty compact convex subsets;
        \item $\phi_t$, $\frac{\partial }{\partial x_{t-1}}\phi_t$, $\frac{\partial }{\partial x_{t}}\phi_t$, and $\frac{\partial }{\partial y_{t}}\phi_t$ are measurable w.r.t. the Borel $\sigma$-algebras on $\R^{n_{t-1}}\times \R^{n_t}\times\R^{m_t}\times \Xi_t$, and $\R, \R^{n_{t-1}},\R^{n_t},\R^{m_t}$, respectively.
    \end{enumerate}
We assume that there is a fixed reference point $(x^{\rm ref},y^{\rm ref})\in\prod_{t=1}^T(\mathcal X_t\times\mathcal Y_t)$ such that \[\E[|\phi_t(x_{t-1}^{\rm ref},x_t^{\rm ref},y_t^{\rm ref},\bx_t)|]<\infty\] for every $t$, with the first argument omitted when $t=1$.
In addition, $X^*$ is a solution to \eqref{eq:minmax-online}, with a corresponding $Y^*$ where $Y_t^*\in \mathcal P_t(\mathcal Y_t)$, such that \[
\begin{aligned}
&\tilde{\phi}_{t}(X^*_{t-1}(\xi_{1:(t-1)}), X^*_t(\xi_{1:t}),\xi_t)\\
&\qquad= \phi_{t}( X^*_{t-1}(\xi_{1:(t-1)}), X^*_t(\xi_{1:t}), Y^*_t(\xi_{1:t}),\xi_t)
\end{aligned}
\] for all $\xi\in \Xi$.
\end{assumption}

Although in \eqref{eq:minmax-online}, the decision variables are $X_{1:T}$, due to the saddle point format of $\widetilde{\phi}_t$, we aim at solving for the pairs $Z_t = (X_t,Y_t)\in \mathcal P_t(\mathcal X_t\times \mathcal Y_t)$, and we measure the quality of $Z_{1:T}$ using the duality gap. More precisely, we say $\mb z$ is $\epsilon$-suboptimal if 
\begin{equation*}
    \gaps(\mb z):=\sup_{Y'_t\in \mathcal P_t(\mathcal Y_t),t=1,\ldots,T}\sup_{X'_t\in \mathcal P_t( \mathcal X_t) ,t=1,\ldots,T}\E[\phi(\mb x,\mb y',\bx)-\phi(\mb x',\mb y,\bx)]\leq\epsilon. 
\end{equation*}
Above, $\mb z_t = (\mb x_t,\mb y_t)$, and $(\mb x'_t,\mb y'_t)=(X'_t,Y'_t)(\bx_{1:t})$. For a policy $Z$, we denote $\gaps(Z)=\gaps(\mb z)$, where $\mb z_t = Z_t(\bx_{1:t})$.

The saddle point problem can be seen as a penalty version of the classical MSSP problems. We provide comparison between the two, with a detailed discussion on the relation between the duality gap and the objective value suboptimal and constraint violation in Appendix \ref{sec:lagrangian}. Moreover, we compare the performance metrics for $X_{1:T}$ and for $X_1$ in Appendix \ref{sec:compare-stage-one}.

\subsection{Sequential output}\label{sec:seq_output}
For both \eqref{eq:obj_unconstrained-online} and \eqref{eq:minmax-online}, the decision variables are infinite-dimensional policies $\mathcal P_t(\mathcal X_t)$ and $\mathcal P_t(\mathcal Z_t)$, respectively. Thus, it is impossible even to write down a solution. In the special case where $\bx$ can take only a finite number of values ($|\Xi_t|<\infty$ for all $t$), these problems can be reformulated as problems of dimension $\sum_{t=1}^T \prod_{s=1}^t|\Xi_s|\times n_t$. However, this reformulation often suffers from the curse of dimensionality, as the effective dimension becomes exponential in $T$ due to the multi-stage setting. 

To reduce the dimension of the output, prior works have taken two approaches. The first approach, represented by the Dynamic Stochastic Approximation algorithm, aims at solving only for the first stage decision $X_1$ (see the discussion in Section \ref{sec:compare-stage-one}). The second approach, as represented by the SDDP algorithm, assumes that $(\bx_1,\bx_2,\ldots,\bx_T)$ are independent. This allows for tractable stage-wise decomposition through dynamic programming, resulting in optimization problems in dimension $n$. 

In this work, we aim at solving for decisions for \textit{all stages} with potentially \textit{dependent} randomness. This is possible due to the following two \textit{key observations}. First, even in the most general setting where there are infinite realizations for $\bx$ and the randomness between stages is correlated, at stage $t$, only $X_t(\bx_{1:t}(w))\in \R^{n_t}$ is needed, where $w\in \Omega$ is the true underlying event. That is, only a $n_t$ dimensional vector is the \textit{relevant} component of the infinite-dimensional decision variable $X_t$; for $T$ stages, the total output relevant is $(X_1(\bx_{1}(w)),\ldots,X_T(\bx_{1:T}(w)))$, a vector of dimension $n$. However, a prior (before stage $1$) the decision maker does not know the realization for $\bx(w)$, making it unclear which components to solve for. Nevertheless, as the second observation, we point out that since $\bx_{1:t}(w)$ is known by stage $t$, the decision maker knows which component of $X_t$ to solve for \textit{at stage $t$}. Thus, if what is needed from $X_t$ is only the relevant component $X_t(\bx_{1:t}(w))$, and it is needed only at stage $t$, then the problem becomes truly finite dimensional. We illustrate this idea using an example where $|\Omega|<\infty$ in Figure \ref{fig:tree}.

\begin{figure}[htbp]
\centering
\resizebox{.68\linewidth}{!}{
    \begin{tikzpicture}[darkstyle/.style={circle,draw,fill=gray!40,minimum size=23},
lightstyle/.style={draw,minimum size=20},
reds/.style={draw,fill=red!20,minimum size=23,font=\small},
greens/.style={draw,fill=green!20,minimum size=23,font=\small},
blues/.style={draw,fill=blue!20,minimum size=23,font=\small}]
\draw[blue, very thick,pattern=north west lines, pattern color=blue!30] (0.8,2.4) rectangle ++ (7.2,1.3);
\draw[green, very thick,pattern=north west lines, pattern color=green!30] (0.8,0.9) rectangle ++ (2.3,1.3);
\draw[green, very thick] (3.25,0.9) rectangle ++ (2.3,1.3);
\draw[green, very thick] (5.7,0.9) rectangle ++ (2.3,1.3);
\draw[red, very thick] (0.8+1.22*0,-0.6) rectangle ++ (1.1,1.3);
\draw[red, very thick,pattern=north west lines, pattern color=red!30] (0.8+1.22*1,-0.6) rectangle ++ (1.1,1.3);
\foreach \xx in {2,...,5} 
    \draw[red, very thick] (0.8+1.22*\xx,-0.6) rectangle ++ (1.1,1.3);
\foreach \y in {0,...,2} 
  \foreach \x in {0,...,5}
  {\pgfmathtruncatemacro{\label}{\x+1}
    \node [darkstyle]  (\x\y) at (1.2*\x+1.4,1.5*\y) {$\label$};}
\foreach \yy in {0,...,2} 
    {\pgfmathtruncatemacro{\label}{3-\yy}
    \node [lightstyle]  (\yy) at (0,1.5*\yy) {$t = \label$};} 
\node [draw,preaction={pattern=north east lines, pattern color=blue!30},minimum size=23,font=\small] (layer1-1) at (12.5,3) {$(1)$};
\node [draw,preaction={pattern=north east lines, pattern color=green!30},minimum size=23,font=\small] (layer2-1) at (10,1.6) {$(1,1)$};
\node [greens] (layer2-2) at (12.5,1.6) {$(1,2)$};
\node [greens] (layer2-3) at (15,1.6) {$(1,3)$};
\node [reds] (layer3-1) at (9+1.4*0,0) {$(1,1,1)$};
\node [draw,preaction={pattern=north east lines, pattern color=red!30},minimum size=23,font=\small] (layer3-2) at (9+1.4*1,0) {$(1,1,2)$};
\foreach \xxx in {2,...,3} 
    {\pgfmathtruncatemacro{\label}{\xxx+1}
    \node [reds] (layer3-\label) at (9+1.4*\xxx,0) {$(1,2,\label)$};}
\foreach \xxx in {4,...,5} 
    {\pgfmathtruncatemacro{\label}{\xxx+1}
    \node [reds] (layer3-\label) at (9+1.4*\xxx,0) {$(1,3,\label)$};}
\draw (layer1-1)--(layer2-1);
\draw (layer1-1)--(layer2-2);
\draw (layer1-1)--(layer2-3);
\draw (layer2-1)--(layer3-1);
\draw (layer2-1)--(layer3-2);
\draw (layer2-2)--(layer3-3);
\draw (layer2-2)--(layer3-4);
\draw (layer2-3)--(layer3-5);
\draw (layer2-3)--(layer3-6);
\end{tikzpicture}}
\caption{Example probability space with $\Omega = \{1,\ldots,6\}$, $\bx_1(w) = 1$, $\bx_2(w) = \lceil w/2\rceil$, $\bx_3(w) = w$. Left: probability space $\Omega$ with $\mathcal F_{1:3}$; each node at layer $t$ represent $\mb x_t(w)$; $\mb x_t$ is measurable w.r.t. $\mathcal F_t$ iff $\mb x_t(w) = X_t(\bx_{1:t}(w))$ is a constant in each box. Right: policy space $\mathcal P$; each node at layer $t$ represents one realization of $\bx_{1:t}=\xi_{1:t}$, and corresponds to $X_t(\xi_{1:t})$. Suppose $w = 2$, only the hatched variables $X_1(1),X_2(1,1),X_3(1,1,2)$ are needed}\label{fig:tree}
\end{figure}

\subsection{Online framework with randomized policies}\label{sec:product_space}
With the two observations in Section \ref{sec:seq_output}, we propose the following \textit{online} framework for problems \eqref{eq:obj_unconstrained-online} and \eqref{eq:minmax-online}. To model randomized policies, we equip the decision maker with \textit{internal} randomness represented by $(\Omega_{\internal},\mathcal F_{\internal},\p_{\internal})$ (The internal randomness serves as the \textit{random seed} for the decision maker. As an example, it could be $([0,1],\mathcal B([0,1]))$ with the uniform measure, which could be used for subsequent sampling), independent of the \textit{external} randomness in $(\Omega,\mathcal F,\p)$. That is, we consider the product space $(\overline{\Omega},\overline{\mathcal F},\overline{\p})=(\Omega,\mathcal F,\p)\otimes (\Omega_{\internal},\mathcal F_{\internal},\p_{\internal})$ where $\overline{w} = (w,w_{\internal})$. With abuse of notation, we use $\bx$ to represent both the random vector in $\Omega$ and in $\overline{\Omega}$, where $\bx(w,w_{\internal}) := \bx(w)$. The decision maker's policies now depend both on $\bx$ and on a set of random variables defined on $\Omega_{\internal}$, denoted as $\{\bz_{k},k\in \mathcal K\}$ where $\bz_k:\Omega_{\internal}\to \Gamma_k$ is a random variable taking values in a generic measurable space $(\Gamma_t,\mathcal G_t)$, independent of the randomness in $(\Omega,\mathcal F,\p)$, and we extend the definition to $\overline{\Omega}$ such that $\bz_t(w,w_{\internal}) = \bz_t(w_{\internal})$. Here, $\mathcal K$ is an index set. 

Equipped with this internal randomness, the decision maker's policy now becomes $\overline{X}_t: \Xi_{1:t} \times \Gamma_{S(t)}\to \R^{n_t}$ for \eqref{eq:obj_unconstrained-online} (and similarly for $\overline{Z}_t$ for \eqref{eq:minmax-online}), where $S(t)\subset \mathcal K$. That is, $\overline{X}_t$ depends both on $\bx_{1:t}$ and $\{\bz_k,~k\in S(t)\}$. In addition, the expectation on $(\Omega_{\internal},\mathcal F_{\internal})$ is denoted as $\E_{\internal}$; on the product space, we denote $\overline{\E}$ as the expectation w.r.t. both $w$ and $w_{\internal}$. 

\textbf{Sequential input, sequential output.} Before stage $1$, the decision maker observes $\{\bz_{k}(w_{\internal}) = \zeta_k,~k\in \mathcal K\}$. At stage $t=1,2,\ldots,T$, the decision maker observes the external stochastic process $\bx_t(w)=\xi_t$, then decides $\overline{X}_t(\xi_{1:t},\zeta_{S(t)}) = x_t\in \R^{n_t}$. 

\textbf{Evaluation metrics.} For \eqref{eq:obj_unconstrained-online}, due to the internal and external randomness, we say $\overline{X}$ is an $\epsilon$-suboptimal policy if the randomized policies $\overline{X}_t(\cdot,\bz_{S(t)}):\Omega_{\internal}\to \mathcal P_t(\mathcal X_t)$, defined as $\overline{X}_t(\cdot,\bz_{S(t)})(w_{\internal})(\xi_{1:t}) = \overline{X}_t(\xi_{1:t},\bz_{S(t)}(w_{\internal}))$ for all $w_\internal\in \Omega_{\internal}$ and $\xi_{1:t}\in \Xi_{1:t}$ satisfies the following: 
\begin{align*}
    \E_{\internal}[\gapu(\overline{X})]:&=\E_{\internal}[\gapu(\overline{X}(\cdot,\bz))]\\
    &=\E_{\internal}[\gapu(\overline{X}_1(\cdot,\bz_{S(1)}),\ldots,\overline{X}_T(\cdot,\bz_{S(T)}))] \leq \epsilon. 
\end{align*}
For convenience, we use the abbreviation $\overline{X}_t^{(w_{\internal})}(\cdot) = \overline{X}_t(\cdot,\bz_{S(t)})(w_{\internal})\in \mathcal P_t(\mathcal X_t)$ for all $w_{\internal}\in \Omega_{\internal}$. Similar metrics can be defined for the saddle point problem \eqref{eq:minmax-online} using $\gaps$. 

\subsection{Outline}\label{sec:outline}
On a high level, our MDSA algorithms use the internal random variables $\bz=\{\bz_k,k\in \mathcal K\}$ to construct stochastic approximations to the gradients, which are then used in the mirror descent updates for the policies. Below is an outline for the rest of the work.
\begin{itemize}
    \item In Section \ref{sec:pathwise}, we fix a realization of $\bz = \zeta$ and propose ``hypothetical'' mirror descent updates, which output $\cup_{t=1}^T \{\overline{X}_t(\xi_{1:t},\zeta_{S(t)}),~ \xi_{1:t}\in \Xi_{1:t}\}$ for \eqref{eq:obj_unconstrained-online} (similarly, $\overline{Z}_t$ for \eqref{eq:minmax-online}). 
    \item In Section \ref{sec:SA}, we consider the problems in the product space $\Omega\times\Omega_{\internal}$:
    inspired by the convergence properties of $\overline{X}_t(\xi_{1:t},\zeta_{S(t)})$ (and $\overline{Z}_t$) found in Section \ref{sec:pathwise}, we propose using the internal randomness $\bz$ to construct stochastic conditional gradients, which leads to the multi-stage MDSA algorithms.
    \item In Section \ref{sec:online-implementation}, we show that these MDSA algorithms induce policies which can be implemented efficiently in the ``sequential input, sequential output'' model. 
    \item The appendices contain further discussion of the setup (Appendix~\ref{app:setup}), accelerated MDSA (Appendix~\ref{appendix:acc}), proofs of the main-text results (Appendix~\ref{app:main-proofs}), and numerical experiments (Appendix~\ref{sec:num_exp}). 
\end{itemize}

\section{Mirror descent: a deterministic perspective}\label{sec:pathwise}
We first give a brief review of mirror descent \citep{Bubeck2015,LanBook,Nemirovski2004Prox}. We equip $\R^{n_0}$ with the Euclidean inner product $\langle \cdot,\cdot \rangle$ and a norm $\|\cdot\|$ not necessarily induced by the inner product. Recall that for a convex compact subset $\mathcal Q\subset \R^{n_0}$, $v:\mathcal Q\to \R$ is a distance generating function \citep{LanBook,NemirovskiRobustSA2009} for $\mathcal Q$ if 
$v$ is convex and continuous on $\mathcal Q$, and 
\begin{equation*}
    \mathcal Q^o:=\{x\in \mathcal Q|\exists g\in \R^{n_0},~x\in \argmin_{x'\in \mathcal Q} v(x') +\langle g, x'\rangle \}
\end{equation*}
is a convex set, and restricted to $\mathcal Q^o$, $v$ is continuously differentiable and $1$-strongly convex, i.e.:
\begin{equation*}
    \langle x'-x,\nabla v(x')-\nabla v(x)\rangle\geq \|x'-x\|^2.
\end{equation*}
Then, initialized at some $x^{(0)}\in \mathcal Q^o$, mirror descent updates the variables iteratively for $l = 0,1,\ldots$
\begin{equation}\label{eq:md_general}
    x^{(l+1)} = \argmin_{x'\in \mathcal Q} \gamma_l\langle g^{(l)},x'\rangle + D_{v}(x',x^{(l)}),
\end{equation}
where $D_{v}(y,x) := v(y) - v(x) - \langle \nabla v(x),y-x\rangle$ is the Bregman divergence induced by $v$ and $g^{(l)}\in \R^n$ is a subgradient of the objective function, and $\gamma_l\geq 0$ is usually chosen based on properties of the objective function $f$ and the set $\mathcal X$. Below, we state a well known property of the mirror descent update \eqref{eq:md_general}.
\begin{lemma}[Lemma 2.1 \citep{NemirovskiRobustSA2009}]\label{lm:MD-prop-lemma}
    For any $x\in \mathcal Q^o$ and $g\in \R^{n_0}$, consider
    \begin{equation*}
            x^{+} = \argmin_{x'\in \mathcal Q} \langle g,x'\rangle + D_{v}(x',x).
    \end{equation*}
    Then the following holds for any $x'\in \mathcal Q$,
    \begin{equation*}
        D_v(x',x^{+})\leq D_v(x',x) + \langle  g, x'-x\rangle +\frac{1}{2}\|g\|_*^2.
    \end{equation*}
\end{lemma}

In fact, we can use arbitrary $g^{(l)}$ in the update \eqref{eq:md_general} (since Lemma \ref{lm:MD-prop-lemma} holds for all $g$). However, the suboptimality of the solution found depends on how close $g^{(l)}$ is to the gradient of the objective.

We propose a natural extension of the above mirror descent step to the infinite-dimensional space, where \eqref{eq:md_general} is applied to $\overline{X}_t(\xi_{1:t},\zeta_{S(t)})$ for all $\xi\in \Xi$ and $t$, with a generic $\overline{G}_t^{(l)}(\cdot,\zeta_{S'(t,l)})\in \mathcal P_t(\R^{n_t})$ in place of $g^{(l)}$. In this section, we fix a realization for all  internal random variables $\{\bz_k = \zeta_k,~k\in \mathcal K\}$ and look at the suboptimality of the policies $\overline{X}_t(\cdot,\zeta_{S(t)})$ (and $\overline{Z}_t$) found by these (infinite-dimensional) hypothetical algorithms. The convergence performance of these algorithms provides insight on desired properties of $\overline{G}_t^{(l)}$ -- it should approximate the conditional gradient -- which we explore in Section \ref{sec:SA}.

\subsection{Mirror descent for \eqref{eq:obj_unconstrained-online}}\label{sec:md_unconstraint}

We equip $\R^{n}$ with the Euclidean inner product $\langle \cdot,\cdot \rangle$. For all $t$, $\R^{n_t}$ is equipped with the norm $\|\cdot\|$ not necessarily induced by the Euclidean inner product, and $v_t:\mathcal X_t\to \R$ is a distance generating function which is $1$-strongly convex w.r.t. the norm $\|\cdot \|$ such that $\mathcal X_t$ admits easy projection using $D_{v_t}$. We equip $\R^n$ with the norm $ \|x_{1:T}\|^2 = \sum_{t=1}^T \|x_t\|^2$ and $\mathcal X=\prod_{t=1}^T \mathcal X_t$ with the distance generating function $v(x_{1:T}):=\sum_{t=1}^T v_t(x_t)$. 

We further make the following assumption.
\begin{assumption}\label{assumption:measurable}
     $\nabla v_t(X_t):\Xi_{1:t}\to \R^{n_t}$ is measurable for any $X_t\in \mathcal P_t(\mathcal X_t^{(o)})$ (and so $\nabla v_t(X_t)\in \mathcal P_t(\R^{n_t})$), and $\overline{G}_t^{(l)}(\cdot,\zeta_{S'(t,l)}):\Xi_{1:t}\to \R^{n_t}$ is measurable (and so $\overline{G}_t^{(l)}(\cdot,\zeta_{S'(t,l)})\in  \mathcal P_t(\R^{n_t})$) for all possible realizations $\{\bz_k = \zeta_{k},~k\in \mathcal K\}$.
\end{assumption}

For the standard mirror descent algorithms, $L\geq 1$ denotes the number of updates, indexed by $l=0,\ldots,L-1$; the output averages the $L$ pre-update iterates using the corresponding step sizes. Consider the following hypothetical iterative updates for all $\xi\in \Xi$, given $\{\zeta_{k},~k\in \mathcal K\}$.

\begin{algorithm}
\caption{Hypothetical mirror descent update for \eqref{eq:obj_unconstrained-online}}\label{alg:hypo-ms-u}
\begin{algorithmic}
\Require Realizations of the random variables $\{\bz_k(w_{\internal})= \zeta_k,~k\in \mathcal K\}$; initialization $X^{(0)}_t\in \mathcal P_t(\mathcal X^o_t)$ for all $t$; step sizes $\gamma_0,\ldots,\gamma_{L-1}>0$; oracles $\overline{G}^{(l)}_t:\Xi_{1:t}\times \Gamma_{S'(t,l)}\to \R^{n_t}$ for all $t$ and $l=0,\ldots,L-1$. 
\Ensure $\overline{X}_t(\cdot,\zeta_{S(t)})\in \mathcal P_t(\mathcal X_t)$ for all $t$.
\For{$l=0,1,\ldots,L-1$}
\For{$\xi\in \Xi$, $t=1,\ldots,T$}
\State Set $G_t^{(l)}(\xi_{1:t}) = \overline{G}_t^{(l)}(\xi_{1:t},\zeta_{S'(t,l)})$.
\State Compute $\overline{X}_t^{(l+1)}(\xi_{1:t},\zeta_{S'(t,l)}) = X_{t}^{(l+1)}(\xi_{1:t})$ using
\begin{equation}\label{eq:MD1}
    X_{t}^{(l+1)}(\xi_{1:t}) = \argmin_{x_t\in \mathcal X_t} \langle \gamma_l G^{(l)}_t(\xi_{1:t}),x_t \rangle+ D_{v_t}(x_t, X^{(l)}_t(\xi_{1:t})).
\end{equation}
\EndFor
\EndFor
\State Compute 
\begin{equation}\label{eq:ergodic-u}
    \overline{X}_t(\cdot,\zeta_{S(t)})=\frac{\sum_{l=0}^{L-1} \gamma_lX_t^{(l)}}{\sum_{l=0}^{L-1} \gamma_l},\quad t=1,\ldots,T. 
\end{equation}
\end{algorithmic}
\end{algorithm}
Recall that $\mb x_t^{(l)}:\Omega\to \R^{n_t}$ is defined as $\mb x_t^{(l)}(w) =X_t^{(l)}(\bx_{1:t}(w))$, $\mb g_t^{(l)}:\Omega\to \R^{n_t}$ is defined as $\mb g_t^{(l)}(w) =G_t^{(l)}(\bx_{1:t}(w))$.
\begin{lemma}\label{lm:MD1}
    Under Assumptions \ref{assu:convex} and \ref{assumption:measurable}, consider the update in Algorithm \ref{alg:hypo-ms-u}. Then, $\overline{X}_t(\cdot,\zeta_{S(t)}),X_t^{(l)}\in \mathcal P_t(\mathcal X_t^{o})$ for all $t,l$. 
    
    In addition, for $\widetilde{\mb x}^{(L)}:=\frac{\sum_{l=0}^{L-1} \gamma_l\mb x^{(l)}}{\sum_{l=0}^{L-1} \gamma_l}$, we have
   \begin{align*}
        \gapu(\widetilde{\mb x}^{(L)}) \leq \frac{\E[D_{v}(\mb x^*,\mb x^{(0)})] + \frac{1}{2}\sum_{l=0}^{L-1} \gamma_l^2  \E[\|\mb g^{(l)}\|_*^2]-\sum_{l=0}^{L-1}\gamma_l \E[\langle \Delta^{(l)},\mb x^{(l)} - \mb x^* \rangle]}{\sum_{l=0}^{L-1} \gamma_l},
   \end{align*}
    where $\Delta^{(l)} = \Delta_{1:T}^{(l)}$ with $\Delta^{(l)}_t=\mb g^{(l)}_t-\E[\frac{\partial}{\partial x_t}f(\mb x^{(l)},\bx)|\mathcal F_t]$, and $\mb g^{(l)} = \mb g_{1:T}^{(l)}$. 
\end{lemma}
The proof is provided in Appendix \ref{app:proof:lm:MD1}.

\subsection{Mirror descent for \eqref{eq:minmax-online}}\label{sec:md_saddle}

For the saddle problem \eqref{eq:minmax-online}, we denote $\mathcal Z_t = \mathcal X_t\times \mathcal Y_t$, and 
\begin{equation}\label{eq:def-phi-grad}
    \frac{\tilde{\partial} \phi}{\partial z_t}(z,\xi) = \begin{bmatrix}
    \frac{\partial}{\partial x_t}\phi(z,\xi)\\-\frac{\partial}{\partial y_t}\phi(z,\xi)
\end{bmatrix},\quad \frac{\tilde{\partial} \phi}{\partial z}(z,\xi) = \begin{bmatrix}
    \frac{\partial}{\partial x}\phi(z,\xi)\\-\frac{\partial}{\partial y}\phi(z,\xi)
\end{bmatrix}.
\end{equation}

We equip $\R^{n+m}$ with the Euclidean inner product $\langle \cdot,\cdot \rangle$. For all $t$, $\R^{n_t}$ ($\R^{m_t}$) are equipped with the norm $\|\cdot\|$ not necessarily induced by the Euclidean inner product, and $v_t:\mathcal X_t\to \R$ ($u_t:\mathcal Y_t\to \R$) is a distance generating function which is $1$-strongly convex w.r.t. the norm $\|\cdot \|$ such that $\mathcal X_t$ ($\mathcal Y_t$) admits easy projecting using $D_{v_t}$ ($D_{u_t}$). We equip $\R^{n+m}$ with the norm $\|z_{1:T}\|^2 = \sum_{t=1}^T (\|x_t\|^2 + \|y_t\|^2)$ and $\mathcal Z=\prod_{t=1}^T \mathcal Z_t$ with the distance generating function $w(z_{1:T}):=\sum_{t=1}^T (v_t(x_t) + u_t(y_t))$. 

Similarly, we make the following assumption.
\begin{assumption}\label{assumption:measurable-saddle}
     $\nabla v_t(X_t):\Xi_{1:t}\to \R^{n_t}$ is measurable for any $X_t\in \mathcal P_t(\mathcal X_t^{(o)})$ (and so $\nabla v_t(X_t)\in \mathcal P_t(\R^{n_t})$), and similarly for $u_t$. $\overline{G}_t^{(l)}(\cdot,\zeta_{S'(t,l)}):\Xi_{1:t}\to \R^{n_t+m_t}$ is measurable (and so $\overline{G}_t^{(l)}(\cdot,\zeta_{S'(t,l)})\in  \mathcal P_t(\R^{n_t+m_t})$) for all possible realizations $\{\bz_k = \zeta_{k},~k\in \mathcal K\}$.
\end{assumption}

\begin{algorithm}
\caption{Hypothetical mirror descent update for \eqref{eq:minmax-online}}\label{alg:hypo-ms-s}
\begin{algorithmic}
\Require Realizations of the random variables $\{\bz_k(w_{\internal})= \zeta_k,~k\in \mathcal K\}$; initialization $Z^{(0)}_t\in \mathcal P_t(\mathcal X^o_t\times \mathcal Y_t^o)$ for all $t$; step sizes $\gamma_0,\ldots,\gamma_{L-1}>0$; oracles $\overline{G}^{(l)}_t:\Xi_{1:t}\times \Gamma_{S'(t,l)}\to \R^{n_t+m_t}$ for all $t$ and $l=0,\ldots,L-1$. 
\Ensure $\overline{Z}_t(\cdot,\zeta_{S(t)}) \in \mathcal P_t(\mathcal Z_t)$ for all $t$.
\For{$l=0,1,\ldots,L-1$}
\For{$\xi\in \Xi$, $t=1,\ldots,T$}
\State Set $G_t^{(l)}(\xi_{1:t}) = \overline{G}_t^{(l)}(\xi_{1:t},\zeta_{S'(t,l)})$.
\State Compute $\overline{Z}_t^{(l+1)}(\xi_{1:t},\zeta_{S'(t,l)}) = Z_{t}^{(l+1)}(\xi_{1:t})$ using
\begin{equation}\label{eq:MD2}
    Z_{t}^{(l+1)}(\xi_{1:t}) = \argmin_{z_t\in \mathcal Z_t} \langle \gamma_l G^{(l)}_t(\xi_{1:t}),z_t \rangle+ D_{w_t}(z_t, Z^{(l)}_t(\xi_{1:t})).
\end{equation}
\EndFor
\EndFor
\State 
Compute 
\begin{equation*}
   \overline{Z}_t(\cdot,\zeta_{S(t)})=\frac{\sum_{l=0}^{L-1} \gamma_lZ_t^{(l)}}{\sum_{l=0}^{L-1} \gamma_l},\quad t=1,\ldots,T. 
\end{equation*}
\end{algorithmic}
\end{algorithm}

\begin{lemma}\label{lm:MD-saddle}
    Under Assumptions \ref{assu:convex_saddle} and \ref{assumption:measurable-saddle}, consider the update in Algorithm \ref{alg:hypo-ms-s}. Then, $\overline{Z}_t,Z_t^{(l)}\in \mathcal P_t(\mathcal X_t^{o}\times \mathcal Y_t^o)$ for all $t,l$. In addition, for $\widetilde{\mb z}^{(L)}:=\frac{\sum_{l=0}^{L-1} \gamma_l\mb z^{(l)}}{\sum_{l=0}^{L-1} \gamma_l}$, we have for any $\mb z_t:\Omega\to \mathcal Z_t$ measurable w.r.t. $\mathcal F_t$,
   \begin{align*}
    &\quad \E[\phi(\widetilde{\mb x}^{(L)},\mb y,\bx) -\phi(\mb x,\widetilde{\mb y}^{(L)},\bx) ]\\
    &\leq \frac{2\E[D_{w}(\mb z,\mb z^{(0)})]+
   \sum_{l=0}^{L-1} \left[\frac{\gamma_l^2}{2}(\E[\|\mb g^{(l)}\|_*^2 +\|\Delta^{(l)}\|_*^2] )-\gamma_l\E[\langle \mb z^{(l)} - \widehat{\mb z}^{(l)},\Delta^{(l)}\rangle]\right]}{\sum_{l=0}^{L-1} \gamma_l},
\end{align*}
   where $\Delta^{(l)}_t = \mb g^{(l)}_t-\E[ \frac{\tilde{\partial}}{\partial z_t}\phi(\mb z^{(l)},\bx)|\mathcal F_t]$. The sequence $\widehat{\mb z}_t^{(0)} =\mb z_t^{(0)} $, for $l=0,1,\ldots$, $w\in \Omega$, $t=1,2,\ldots,T$
    \begin{equation*}
            \widehat{\mb z}_t^{(l+1)}(w) = \argmin_{z_t'\in {\mathcal Z}_t} \langle -\gamma_l\Delta_t^{(l)}(w),z_t'\rangle + D_{w_t}(z_t',\widehat{\mb z}_t^{(l)}(w)). 
        \end{equation*}

\end{lemma}

The proof is provided in Appendix \ref{app:proof:lm:MD-saddle}.
\section{Mirror descent stochastic approximation}\label{sec:SA}

To (conceptually) implement the mirror descent updates in Section \ref{sec:pathwise}, one still needs to decide what $\overline{G}_t^{(l)}$ is. Two factors should be taken into consideration. 

\textbf{Convergence guarantee.} First, Lemmas \ref{lm:MD1} and \ref{lm:MD-saddle} suggest that the suboptimalities (in terms of the function values or duality gaps) of the mirror descent updates depend on $\Delta_t^{(l)}$. For \eqref{eq:obj_unconstrained-online} and $2\leq t\leq T-1$, recall that $\mathcal F_t = \sigma(\bx_{1:t})$ and
\begin{equation}\label{eq:grad-compute}
    \frac{\partial }{\partial x_t}f(x_{1:T},\xi_{1:T}) = \frac{\partial }{\partial x_t}f_{t}(x_{t-1},x_t,\xi_t) + \frac{\partial }{\partial x_t}f_{t+1}(x_{t},x_{t+1},\xi_{t+1}),
\end{equation}
and so we have
\begin{align*}
    \Delta^{(l)}_t&= \mb g^{(l)}_t-\E[\frac{\partial}{\partial x_t}f(\mb x^{(l)},\bx)|\bx_{1:t}] \\
    &= G_t^{(l)}(\bx_{1:t}) - \frac{\partial}{\partial x_t}f_t(X_{t-1}^{(l)}(\bx_{1:(t-1)}),X_{t}^{(l)}(\bx_{1:t}),\bx_t)\\
    &\quad -\E[\frac{\partial}{\partial x_t}f_{t+1}(X_{t}^{(l)}(\bx_{1:t}),X_{t+1}^{(l)}(\bx_{1:(t+1)}),\bx_{t+1})|\bx_{1:t}].
\end{align*}
Thus, one might hope to design $\overline{G}_t^{(l)}$ in such a way that $\Delta_t^{(l)}\approx \mb 0$, or, $\mb g_t^{(l)}\approx\E[\frac{\partial}{\partial x_t}f(\mb x^{(l)},\bx)|\bx_{1:t}]$ approximates the \textit{conditional gradient}. 

\textbf{Inefficiency in computing (conditional) expectation.} However, the term $\E[\frac{\partial}{\partial x_t}f_{t+1}|\bx_{1:t}]$ in this conditional gradient is of the \textit{conditional expectation} form. Thus, computing $\E[\frac{\partial}{\partial x_t}f|\bx_{1:t}]$ requires integration (using the distribution of $\bx_{t+1}$ conditioned $\bx_{1:t}$), which could be inefficient \citep{NemirovskiRobustSA2009}. 

\textbf{Our solution: stochastic approximation.} Taking inspiration from the stochastic approximation type of algorithms for the classical stochastic programming problems, we resort to randomness to approximate these conditional expectations. On a high level, $\overline{G}_t^{(l)}$ will use the random variables $\bz_{S'(t,l)}$ to sample $\widehat{\bx}_{t+1}$ under (conditional) distribution of $\bx_{t+1}$ conditioned on $\bx_{1:t} = \xi_{1:t}$, and use $\frac{\partial}{\partial x_t}f_{t+1}(X_{t}^{(l)}(\xi_{1:t}),\overline{X}_{t+1}^{(l)}(\xi_{1:t},\widehat{\bx}_{t+1}),\widehat{\bx}_{t+1})$ as a \textit{stochastic approximation} of $\E[\frac{\partial}{\partial x_t}f_{t+1}|\bx_{1:t} = \xi_{1:t}]$. Since this construction resembles the classical stochastic gradient, we call it \textit{stochastic conditional gradient}.

Below, in Section \ref{sec:def-grad-cvx}, we formally introduce this product space setup and the stochastic conditional gradients. Then, in Section \ref{sec:MDSA}, we provide the performance of the mirror descent updates above, with these stochastic conditional gradients.

\subsection{Stochastic conditional gradient}\label{sec:def-grad-cvx}
First, recall that given two probability measures $P,Q$ defined on some measurable space $(\Omega',\mathcal F')$, the total variation distance between them is $TV(P,Q):=\sup_{A\in \mathcal  F'} |P(A)-Q(A)|$. We define the following conditional scenario sampler.

\textbf{Conditional-distribution convention.} We assume that each $\Xi_t$ is a standard Borel space and fix a measurable probability kernel $K_t(\xi_{1:(t-1)},\cdot)$ representing a regular conditional distribution of $\bx_t$ given $\bx_{1:(t-1)}$. The kernel is specified on all histories, including histories of probability zero. Every conditional-distribution expression below refers to this chosen version, and the sampler guarantee in Definition \ref{def:conditional-scenario-sampler} is required uniformly over all histories, including those queried during recursive simulation.

\begin{definition}\label{def:conditional-scenario-sampler}
    Let $\widehat{\bx}_t:\Xi_{1:(t-1)}\times [0,1]\to \Xi_t$ be a measurable function. We say that $\widehat{\bx}_t$ is a $\beta_t$-biased \textit{conditional scenario sampler} if for any $\xi_{1:(t-1)}\in \Xi_{1:(t-1)}$, given a random variable uniform in $[0,1]$, i.e. $U\sim \text{Uniform}([0,1])$, denoting the distribution of $\widehat{\bx}_t(\xi_{1:(t-1)},U)$ as $\mu(\xi_{1:(t-1)})$, then $TV(\mu(\xi_{1:(t-1)}),\p(\bx_t|\bx_{1:(t-1)} = \xi_{1:(t-1)}))\leq \beta_t$. 
\end{definition}

That is, $\widehat{\bx}_t(\xi_{1:(t-1)},U)$ is a random variable whose distribution is (approximately) the distribution of $\bx_{t}$ conditioning on $\bx_{1:(t-1)} = \xi_{1:(t-1)}$. In case when all $\bx_{t'}$ have finite possible realizations (i.e., $|\Xi_{t'}|<\infty$), the policy space can be represented using a tree (e.g. the right figure in Figure \ref{fig:tree}), then $\widehat{\bx}_t(\xi_{1:(t-1)},U)$ is a random child node of the node corresponding to $\xi_{1:(t-1)}$. 

Above, we use a uniform random variable in $[0,1]$ as the \textit{random seed}. This choice is arbitrary, and one can replace it with any measurable space $(\Omega',\mathcal F')$ as long as $\widehat{\bx}_t:\Xi_{1:(t-1)}\times \Omega'\to \Xi_t$ is measurable, there exists a random variable $U':\Omega_{\internal}\to \Omega'$, and the bias will be for the distribution of $\widehat{\bx}_t(\xi_{1:(t-1)},U')$.

Next, we discuss how to construct stochastic conditional gradients using the conditional scenario sampler. From \eqref{eq:grad-compute}, we see that $\frac{\partial}{\partial x_t}f$ depends only on a subset of $x_{1:T}$ and $\xi_{1:T}$. That is, for $t=2,\ldots,T-1$, we have 
\begin{equation*} 
    \frac{\partial }{\partial x_t}f(x_{1:T},\xi_{1:T}) =\frac{\partial }{\partial x_t}f(x_{(t-1):(t+1)},\xi_{t:(t+1)})=\frac{\partial }{\partial x_t}f_{t}(x_{t-1},x_t,\xi_t) + \frac{\partial }{\partial x_t}f_{t+1}(x_{t},x_{t+1},\xi_{t+1}). 
\end{equation*}
For $t=T$, we have $\frac{\partial }{\partial x_T}f(x_{1:T},\xi_{1:T}) = \frac{\partial }{\partial x_T}f(x_{(T-1):T},\xi_{T})=\frac{\partial }{\partial x_T}f_{T}(x_{T-1},x_{T},\xi_{T})$. For $t=1$, we have $\frac{\partial }{\partial x_1}f(x_{1:T},\xi_{1:T}) =\frac{\partial }{\partial x_1}f(x_{1:2},\xi_{1:2})
    =\frac{\partial }{\partial x_1}f_{1}(x_1,\xi_1) + \frac{\partial }{\partial x_1}f_{2}(x_{1},x_{2},\xi_{2})$. 

\begin{lemma}\label{lm:prop-oracle}
For $2\leq t\leq T-1$, let $\widehat{\bx}_{t+1}$ be a $\beta_{t+1}$-biased conditional scenario sampler as defined in Definition \ref{def:conditional-scenario-sampler}. Under Assumption \ref{assu:convex}, with $X_s\in \mathcal P_s(\mathcal X_s)$, let $G_t(\cdot,\cdot;X_{(t-1):(t+1)}):\Xi_{1:t}\times [0,1]\to \R^{n_t}$ be defined as 
\begin{equation}\label{eq:choice-of-G}
    G_t(\xi_{1:t},u;X_{(t-1):(t+1)}) = \frac{\partial}{\partial x_t} f(x_{t-1},x_t,x_{t+1},\xi_t,\widehat{\bx}_{t+1}(\xi_{1:t},u)), 
\end{equation}
where $x_{t-1} = X_{t-1}(\xi_{1:(t-1)})$, $x_t = X_t(\xi_{1:t})$, $x_{t+1} = X_{t+1}(\xi_{1:t},\widehat{\bx}_{t+1}(\xi_{1:t},u))$. Then $G_t(\cdot,\cdot;X_{(t-1):(t+1)})$ is measurable (for fixed $X_{(t-1):(t+1)}$). 

Assume that $\|\frac{\partial}{\partial x_t}f\|_*\leq \lips_t$ for all $x\in \mathcal X,\xi\in \Xi$, and take $b_t = 2\beta_{t+1} \lips_t$. Let $U\sim \text{Uniform}([0,1])$, then there exists $\sigma_t\leq 2\lips_t$, such that for any $\xi\in \Xi$, 
\begin{subequations}\label{eq:grad-conditions}
\begin{align}
    & \|\E[G_t(\xi_{1:t},U)] - G^*_t(\xi_{1:t})\|_*\leq b_t,\label{eq:grad-unbias}\\
    & \E[\|G_t(\xi_{1:t},U) -\E\left[G_t(\xi_{1:t},U)\right] \|^2_*]\leq \sigma^2_t,\label{eq:grad-var} 
\end{align}
\end{subequations}
where we abbreviate $G_t(\xi_{1:t},u;X_{(t-1):(t+1)}) = G_t(\xi_{1:t},u)$, and define \[
\begin{aligned}
G^*_t(\xi_{1:t}) = \E\Big[&\frac{\partial}{\partial x_t} f\big(X_{t-1}(\xi_{1:(t-1)}),X_t(\xi_{1:t}),\\
&X_{t+1}(\xi_{1:t},\bx_{t+1}),\xi_t,\bx_{t+1}\big)\mid\bx_{1:t} = \xi_{1:t}\Big].
\end{aligned}
\] 
\end{lemma}

The proof is provided in Appendix \ref{app:proof:lm:prop-oracle}.

The above definition can be extended to $t=1,T$ and to the problem \eqref{eq:minmax-online} (for $\frac{\tilde{\partial} \phi}{\partial z_t}$ as defined in \eqref{eq:def-phi-grad}) in a similar fashion. 

\subsection{Mirror descent stochastic approximation}\label{sec:MDSA}
The construction of stochastic conditional gradients in Lemma \ref{lm:prop-oracle} can be used to define $\overline{G}_t^{(l)}$ in Algorithms \ref{alg:hypo-ms-u}, \ref{alg:hypo-ms-s}, and \ref{alg:hypo-ms-acc}.

The overall idea is to associate each $(t,l)$ with a random variable $\bz_{t,l}$ following $\text{Uniform}([0,1])$, independently, and use it as the random seed for the sampled $\widehat{\bx}_{t+1}$ in $\overline{G}_t^{(l)}$. We point out the convergence/performance of our algorithms does not require a separate random seed for each $\xi$, and thus $\bz_{t,l}$ can be shared among different $\xi_{1:t}$.

\begin{lemma}\label{lm:MDSA-1}
Let $\widehat{\bx}_t:\Xi_{1:(t-1)}\times [0,1]\to \Xi_t$ be $\beta_t$-biased conditional scenario samplers for all $t$. Assume that $\mathcal K = \{1,\ldots,T\} \times \{0,\ldots,L\}$, $\bz_{t,l}:\Omega_{\internal}\to [0,1]$ such that $\bz_{t,l}\sim_{i.i.d.} \text{Uniform}([0,1])$. For Algorithm \ref{alg:hypo-ms-u}, the following are valid choices of $S(t),S'(t,l)$ and $\overline{G}_t^{(l)}$:
\begin{enumerate}
    \item $S'(t,-1) =\emptyset$ for all $t$; $S'(t,l) = \cup_{s=\max(1,t-1)}^{\min(t+1,T)} S'(s,l-1)\cup \{(t,l)\}$ for all $t$ and $0\leq l\leq L$; moreover, $X_t^{(l)}(\xi_{1:t}) = \overline{X}^{(l)}_t(\xi_{1:t},\zeta_{S'(t,l-1)})$ depends only on $\bz_k$ for $k\in S'(t,l-1)$;
    \item $\overline{G}_t^{(l)}(\xi_{1:t},\zeta_{S'(t,l)})=G_t(\xi_{1:t},\zeta_{t,l};X_{(t-1):(t+1)}^{(l)})$ using \ref{eq:choice-of-G}, and depends only on  $\bz_k$ for $k\in S'(t,l)$; thus, $G_t^{(l)}$ depends only on $\bz_k$ for $k\in S'(t,l)$; similarly for $t=1,T$;
    \item $S(t) = S'(t,L-1)$; moreover, $S'(t,l) = \{(t',l')\in \mathcal K,~|t-t'|\leq l-l'\}$ for all $(t,l)\in \mathcal K$. 
\end{enumerate}
\end{lemma}

The proof is provided in Appendix \ref{app:proof:lm:MDSA-1}.

In fact, $S'(t,l)$ consists of indices in the \textit{triangle} with vertices $(t,l),(t\pm l,0)$ (assuming $1\leq t-l\leq t+l\leq T$). See Figure \ref{fig:S-fig} for an illustration. Next, we state the performance of our Algorithm \ref{alg:hypo-ms-u} with this choice of $\bz$ as the random seeds.

\begin{figure}[h!]
\centering
\resizebox{.85\linewidth}{!}{%
\begin{minipage}{\linewidth}
    \centering
    \begin{tikzpicture}[
lightstyle/.style={draw,minimum size=15,font = \small}]
  \def\nrows{5}
  \def\ncols{9}
  \def\spacingrow{0.5}
  \def\spacingcol{1}
  \def\gap{0.25}
  \foreach \i in {1,...,\numexpr\ncols}
  {\pgfmathtruncatemacro{\label}{\i}
    \node [lightstyle]  (\i) at (\spacingcol*\i,-\spacingrow-\gap) {$t = \label$};
    \foreach \j in {0,...,\numexpr\nrows-1}
      \fill (\i*\spacingcol,\j*\spacingrow) circle (1.5pt);}
    \foreach \j in {0,...,\numexpr\nrows-1}{
    \pgfmathtruncatemacro{\label}{\j}
    \node [lightstyle]  (\j) at (-\gap,\spacingrow*\j) {$l = \label$};
    }
    \filldraw[fill=red!60, fill opacity=0.3, draw=red, thick] (4*\spacingcol,2*\spacingrow+\gap) -- (2*\spacingcol-2*\gap,-\gap/2) -- (6*\spacingcol+2*\gap,-\gap/2) -- cycle;
    \filldraw[fill=blue!60, fill opacity=0.3, draw=blue, thick] (8*\spacingcol,4*\spacingrow+\gap) -- (4*\spacingcol-2*\gap,-\gap/2) -- (9*\spacingcol+2*\gap,-\gap/2)--(9*\spacingcol+2*\gap,3*\spacingrow+\gap/2) -- cycle;
\end{tikzpicture}
    \end{minipage}%
}
\par
\caption{$S'(t,l) = \{(t',l')\in \mathcal K,~|t-t'|\leq l-l'\}$. Red region: $S'(4,2)$; blue region: $S'(8,4)$.}
    \label{fig:S-fig}
\end{figure}

\begin{theorem}\label{thm:MDSA-1}
Under Assumption \ref{assu:convex}, \ref{assumption:measurable}, and the conditions of Lemmas \ref{lm:prop-oracle} and \ref{lm:MDSA-1}, further assuming that $\|\frac{\partial}{\partial x_t} f\|_*\leq \lips_t$ for all $t$, $\E[D_v(\mb x^*,\mb x^{(0)})]\leq \tilde{D}^2$, and $\|x-x'\|\leq D$ for all $x,x'\in \mathcal X$. Denote $b^2 = \sum_{t=1}^T b_t^2$, and $\tilde{L}^2 = \sum_{t=1}^T \lips_t^2$. 

Recall that $\gapu(\overline{X})$ is a random variable on the space $(\Omega_{\internal},\mathcal F_{\internal})$ such that $\gapu(\overline{X})(w_{\internal}) = \gapu(\overline{X}^{(w_{\internal})})$ where $\overline{X}_t^{(w_{\internal})}(\xi_{1:t}) = \overline{X}_t(\xi_{1:t},\bz_{S(t)}(w_{\internal}))$ for all $\xi\in \Xi$. With $\gamma_l = \frac{\sqrt{2}\tilde{D}/\tilde{L}}{\sqrt{L}}$,
\begin{equation}\label{eq:convergence-in-exp}
\E_{\internal}[\gapu(\overline{X})] \leq \frac{\sqrt{2}\tilde{D}\tilde{L}}{\sqrt{L}}+bD.
\end{equation}

\end{theorem}
In Lemma \ref{lm:MD1}, we provide performance guarantees for $\overline{X}_t(\cdot,\zeta_{S(t)})$ for each realization of $\bz = \zeta\in \Gamma$. Below, we use the randomness in $\bz$, together with Lemmas \ref{lm:prop-oracle} and \ref{lm:MDSA-1} to provide performance guarantees for Algorithm \ref{alg:hypo-ms-u} on the product space $\Omega\times\Omega_{\internal}$. 

The proof is provided in Appendix \ref{app:proof:thm:MDSA-1}.

\textbf{Remark.} As a special case, suppose that for all $t$, the diameter of $\mathcal X_t$, $\E[D_{v_t}(\mb x_t^*,\mb x_t^{(0)})]$, the bias $b_t$, and $\lips_t$ are upper bounded by $D_0, \tilde{D}^2_0$, $b_0$, and $L_0$, respectively. Then we have ${D} \leq {D}_0\sqrt{T}$, $\tilde{D} \leq \tilde{D}_0\sqrt{T}$, ${b} \leq b_0\sqrt{T}$, and $\tilde{L} \leq  L_0\sqrt{T}$. With the suggested $\gamma_l = \frac{\sqrt{2}\tilde{D}_0/L_0}{\sqrt{L}}$,
\begin{equation*}
    \E_{\internal}[\gapu(\overline{X})]\leq \sqrt{2}T\frac{\tilde{D}_0L_0}{\sqrt{L}}+b_0D_0T.
\end{equation*}
Thus, to get $O(T\epsilon)$ suboptimality (in expectation), one needs to ensure that $b_0D_0=O( \epsilon )$ and set $L = O(1/\epsilon^2)$, which is independent of the number of stages $T$. Moreover, when $\epsilon=\Omega(1)$, the suggested $\gamma_l$ does not depend on $T$, the total number of stages.

\begin{theorem}\label{thm:MDSA-2}

Under Assumption \ref{assu:convex_saddle} and \ref{assumption:measurable-saddle} and the conditions of Lemmas \ref{lm:prop-oracle} and \ref{lm:MDSA-1}, further assuming that $\|\frac{\tilde{\partial}}{\partial z_t} \phi\|_*\leq \lips_t$ for all $t$, $\sup_{Z_t\in\mathcal P_t(\mathcal Z_t),\,t=1,\ldots,T}\E[D_w(\mb z,\mb z^{(0)})]\leq \tilde{D}^2$, where $\mb z_t=Z_t(\bx_{1:t})$, and $\|z-z'\|\leq D$ for all $z,z'\in \mathcal Z$. Denote $b^2 = \sum_{t=1}^T b_t^2$ and $\tilde{L}^2 = \sum_{t=1}^T \lips_t^2$. 

With $\gamma_l = \frac{2\tilde{D}/\tilde{L}}{\sqrt{5L}}$,
\begin{equation}\label{eq:convergence-in-exp-saddle}
\E_{\internal}[\gaps(\overline{Z})] \leq \frac{2\sqrt{5}\tilde{D}\tilde{L}}{\sqrt{L}}+bD.
\end{equation}
\end{theorem}

The proof of Theorem \ref{thm:MDSA-2} is similar to the proof of Theorem \ref{thm:MDSA-1} and uses Lemma \ref{lm:MD-saddle}. In addition, one can show convergence with high probability following similar arguments as in \citet{NemirovskiRobustSA2009}. We omit these proofs due to space constraints. 

From above theorems, we see that the suboptimalities of the proposed MDSA algorithms only suffer an additive term $bD$ due to the bias in the sampling distributions of $\widehat{\bx}_t$, thereby are robust. This could be useful in settings where exact sampling is difficult, or the distribution of $\bx_{1:T}$ is not know exactly, e.g. is learned from data.

\section{Efficient online implementation}\label{sec:online-implementation}

The mirror descent stochastic approximation algorithms presented in Section \ref{sec:SA} converge to optimal solutions under mild assumptions on the problems and the stochastic conditional gradients. The output, after running $L$ iterations, is a set of $T$ policies $\overline{X}_{1:T}$ where $\overline{X}_t(\cdot,\bz_{S(t)}(w_{\internal}))\in \mathcal P_t(\mathcal X_t)$ for all $t$ (for \eqref{eq:obj_unconstrained-online}). However, even for problems where $\Xi_t = \{0,1\}$ for $2\leq t\leq T$ (and $\Xi_1 = \{0\}$), i.e. the external random variable $\bx_t$ at stage $t$ only has $2$ possible realizations, the effective dimension of the problem becomes exponential in $T$: for instance, for the stage $T$ policy $X_t\in \mathcal P_t(\mathcal X_t)$, one needs to solve for $X_t(\xi_{1:t})$ for all $\xi\in\{0\}\times \{0,1\}^{T-1}$. Moreover, in this work, we consider the general setup where the external random variables $\bx_{1:T}$ have infinitely many possible realizations, making it impossible to write down explicitly the full policies.

As a result, instead of aiming to solve all policies $X_t$, we take an \textit{online} perspective, where only the relevant component, $X_t(\bx_{1:t}(w))$ is needed, and it is needed only at stage $t$. In the product space $\Omega\times \Omega_{\internal}$, this means that we only need to \textit{evaluate} $\overline{X}_t(\bx_{1:t}(w),\bz_{S(t)}(w_{\internal}))$ at stage $t$. In Algorithm \ref{alg:hypo-ms-u}, correspondingly, one only needs to evaluate $\overline{X}^{(0:L)}_t(\bx_{1:t}(w),\bz_{S(t)}(w_{\internal})) = X^{(0:L)}_t(\bx_{1:t}(w))$ and then take the weighted average of iterates $0,\ldots,L-1$, suggesting that there is no need to evaluate the updates for all $\overline{X}_t$.  

Fortunately, the hypothetical MDSA algorithms in Section \ref{sec:SA} enjoy the benign structure that the updates are decomposable across stages and scenarios. For instance, to evaluate $\overline{X}_t^{(l)}(\xi_{1:t},\zeta_{S'(t,l-1)})$ for a particular $\xi_{1:t}\in \Xi_{1:t}$, all information needed is the following:
\begin{enumerate}
    \item $X_{t-1}^{(l-1)}(\xi_{1:(t-1)}) = \overline{X}_{t-1}^{(l-1)}(\xi_{1:(t-1)},\zeta_{S'(t-1,l-2)})$, the $(l-1)$-th iteration update of the $(t-1)$-th stage policy for $\xi_{1:(t-1)}$; 
    \item $X_{t}^{(l-1)}(\xi_{1:t}) =\overline{X}_t^{(l-1)}(\xi_{1:t},\zeta_{S'(t,l-2)})$, the $(l-1)$-th iteration update of the $t$-th stage policy for $\xi_{1:t}$; 
    \item $X_{t+1}^{(l-1)}(\xi_{1:t},\widehat{\bx}_{t+1}(\xi_{1:t},\zeta_{t,l-1})) =\overline{X}_{t+1}^{(l-1)}(\xi_{1:t},\widehat{\bx}_{t+1}(\xi_{1:t},\zeta_{t,l-1}),\zeta_{S'(t+1,l-2)})$, the $(l-1)$-th iteration update of the $(t+1)$-th stage policy for $\xi_{1:t},\widehat{\bx}_{t+1}(\xi_{1:t},\zeta_{t,l-1})$, where $\widehat{\bx}_{t+1}$ is a scenario sampler using the random variable $\bz_{t,l-1}$ (whose realization is $\zeta_{t,l-1} = \bz_{t,l-1}(w_{\internal})$). 
\end{enumerate}

Thus, information such as the policies $\overline{X}_t^{(l)}(\xi'_{1:t},\zeta_{S'(t,l)})$ for $\xi'_{1:t}\neq \xi_{1:t}$, and for policies $\overline{X}^{(l)}_s$ for $s> t+1$ and $s<t-1$ are not needed at all! Importantly, not all updates in the MDSA are needed if the output is only the relevant component of the (infinite-dimensional) policies. See Figure \ref{fig:tree-visited-3d} for an example.

\begin{figure}
\centering
\resizebox{.85\linewidth}{!}{%
\begin{minipage}{\linewidth}
\pgfdeclarelayer{main0}
\pgfdeclarelayer{main1}
\pgfdeclarelayer{main2}
\pgfdeclarelayer{edges1}
\pgfdeclarelayer{edges2}
\pgfdeclarelayer{nodes}
\pgfsetlayers{main0,edges1,main1,edges2,main2,nodes}
\centering
\resizebox{.75\linewidth}{!}{
\begin{tikzpicture}[
  x={(1cm,0cm)},
  y={(0.5cm,0.2cm)},
  z={(0cm,0.7cm)},
  node/.style={circle,draw,minimum size=2.5mm,inner sep=0pt},
  nodenode/.style={star,fill, yellow,opacity = 0.8, draw=none, minimum size=5mm, inner sep=0pt}, nodenodelabel/.style={fill=yellow, fill opacity = 0.5,font=\footnotesize},
  edge/.style={draw},
  edgeedge/.style={draw,ultra thick}
]
\def\xs{2.1}
\def\ys{1.5}
\def\zs{2.4}
\def\xmax{10}
\def\ymax{3.5}
\foreach \k in {0,1,2}{
\begin{pgfonlayer}{main\k}
  \pgfmathsetmacro{\z}{\k*\zs}
  \fill[gray!10,opacity=0.8]
    (-0.6,-\ymax,\z) --
    (\xmax,-\ymax,\z) --
    (\xmax, \ymax,\z) --
    (-0.6, \ymax,\z) -- cycle;
  \draw[gray]
    (-0.6,-\ymax,\z) --
    (\xmax,-\ymax,\z) --
    (\xmax, \ymax,\z) --
    (-0.6, \ymax,\z) -- cycle;
  \node[node] (r\k) at (0,0,\z) {};
  \node[node] (a\k) at (\xs,0,\z) {};
  \node[node] (b\k1) at (2*\xs, \ys,\z) {};
  \node[node] (b\k2) at (2*\xs, 0,\z) {};
  \node[node] (b\k3) at (2*\xs,-\ys,\z) {};
  \node[node] (c\k1) at (3*\xs, 1.2*\ys,\z) {};
  \node[node] (c\k2) at (3*\xs, 0.4*\ys,\z) {};
  \node[node] (c\k3) at (3*\xs,-0.4*\ys,\z) {};
  \node[node] (c\k4) at (3*\xs,-1.2*\ys,\z) {};
  \node[node] (d\k1) at (4*\xs, 1.6*\ys,\z) {};
  \node[node] (d\k2) at (4*\xs, 0.8*\ys,\z) {};
  \node[node] (d\k3) at (4*\xs,-0.8*\ys,\z) {};
  \node[node] (d\k4) at (4*\xs,-1.6*\ys,\z) {};
  \draw[edge] (r\k) -- (a\k);
  \draw[edge] (a\k) -- (b\k1);
  \draw[edge] (a\k) -- (b\k2);
  \draw[edge] (a\k) -- (b\k3);
  \draw[edge] (b\k1) -- (c\k1);
  \draw[edge] (b\k1) -- (c\k2);
  \draw[edge] (b\k2) -- (c\k3);
  \draw[edge] (b\k3) -- (c\k4);
  \draw[edge] (c\k1) -- (d\k1);
  \draw[edge] (c\k2) -- (d\k2);
  \draw[edge] (c\k3) -- (d\k3);
  \draw[edge] (c\k4) -- (d\k4);
  \end{pgfonlayer}
}
\begin{pgfonlayer}{edges1}
\draw[edgeedge,orange] (b01) -- (b11);
\draw[edgeedge,orange] (b01) -- (a1);
\draw[edgeedge,orange] (b01) -- (c11);
\end{pgfonlayer}
\begin{pgfonlayer}{edges2}
\draw[edgeedge,green] (b11) -- (b21);
\draw[edgeedge,green] (b11) -- (a2);
\draw[edgeedge,green] (b11) -- (c21);
\draw[edgeedge,red] (a1) -- (a2);
\draw[edgeedge,red] (a1) -- (r2);
\draw[edgeedge,red] (a1) -- (b22);
\draw[edgeedge,blue] (c11) -- (c21);
\draw[edgeedge,blue] (c11) -- (b21);
\draw[edgeedge,blue] (c11) -- (d21);
\end{pgfonlayer}
\begin{pgfonlayer}{nodes}
\node[anchor=east]
  at (-0.5,0,0*\zs) {$X^{(l)}$};
  \node[anchor=east]
  at (-0.5,0,1*\zs) {$X^{(l-1)}$};
  \node[anchor=east]
  at (-0.5,0,2*\zs) {$X^{(l-2)}$};
\node[nodenode] at (a1) {};
\node[nodenodelabel] at (a1) {$X_{t-1}^{(l-1)}(\xi_{1:(t-1)})$};
\node[nodenode] at (b11) {};
\node[nodenodelabel] at (b11) {$X_t^{(l-1)}(\xi_{1:t})$};
\node[nodenode] at (c11) {};
\node[nodenodelabel] at (c11) {$X_{t+1}^{(l-1)}(\xi_{1:t},\widehat{\bx}_{t+1})$};
\node[nodenode] at (b01) {};
\node[nodenodelabel] at (b01) {$X_t^{(l)}(\xi_{1:t})$};
\node[nodenode] at (a2) {};
\node[nodenode] at (r2) {};
\node[nodenode] at (b22) {};
\node[nodenode] at (b21) {};
\node[nodenode] at (c21) {};
\node[nodenode] at (d21) {};
\end{pgfonlayer}
\end{tikzpicture}}
    \end{minipage}%
}
\par
\caption{Policies computed using Algorithm \ref{alg:hypo-ms-u} for a problem where $T=5$. The $3$ planes represent the policies $X^{(l-2)}$, $X^{(l-1)}$, and $X^{(l)}$ (from top to bottom), computed by Algorithm \ref{alg:hypo-ms-u} plane by plane, in the order of $\{X^{(l-2)}(\xi),~\xi\in \Xi\}\to \{X^{(l-1)}(\xi),~\xi\in \Xi\}\to \{X^{(l)}(\xi),~\xi\in \Xi\}$. In each plane, the tree has $5$ layers representing $5$ stages, and each node in (horizontal) layer $t$ represents the policy $X_t$ for one realization of $\bx_{1:t}$. To evaluate $X_t^{(l)}(\xi_{1:t})$ for a particular $\xi_{1:t}$ (the yellow node in the bottom plane), only the yellow nodes connected to it using orange edges in the plane above it are needed. They correspond to $X_{t-1}^{(l-1)}(\xi_{1:(t-1)})$, $X_t^{(l-1)}(\xi_{1:t})$, and $X_{t+1}^{(l-1)}(\xi_{1:t},\widehat{\bx}_{t+1}(\xi_{1:t},\bz_{t,l-1}))$, which further need the starred nodes in the top plane connected to them using red, green, and blue edges, respectively.}
    \label{fig:tree-visited-3d}
\end{figure}

In this section, we take advantage of this decomposability and propose an efficient online updating mechanism, which computes only the \textit{necessary information} adaptive to the sequentially revealed $\bx_{1:T}$. In Section \ref{sec:eval_oracle}, we construct recursively an evaluation oracle for $\overline{X}_t^{(l)}$ and analyze its oracle complexity and additional evaluation storage; then in Section \ref{sec:online_MDSA}, we apply the proposed evaluation oracle to our MDSA algorithms, which result in implementable online algorithms. 

\subsection{Recursive construction of the evaluation oracle}\label{sec:eval_oracle}

In this section, we state the updating mechanism to evaluate $\overline{X}_t^{(l)}(\xi_{1:t},\zeta_{S'(t,l-1)})$ for a particular $\xi_{1:t}\in \Xi_{1:t}$ and $\zeta\in \bz$. Since $\zeta$ is assumed fixed, we adopt the notation in Algorithm \ref{alg:hypo-ms-u} and abbreviate it as $X_t^{(l)}(\xi_{1:t})$. 

To formalize the discussion on the memory requirement during the update, we assume that all processes during an algorithm have access to a shared memory space denoted as $\memo$, which admits the following operations: $\rr{x}$, $\ww{x}$, $\del{x}$, which reads, writes, and deletes the value of $x$ from $\memo$. When some information $x$ or some set $\mathcal S$ is in $\memo$, we abbreviate it as $x\in \memo$ and $\mathcal S\subset \memo$.  

\textbf{Oracle costs.} We assume access to conditional scenario samplers satisfying Definition \ref{def:conditional-scenario-sampler}, stage derivatives, initial-policy evaluation procedures, and Bregman proximal maps. Our complexity bounds count stochastic conditional-gradient evaluations and proximal updates as primitive calls; each conditional-gradient evaluation uses at most one conditional sample and two stage-derivative evaluations. Constructing these oracles and the arithmetic cost of their calls, including conditional sampling and initialization evaluation, are not included in the reported bounds.

\textbf{Storage and access model.} We assume read access to representations of the initial policies $X_t^{(0)}$ and to the indexed random seeds $\zeta_{t,l}$ used in constructing the stochastic conditional gradients. These objects belong to the permanent initial memory
\begin{equation*}
    \memo^{(init)}:=\{X_t^{(0)}:t=1,\ldots,T\}\cup\{\zeta_{t,l}:t=1,\ldots,T,\ l=0,\ldots,L\}.
\end{equation*}
An initial policy is represented by an evaluation procedure, not by enumerating its values at all histories. The size of that representation is not bounded here. In particular, stage-dependent constant policies require $T$ stored decision vectors if represented explicitly. As written, the seed array contains $T(L+1)$ seeds; its storage is also excluded from the additional evaluation bounds. We do not assume seed regeneration.

The bounds below count only the decision vectors retained by the recursive evaluation procedure beyond its initial memory state. They exclude the following storage, which must be accounted for separately in a complete implementation:
\begin{itemize}
    \item \textit{Initialization and seeds:} the policy representations and indexed seeds in $\memo^{(init)}$.
    \item \textit{Model:} the stage functions, feasible sets, conditional-sampler representations, and other oracle data.
    \item \textit{Observed and simulated histories:} the conditional sampler may require the entire history $\xi_{1:t}$; neither a bounded sufficient statistic nor horizon-independent history storage is assumed. Storage of simulated histories, node identifiers, and recursion bookkeeping is also outside the decision-vector count.
    \item \textit{Output archive:} decisions are emitted sequentially. Retaining all $T$ outputs requires storage for $T$ decision vectors in addition to the evaluation storage.
    \item \textit{Oracle workspace:} temporary storage used inside initialization evaluation, conditional sampling, derivative evaluation, and Bregman projection.
\end{itemize}
Thus $O(L^2B)$ is an additional evaluation-storage bound, not a total-memory bound. At fixed accuracy and fixed vector representation size, it is independent of $T$, whereas the excluded storage can grow with $T$.

For the storage formulas, $B$ denotes an assumed representation size per decision vector, measured in bits under an idealized storage convention. This is bookkeeping for the number of stored vectors, not a claim that arbitrary real vectors admit exact finite-bit representations. The convergence analysis assumes exact arithmetic and does not quantify roundoff or finite-precision errors. Alternatively, the storage result can be read directly as $O(L^2)$ decision vectors.

On a high level, to evaluate $X_t^{(l)}(\xi_{1:t})$, we assume that $X_{t-1}^{(0:(l-1))}(\xi_{1:(t-1)})$ has already been evaluated, then for $l'=1,\ldots,l$, the following is computed:
\begin{itemize}
    \item {\color{red}{step 1}}: calling the evaluation procedure for $X_{t+1}^{(l'-1)}(\xi_{1:t},\widehat{\bx}_{t+1}(\xi_{1:t},\zeta_{t,l'-1}))$, i.e. the $(l'-1)$-th update of $X_{t+1}$ at the sampled ($(\xi_{1:t},\widehat{\bx}_{t+1}(\xi_{1:t},\zeta_{t,l'-1}))$;
    \item {\color{blue}{step 2}}: use $X_{t-1}^{(l'-1)}(\xi_{1:(t-1)})$, $X_{t}^{(l'-1)}(\xi_{1:t})$, and $X_{t+1}^{(l'-1)}(\xi_{1:t},\widehat{\bx}_{t+1}(\xi_{1:t},\zeta_{t,l'-1}))$ to evaluate $X_t^{(l')}(\xi_{1:t})$.
\end{itemize}

\begin{figure}[h!]
\centering
\resizebox{.85\linewidth}{!}{%
\begin{minipage}{\linewidth}
    \centering
    \resizebox{.75\linewidth}{!}{
    \begin{tikzpicture}[
  sq/.style={draw, rectangle, minimum width=4cm},
  ssq/.style={draw, rectangle, minimum width=4cm,pattern=north east lines,
    pattern color=black,
    fill opacity=0.3,     
    text opacity=1},
  sqq/.style={rectangle, minimum width=4cm},
edge/.style={draw}
]
\def\xs{5cm}
\def\ys{1.3cm}
\draw[fill=red!10, fill opacity=0.5,dashed,rounded corners,draw=red] (0.55*\xs,-1.4*\ys) rectangle (1.45*\xs,0.4*\ys);
\coordinate (rect) at (1.45*\xs,-0.5*\ys);
\draw[fill=blue!10, fill opacity=0.5,rounded corners,draw=blue] (-.45*\xs,-2.4*\ys) rectangle (2.55*\xs,-1.6*\ys);
\coordinate (rect2) at (1*\xs,-2.4*\ys);
\node[ssq] (a0) at (0, 0) {$X^{(0)}_{t-1}(\xi_{1:(t-1)})$};
\node[sqq] (a1) at (0, -\ys) {$\vdots$};
\node[ssq] (a2) at (0, -2*\ys) {$X^{(l'-1)}_{t-1}(\xi_{1:(t-1)})$};
\node[sqq] (a3) at (0, -3*\ys) {$\vdots$};
\node[ssq] (a4) at (0, -4*\ys) {$X^{(l-1)}_{t-1}(\xi_{1:(t-1)})$};
\node[ssq] (b0) at (\xs, 0) {$X^{(0)}_{t}(\xi_{1:t})$};
\node[sqq] (b1) at (\xs, -\ys) {$\vdots$};
\node[ssq] (b2) at (\xs, -2*\ys) {$X^{(l'-1)}_{t}(\xi_{1:t})$};
\node[sq] (b3) at (\xs, -3*\ys) {$X^{(l')}_{t}(\xi_{1:t})$};
\node[sq] (c) at (2*\xs, -2*\ys) {$X^{(l'-1)}_{t+1}(\xi_{1:t},\widehat{\bx}_{t+1}(\xi_{1:t},\zeta_{t,l'-1}))$};
\draw[->,thick,dashed,red] (rect) --node[midway, right] { \color{red}{step 1}} (c) ;
\draw[->,thick,blue] (rect2) --node[midway, right] { \color{blue}{step 2}} (b3) ;
\end{tikzpicture}}
    \end{minipage}%
}
\par
\caption{$l'$-th iteration in the evaluation procedure for $X_t^{(l)}(\xi_{1:t})$. Assume that all the hatched nodes $X_{t-1}^{(0:(l-1))}(\xi_{1:(t-1)})$ and $X_t^{(0:(l'-1))}(\xi_{1:t})$ have already been evaluated, then Algorithm \ref{alg:update-procedure} 1) calls the evaluation procedure for $X_{t+1}^{(l'-1)}(\xi_{1:t},\widehat{\bx}_{t+1}(\xi_{1:t},\zeta_{t,l'-1}))$ (this is valid since $X_t^{(0:(l'-2))}(\xi_{1:t})$ has already been computed); 2) uses $X_{t-1}^{(l'-1)}(\xi_{1:(t-1)})$, $X_{t}^{(l'-1)}(\xi_{1:t})$, and $X_{t+1}^{(l'-1)}(\xi_{1:t},\widehat{\bx}_{t+1}(\xi_{1:t},\zeta_{t,l'-1}))$ to evaluate $X_t^{(l')}(\xi_{1:t})$.
}
    \label{fig:algo-idea}
\end{figure}

More precisely, we have the following evaluation procedure in Algorithm \ref{alg:update-procedure}. We show its validity, and oracle complexity and additional evaluation storage in Lemma \ref{lm:procedure-prop}.

\begin{algorithm}
\caption{Evaluation procedure for $X_t^{(l)}(\xi_{1:t})$}\label{alg:update-procedure}
\begin{algorithmic}
\Require $\xi_{1:t}$, $\memo^{(init)}\subset \memo$, and if $t\geq 2$, $ X_{t-1}^{(l')}(\xi_{1:(t-1)})\in \memo$ for $l'=0,\ldots,l-1$.
\Ensure $X_t^{(l')}(\xi_{1:t})\in \memo$ for $l'=0,\ldots,l$. 

\If{$t = T$} 
\For{$l'=1,\ldots,l$}
\State $\rr{X_T^{(l'-1)}(\xi_{1:T}),X_{T-1}^{(l'-1)}(\xi_{1:(T-1)})}$
\State {\color{blue}{Step 2}}: compute $G_T^{(l'-1)}(\xi_{1:T}) = \frac{\partial}{\partial x_t} f_t(X_{T-1}^{(l'-1)}(\xi_{1:(T-1)}),X_T^{(l'-1)}(\xi_{1:T}),\xi_T)$, then compute $X_T^{(l')}(\xi_{1:T})$ using \eqref{eq:MD1}, $\ww{X_T^{(l')}(\xi_{1:T})}$
\EndFor
\Else 
\For{$l'=1,\ldots,l$}
\State $\rr{\zeta_{t,l'-1}}$, compute $\widehat{\bx}_{t+1}(\xi_{1:t},\zeta_{t,l'-1})$
\State {\color{red}{Step 1}}: evaluate $X_{t+1}^{(l'-1)}(\xi_{1:t},\widehat{\bx}_{t+1}(\xi_{1:t},\zeta_{t,l'-1}))$
\State $\rr{X_t^{(l'-1)}(\xi_{1:t}),X_{t+1}^{(l'-1)}(\xi_{1:t},\widehat{\bx}_{t+1}(\xi_{1:t},\zeta_{t,l'-1}))}$, and if $t\geq 2$, $\rr{X_{t-1}^{(l'-1)}(\xi_{1:(t-1)})}$
\State {\color{blue}{Step 2}}: compute $G_t^{(l'-1)}(\xi_{1:t}) $ using \eqref{eq:choice-of-G}, then compute $X_t^{(l')}(\xi_{1:t})$ using \eqref{eq:MD1}
\State $\del{X_{t+1}^{(1:(l'-1)}(\xi_{1:t},\widehat{\bx}_{t+1}(\xi_{1:t},\zeta_{t,l'-1}))}$, $\ww{X_t^{(l')}(\xi_{1:t})}$
\EndFor
\EndIf
\end{algorithmic}
\end{algorithm}

\textbf{Remark.} Similar evaluation oracle can be constructed for Algorithm \ref{alg:hypo-ms-s} for the saddle point problem \eqref{eq:minmax-online}, with $X_t^{(l)}$ replaced with $Z_t^{(l)}$, \eqref{eq:choice-of-G} replaced with the corresponding one for the saddle point problem, and \eqref{eq:MD1} replaced with \eqref{eq:MD2}. 

\begin{lemma}\label{lm:procedure-prop}
    The updates in Algorithm \ref{alg:update-procedure} are valid: when $X_{t+1}^{(l'-1)}(\xi_{1:t},\widehat{\bx}_{t+1}(\xi_{1:t},\zeta_{t,l'-1}))$ is evaluated, $\memo$ contains all the required information (i.e. $X_{t}^{(0:(l'-2))}(\xi_{1:t})$).
    
The oracle complexity for the stochastic conditional gradient and the proximal update is no more than $\min(2^l-1,l^{T-t+1})$.

Denoting the shared memory space at the start and the end of the procedure as $\overline{\memo}_0$ and $\overline{\memo}_1$ respectively, then $\overline{\memo}_1 = \overline{\memo}_0\cup \{X_t^{(1:l)}(\xi_{1:t})\}$. Further assuming that $x_{t'}\in \mathcal X_{t'}$ can be stored with $B$ bits for all $t'$, then $\memo\setminus\overline{\memo}_0$ requires no more than $O(l^2B)$ bits of additional evaluation storage throughout the procedure, under the storage convention above. 
\end{lemma}
The proof is provided in Appendix \ref{app:proof:lm:procedure-prop}.

\textbf{Remark.} Lemma \ref{lm:procedure-prop} bounds additional decision-vector storage only. The initialization, seed, model, history, output-archive, and oracle-workspace costs listed above must be added when assessing total memory.

\subsection{Efficient online MDSA}\label{sec:online_MDSA}

To evaluate the sequence $\overline{X}_t(\bx_{1:t}(w),\bz_{S(t)}(w_{\internal})$ in an online fashion, the overall updates become Algorithm \ref{alg:async}, where at time $t$, we apply Algorithm \ref{alg:update-procedure} to evaluate $X^{(L)}_t(\bx_{1:t}(w))$, which along the way evaluates ${X}^{(l)}_t(\bx_{1:t}(w))$ for $l=1,\ldots,L-1$. Then, the output is the weighted average of ${X}^{(0:(L-1))}_t(\bx_{1:t}(w))$ using weights $\gamma_0,\ldots,\gamma_{L-1}$. Its correctness and complexity are stated in Theorem \ref{thm:main-async}. Figure \ref{fig:update} illustrates the stagewise evaluation order for $L=2$. 

\begin{figure}[htbp]
\centering
\resizebox{.9\linewidth}{!}{
\begin{tikzpicture}
    \node[font=\huge] at (0,4.5) {$t=1$};
    \draw (0,-2) -- ++(0,6);
\end{tikzpicture}
\hspace{0.2em}
\begin{tikzpicture}[
reds/.style={circle,draw,fill=blue!10,minimum size=20},
redsv/.style={circle,draw,line width=0.8mm,fill=blue!10,minimum size=20},
redsvh/.style={circle,draw,line width=0.8mm,minimum size=20,preaction={pattern=north east lines, pattern color=blue}},
greens/.style={circle,draw,fill=green!10,minimum size=20},
greensv/.style={circle,draw,line width=0.8mm,fill=green!10,minimum size=20},
blues/.style={circle,draw,fill=red!10,minimum size=20},
bluesv/.style={circle,draw,line width=0.8mm,fill=red!10,minimum size=20}, nodenode/.style={star,fill=yellow, fill opacity=0.5, text opacity=1,draw=none, minimum size=20}]
\node [redsv] (layer1-0) at (0,2.5) {};
\node [nodenode] at (layer1-0) {$0$};
\node [greensv] (layer2-0) at (1,1.5) {$0$};
\foreach \x in {1,...,2}
{\pgfmathtruncatemacro{\label}{\x}
    \node [greens] (layer2-\x) at (1,1.5 +\x) {$0$};}
\foreach \xx in {0,...,5}
{\pgfmathtruncatemacro{\label}{\xx}
    \node [blues] (layer3-\xx) at (2,\xx) {$0$};}
\draw [line width=0.8mm,<-] (layer1-0)--(layer2-0);
\draw (layer1-0)--(layer2-1);
\draw (layer1-0)--(layer2-2);
\draw (layer2-0)--(layer3-0);
\draw (layer2-0)--(layer3-1);
\draw (layer2-1)--(layer3-2);
\draw (layer2-1)--(layer3-3);
\draw (layer2-2)--(layer3-4);
\draw (layer2-2)--(layer3-5);
\end{tikzpicture}
\hspace{0.2em}
\begin{tikzpicture}[
reds/.style={circle,draw,fill=blue!10,minimum size=20},
redsv/.style={circle,draw,line width=0.8mm,fill=blue!10,minimum size=20},
greens/.style={circle,draw,fill=green!10,minimum size=20},
greensv/.style={circle,draw,line width=0.8mm,fill=green!10,minimum size=20},
blues/.style={circle,draw,fill=red!10,minimum size=20},
bluesv/.style={circle,draw,line width=0.8mm,fill=red!10,minimum size=20}, nodenode/.style={star,fill=yellow, fill opacity=0.5, text opacity=1,draw=none, minimum size=20}]
\node [redsv] (layer1-0) at (0,2.5) {};
\node [nodenode] at (layer1-0) {$1$};
\node [greens] (layer2-0) at (1,1.5) {$0$};
\node [greens] (layer2-1) at (1,2.5) {$0$};
\node [greensv] (layer2-2) at (1,3.5) {$0$};
\foreach \xx in {0,...,3}
{\pgfmathtruncatemacro{\label}{\xx}
    \node [blues] (layer3-\xx) at (2,\xx) {$0$};}
\node [bluesv] (layer3-4) at (2,4) {$0$};
\node [blues] (layer3-5) at (2,5) {$0$};
\draw (layer1-0)--(layer2-0);
\draw (layer1-0)--(layer2-1);
\draw [line width=0.8mm,<-,dashed] (layer1-0)--(layer2-2);
\draw (layer2-0)--(layer3-0);
\draw (layer2-0)--(layer3-1);
\draw (layer2-1)--(layer3-2);
\draw (layer2-1)--(layer3-3);
\draw [line width=0.8mm,<-](layer2-2)--(layer3-4);
\draw (layer2-2)--(layer3-5);
\end{tikzpicture}
\hspace{0.2em}
\begin{tikzpicture}[
reds/.style={circle,draw,fill=blue!10,minimum size=20},
redsv/.style={circle,draw,line width=0.8mm,fill=blue!10,minimum size=20},
greens/.style={circle,draw,fill=green!10,minimum size=20},
greensv/.style={circle,draw,line width=0.8mm,fill=green!10,minimum size=20},
blues/.style={circle,draw,fill=red!10,minimum size=20},
bluesv/.style={circle,draw,line width=0.8mm,fill=red!10,minimum size=20}, nodenode/.style={star,fill=yellow, fill opacity=0.5, text opacity=1,draw=none, minimum size=20}]
\node [redsv] (layer1-0) at (0,2.5) {};
\node [nodenode] at (layer1-0) {$1$};
\node [greens] (layer2-0) at (1,1.5) {$0$};
\node [greens] (layer2-1) at (1,2.5) {$0$};
\node [greensv] (layer2-2) at (1,3.5) {$1$};
\foreach \xx in {0,...,3}
{\pgfmathtruncatemacro{\label}{\xx}
    \node [blues] (layer3-\xx) at (2,\xx) {$0$};}
\node [blues] (layer3-4) at (2,4) {$0$};
\node [blues] (layer3-5) at (2,5) {$0$};
\draw (layer1-0)--(layer2-0);
\draw (layer1-0)--(layer2-1);
\draw [line width=0.8mm,<-] (layer1-0)--(layer2-2);
\draw (layer2-0)--(layer3-0);
\draw (layer2-0)--(layer3-1);
\draw (layer2-1)--(layer3-2);
\draw (layer2-1)--(layer3-3);
\draw (layer2-2)--(layer3-4);
\draw (layer2-2)--(layer3-5);
\end{tikzpicture}
\hspace{0.2em}
\begin{tikzpicture}[
reds/.style={circle,draw,fill=blue!10,minimum size=20},
redsv/.style={circle,draw,line width=0.8mm,fill=blue!10,minimum size=20},
greens/.style={circle,draw,fill=green!10,minimum size=20},
greensv/.style={circle,draw,line width=0.8mm,fill=green!10,minimum size=20},
blues/.style={circle,draw,fill=red!10,minimum size=20},
bluesv/.style={circle,draw,line width=0.8mm,fill=red!10,minimum size=20}, nodenode/.style={star,fill=yellow, fill opacity=0.5, text opacity=1,draw=none, minimum size=20}]
\node [reds] (layer1-0) at (0,2.5) {};
\node [nodenode] at (layer1-0) {$2$};
\node [greens] (layer2-0) at (1,1.5) {$0$};
\node [greens] (layer2-1) at (1,2.5) {$0$};
\node [greens] (layer2-2) at (1,3.5) {$0$};
\foreach \xx in {0,...,3}
{\pgfmathtruncatemacro{\label}{\xx}
    \node [blues] (layer3-\xx) at (2,\xx) {$0$};}
\node [blues] (layer3-4) at (2,4) {$0$};
\node [blues] (layer3-5) at (2,5) {$0$};
\draw (layer1-0)--(layer2-0);
\draw (layer1-0)--(layer2-1);
\draw (layer1-0)--(layer2-2);
\draw (layer2-0)--(layer3-0);
\draw (layer2-0)--(layer3-1);
\draw (layer2-1)--(layer3-2);
\draw (layer2-1)--(layer3-3);
\draw (layer2-2)--(layer3-4);
\draw (layer2-2)--(layer3-5);
\end{tikzpicture}
\hspace{0.2em}
\begin{tikzpicture}
    \node[font=\huge] at (0,4.5) {$t=2$};
    \draw (0,-2) -- ++(0,6);
\end{tikzpicture}
\hspace{0.2em}
\begin{tikzpicture}[
reds/.style={circle,draw,fill=blue!10,minimum size=20},
redsv/.style={circle,draw,line width=0.8mm,fill=blue!10,minimum size=20},
greens/.style={circle,draw,fill=green!10,minimum size=20},
greensv/.style={circle,draw,line width=0.8mm,fill=green!10,minimum size=20},
blues/.style={circle,draw,fill=red!10,minimum size=20},
bluesv/.style={circle,draw,line width=0.8mm,fill=red!10,minimum size=20}, nodenode/.style={star,fill=yellow, fill opacity=0.5, text opacity=1,draw=none, minimum size=20}]
\node [redsv] (layer1-0) at (0,2.5) {};
\node [nodenode] at (layer1-0) {$2$};
\node [greens] (layer2-0) at (1,1.5) {$0$};
\node [greensv] (layer2-1) at (1,2.5) {};
\node [nodenode] at (layer2-1) {$0$};
\node [greens] (layer2-2) at (1,3.5) {$0$};
\foreach \xx in {0,...,2}
{\pgfmathtruncatemacro{\label}{\xx}
    \node [blues] (layer3-\xx) at (2,\xx) {$0$};}
\node [bluesv] (layer3-3) at (2,3) {$0$};
\node [blues] (layer3-4) at (2,4) {$0$};
\node [blues] (layer3-5) at (2,5) {$0$};
\draw (layer1-0)--(layer2-0);
\draw [line width=0.8mm,->] (layer1-0)--(layer2-1);
\draw (layer1-0)--(layer2-2);
\draw (layer2-0)--(layer3-0);
\draw (layer2-0)--(layer3-1);
\draw (layer2-1)--(layer3-2);
\draw [line width=0.8mm,<-] (layer2-1)--(layer3-3);
\draw (layer2-2)--(layer3-4);
\draw (layer2-2)--(layer3-5);
\end{tikzpicture}
\hspace{0.2em}
\begin{tikzpicture}[
reds/.style={circle,draw,fill=blue!10,minimum size=20},
redsv/.style={circle,draw,line width=0.8mm,fill=blue!10,minimum size=20},
greens/.style={circle,draw,fill=green!10,minimum size=20},
greensv/.style={circle,draw,line width=0.8mm,fill=green!10,minimum size=20},
blues/.style={circle,draw,fill=red!10,minimum size=20},
bluesv/.style={circle,draw,line width=0.8mm,fill=red!10,minimum size=20}, nodenode/.style={star,fill=yellow, fill opacity=0.5, text opacity=1,draw=none, minimum size=20}]
\node [redsv] (layer1-0) at (0,2.5) {};
\node [nodenode] at (layer1-0) {$2$};
\node [greens] (layer2-0) at (1,1.5) {$0$};
\node [greensv] (layer2-1) at (1,2.5) {};
\node [nodenode] at (layer2-1) {$1$};
\node [greens] (layer2-2) at (1,3.5) {$0$};
\node [blues] (layer3-5) at (2,5) {$0$};
\foreach \xx in {0,...,1}
{\pgfmathtruncatemacro{\label}{\xx}
    \node [blues] (layer3-\xx) at (2,\xx) {$0$};}
\foreach \xx in {3,...,5}
{\pgfmathtruncatemacro{\label}{\xx}
    \node [blues] (layer3-\xx) at (2,\xx) {$0$};}
\node [bluesv] (layer3-2) at (2,2) {$0$};
\draw (layer1-0)--(layer2-0);
\draw [line width=0.8mm,->] (layer1-0)--(layer2-1);
\draw (layer1-0)--(layer2-2);
\draw (layer2-0)--(layer3-0);
\draw (layer2-0)--(layer3-1);
\draw [line width=0.8mm,<-,dashed] (layer2-1)--(layer3-2);
\draw (layer2-1)--(layer3-3);
\draw (layer2-2)--(layer3-4);
\draw (layer2-2)--(layer3-5);
\end{tikzpicture}
\hspace{0.2em}
\begin{tikzpicture}[
reds/.style={circle,draw,fill=blue!10,minimum size=20},
redsv/.style={circle,draw,line width=0.8mm,fill=blue!10,minimum size=20},
greens/.style={circle,draw,fill=green!10,minimum size=20},
greensv/.style={circle,draw,line width=0.8mm,fill=green!10,minimum size=20},
blues/.style={circle,draw,fill=red!10,minimum size=20},
bluesv/.style={circle,draw,line width=0.8mm,fill=red!10,minimum size=20}, nodenode/.style={star,fill=yellow, fill opacity=0.5, text opacity=1,draw=none, minimum size=20}]
\node [redsv] (layer1-0) at (0,2.5) {};
\node [nodenode] at (layer1-0) {$2$};
\node [greens] (layer2-0) at (1,1.5) {$0$};
\node [greensv] (layer2-1) at (1,2.5) {};
\node [nodenode] at (layer2-1) {$1$};
\node [greens] (layer2-2) at (1,3.5) {$0$};
\node [blues] (layer3-5) at (2,5) {$0$};
\foreach \xx in {0,...,1}
{\pgfmathtruncatemacro{\label}{\xx}
    \node [blues] (layer3-\xx) at (2,\xx) {$0$};}
\foreach \xx in {3,...,5}
{\pgfmathtruncatemacro{\label}{\xx}
    \node [blues] (layer3-\xx) at (2,\xx) {$0$};}
\node [bluesv] (layer3-2) at (2,2) {$1$};
\draw (layer1-0)--(layer2-0);
\draw [line width=0.8mm,->] (layer1-0)--(layer2-1);
\draw (layer1-0)--(layer2-2);
\draw (layer2-0)--(layer3-0);
\draw (layer2-0)--(layer3-1);
\draw [line width=0.8mm,<-] (layer2-1)--(layer3-2);
\draw (layer2-1)--(layer3-3);
\draw (layer2-2)--(layer3-4);
\draw (layer2-2)--(layer3-5);
\end{tikzpicture}
\hspace{0.2em}
\begin{tikzpicture}[
reds/.style={circle,draw,fill=blue!10,minimum size=20},
redsv/.style={circle,draw,line width=0.8mm,fill=blue!10,minimum size=20},
greens/.style={circle,draw,fill=green!10,minimum size=20},
greensv/.style={circle,draw,line width=0.8mm,fill=green!10,minimum size=20},
blues/.style={circle,draw,fill=red!10,minimum size=20},
bluesv/.style={circle,draw,line width=0.8mm,fill=red!10,minimum size=20}, nodenode/.style={star,fill=yellow, fill opacity=0.5, text opacity=1,draw=none, minimum size=20}]
\node [reds] (layer1-0) at (0,2.5) {};
\node [nodenode] at (layer1-0) {$2$};
\node [greens] (layer2-0) at (1,1.5) {$0$};
\node [greens] (layer2-1) at (1,2.5) {};
\node [nodenode] at (layer2-1) {$2$};
\node [greens] (layer2-2) at (1,3.5) {$0$};
\node [blues] (layer3-5) at (2,5) {$0$};
\foreach \xx in {0,...,1}
{\pgfmathtruncatemacro{\label}{\xx}
    \node [blues] (layer3-\xx) at (2,\xx) {$0$};}
\foreach \xx in {3,...,5}
{\pgfmathtruncatemacro{\label}{\xx}
    \node [blues] (layer3-\xx) at (2,\xx) {$0$};}
\node [blues] (layer3-2) at (2,2) {$0$};
\draw (layer1-0)--(layer2-0);
\draw (layer1-0)--(layer2-1);
\draw (layer1-0)--(layer2-2);
\draw (layer2-0)--(layer3-0);
\draw (layer2-0)--(layer3-1);
\draw (layer2-1)--(layer3-2);
\draw (layer2-1)--(layer3-3);
\draw (layer2-2)--(layer3-4);
\draw (layer2-2)--(layer3-5);
\end{tikzpicture}
\hspace{0.2em}
\begin{tikzpicture}
    \node[font=\huge] at (0,4.5) {$t=3$};
    \draw (0,-2) -- ++(0,6);
\end{tikzpicture}
\hspace{0.2em}
\begin{tikzpicture}[
reds/.style={circle,draw,fill=blue!10,minimum size=20},
redsv/.style={circle,draw,line width=0.8mm,fill=blue!10,minimum size=20},
greens/.style={circle,draw,fill=green!10,minimum size=20},
greensv/.style={circle,draw,line width=0.8mm,fill=green!10,minimum size=20},
blues/.style={circle,draw,fill=red!10,minimum size=20},
bluesv/.style={circle,draw,line width=0.8mm,fill=red!10,minimum size=20}, nodenode/.style={star,fill=yellow, fill opacity=0.5, text opacity=1,draw=none, minimum size=20}]
\node [reds] (layer1-0) at (0,2.5) {};
\node [nodenode] at (layer1-0) {$2$};
\node [greens] (layer2-0) at (1,1.5) {$0$};
\node [greensv] (layer2-1) at (1,2.5) {};
\node [nodenode] at (layer2-1) {$2$};
\node [greens] (layer2-2) at (1,3.5) {$0$};
\node [blues] (layer3-5) at (2,5) {$0$};
\foreach \xx in {0,...,1}
{\pgfmathtruncatemacro{\label}{\xx}
    \node [blues] (layer3-\xx) at (2,\xx) {$0$};}
\foreach \xx in {3,...,5}
{\pgfmathtruncatemacro{\label}{\xx}
    \node [blues] (layer3-\xx) at (2,\xx) {$0$};}
\node [bluesv] (layer3-2) at (2,2) {};
\node [nodenode] at (layer3-2) {$0$};
\draw (layer1-0)--(layer2-0);
\draw  (layer1-0)--(layer2-1);
\draw (layer1-0)--(layer2-2);
\draw (layer2-0)--(layer3-0);
\draw (layer2-0)--(layer3-1);
\draw [line width=0.8mm,->] (layer2-1)--(layer3-2);
\draw (layer2-1)--(layer3-3);
\draw (layer2-2)--(layer3-4);
\draw (layer2-2)--(layer3-5);
\end{tikzpicture}
\hspace{0.2em}
\begin{tikzpicture}[
reds/.style={circle,draw,fill=blue!10,minimum size=20},
redsv/.style={circle,draw,line width=0.8mm,fill=blue!10,minimum size=20},
greens/.style={circle,draw,fill=green!10,minimum size=20},
greensv/.style={circle,draw,line width=0.8mm,fill=green!10,minimum size=20},
blues/.style={circle,draw,fill=red!10,minimum size=20},
bluesv/.style={circle,draw,line width=0.8mm,fill=red!10,minimum size=20}, nodenode/.style={star,fill=yellow, fill opacity=0.5, text opacity=1,draw=none, minimum size=20}]
\node [reds] (layer1-0) at (0,2.5) {};
\node [nodenode] at (layer1-0) {$2$};
\node [greens] (layer2-0) at (1,1.5) {$0$};
\node [greensv] (layer2-1) at (1,2.5) {};
\node [nodenode] at (layer2-1) {$2$};
\node [greens] (layer2-2) at (1,3.5) {$0$};
\node [blues] (layer3-5) at (2,5) {$0$};
\foreach \xx in {0,...,1}
{\pgfmathtruncatemacro{\label}{\xx}
    \node [blues] (layer3-\xx) at (2,\xx) {$0$};}
\foreach \xx in {3,...,5}
{\pgfmathtruncatemacro{\label}{\xx}
    \node [blues] (layer3-\xx) at (2,\xx) {$0$};}
\node [bluesv] (layer3-2) at (2,2) {};
\node [nodenode] at (layer3-2) {$1$};
\draw (layer1-0)--(layer2-0);
\draw  (layer1-0)--(layer2-1);
\draw (layer1-0)--(layer2-2);
\draw (layer2-0)--(layer3-0);
\draw (layer2-0)--(layer3-1);
\draw [line width=0.8mm,->] (layer2-1)--(layer3-2);
\draw (layer2-1)--(layer3-3);
\draw (layer2-2)--(layer3-4);
\draw (layer2-2)--(layer3-5);
\end{tikzpicture}
\hspace{0.2em}
\begin{tikzpicture}[
reds/.style={circle,draw,fill=blue!10,minimum size=20},
redsv/.style={circle,draw,line width=0.8mm,fill=blue!10,minimum size=20},
greens/.style={circle,draw,fill=green!10,minimum size=20},
greensv/.style={circle,draw,line width=0.8mm,fill=green!10,minimum size=20},
blues/.style={circle,draw,fill=red!10,minimum size=20},
bluesv/.style={circle,draw,line width=0.8mm,fill=red!10,minimum size=20}, nodenode/.style={star,fill=yellow, fill opacity=0.5, text opacity=1,draw=none, minimum size=20}]
\node [reds] (layer1-0) at (0,2.5) {};
\node [nodenode] at (layer1-0) {$2$};
\node [greens] (layer2-0) at (1,1.5) {$0$};
\node [greens] (layer2-1) at (1,2.5) {};
\node [nodenode] at (layer2-1) {$2$};
\node [greens] (layer2-2) at (1,3.5) {$0$};
\node [blues] (layer3-5) at (2,5) {$0$};
\foreach \xx in {0,...,1}
{\pgfmathtruncatemacro{\label}{\xx}
    \node [blues] (layer3-\xx) at (2,\xx) {$0$};}
\foreach \xx in {3,...,5}
{\pgfmathtruncatemacro{\label}{\xx}
    \node [blues] (layer3-\xx) at (2,\xx) {$0$};}
\node [blues] (layer3-2) at (2,2) {};
\node [nodenode] at (layer3-2) {$2$};
\draw (layer1-0)--(layer2-0);
\draw  (layer1-0)--(layer2-1);
\draw (layer1-0)--(layer2-2);
\draw (layer2-0)--(layer3-0);
\draw (layer2-0)--(layer3-1);
\draw (layer2-1)--(layer3-2);
\draw (layer2-1)--(layer3-3);
\draw (layer2-2)--(layer3-4);
\draw (layer2-2)--(layer3-5);
\end{tikzpicture}
\hspace{0.2em}

}
\caption{Example of Algorithm \ref{alg:async} with $L=2$, following the notations in Figure \ref{fig:tree}. Numbers represent the number of updates that has been applied to each node. Only nodes with bold boundaries are visited, and only the starred nodes are the relevant output. At stage $1$, $\xi_1 = 1$ is revealed, then $X_1^{(1)}(1)$, $X_2^{(1)}(1,3)$, and $X_1^{(2)}(1)$ are computed (left 4 figures). At stage $2$, $\xi_2 = 2$ is revealed, then $X_2^{(1)}(1,2)$, $X_3^{(1)}(1,2,3)$, and $X_2^{(2)}(1,2)$ are computed (middle 4 figures). At stage $3$, $\xi_3 = 3$ is revealed, then $X_3^{(1)}(1,2,3)$ are $X_3^{(2)}(1,2,3)$ are computed (right 3 figures).}\label{fig:update}
\end{figure}

\begin{algorithm}
\caption{Efficient online MDSA}\label{alg:async}
\begin{algorithmic}
\Require At $t=1,\ldots,T$, $\bx_t(w) = \xi_t$ is revealed. At $t=1$, $\memo^{(init)}\subset \memo$. $L$, the number of updates required. 
\Ensure At (the end of) $t = 1,\ldots,T$, $X_t(\xi_{1:t})\in \memo$.
\For{$t = 1,\ldots,T$}
\State Evaluate $X_t^{(L)}(\xi_{1:t})$ using Algorithm \ref{alg:update-procedure}
\State Compute $X_t(\xi_{1:t})$ using \eqref{eq:ergodic-u}, $\ww{X_t(\xi_{1:t})}$
\State Emit $X_t(\xi_{1:t})$ to the output stream
\State If $t\geq 2$, $\del{X_{t-1}(\xi_{1:(t-1)}),X_{t-1}^{(1:L)}(\xi_{1:(t-1)})}$
\EndFor
\end{algorithmic}
\end{algorithm}

\begin{theorem}\label{thm:main-async}
The updates in Algorithm \ref{alg:async} are valid: when $X_t^{(L)}(\xi_{1:t})$ is evaluated, all the needed information is available. Denoting the shared memory space at the start of Algorithm \ref{alg:async} as $\overline{\memo}_0$ and at the end of stage $t$ as $\overline{\memo}_t$, then 
\begin{equation}\label{eq:memo_update}
    \overline{\memo}_t= \overline{\memo}_0\cup \{X_{t}^{(1:L)}(\xi_{1:t}), X_t(\xi_{1:t})\}.
\end{equation}
That is, $\overline{\memo}_t\setminus\overline{\memo}_0$ contains $L+1$ vectors in $\mathcal X_t$. 

The oracle complexity for the stochastic conditional gradient and the proximal update is $O(T \cdot \min(2^L,L^T))$.

Under the vector-representation convention in Section \ref{sec:eval_oracle}, $\memo\setminus\overline{\memo}_0$ uses no more than $O(L^2B)$ bits of additional evaluation storage throughout the algorithm. Outputs are emitted sequentially; total implementation memory additionally includes the excluded storage listed in that section.

\end{theorem}

The proof is provided in Appendix \ref{app:proof:thm:main-async}.

\textbf{Remark.} The online setting concerns sequential revelation of the actual observations. At stage $t$, recursive evaluation can query simulated histories through stage $\min\{T,t+L\}$. The stage functions, initial-policy evaluation procedures, and conditional samplers needed for these future-stage queries must already be accessible; they need not be stored as explicit scenario trees. Sequential availability of the observations does not imply sequential availability of the model itself.

Finally, combining the in-expectation guarantees from Theorem \ref{thm:MDSA-1} with the Algorithm \ref{alg:async}, we have the following convergence guarantee. 
\begin{corollary}\label{cor:overall}
    Assume that the assumptions in Theorem \ref{thm:MDSA-1} hold, then the implementation in Algorithm \ref{alg:async} achieves the following
\begin{equation*}
    \overline{\E}[\gapu(\overline{X})]\leq \frac{\sqrt{2}\tilde{D}\tilde{L}}{\sqrt{L}}+bD,
\end{equation*}
where $\overline{X}$ is defined as $\overline{X}_t(\xi_{1:t},\zeta_{S(t)}) = X_t(\xi_{1:t})$ when $\bz = \zeta$. In addition, with $L = \epsilon^{-2}$, the oracle complexity is $O(T\min(2^{1/\epsilon^2},\epsilon^{-2T}))$ while the additional evaluation storage is $O(\epsilon^{-4}B)$, with the exclusions and sequential-output convention in Section \ref{sec:eval_oracle}. 
\end{corollary}

In particular, when the sampling is unbiased, i.e. $b = 0$, further assuming $\tilde{D} \leq D_0\sqrt{T}$, $\tilde{L} \leq L_0\sqrt{T}$ where $D_0 = O(1)$ and $L_0 = O(1)$, the suboptimality $\overline{\E}[\gapu(\overline{X})]\leq \sqrt{2}D_0L_0T\cdot \epsilon = O(\epsilon T)$. Thus, to find an $O(\epsilon T)$ suboptimal solution, in the regime where $\epsilon = \Omega(1/\sqrt{T})$, the oracle complexity is $O(T2^{1/\epsilon^2})$, and if $\epsilon = O(1/\sqrt{T})$, the oracle complexity is $O(T\epsilon^{-2T})$.

Similar results hold for \eqref{eq:MD2}. The accelerated extension is discussed in Appendix \ref{appendix:acc}.

\section{Conclusion}
In this work, we study the unconstrained \eqref{eq:obj_unconstrained-online} and the saddle point \eqref{eq:minmax-online} variant of the multi-stage stochastic programming problems. We show the convergence of mirror descent stochastic approximation with stochastic conditional gradient oracles. An accelerated extension is provided in Appendix \ref{appendix:acc}. 

To further improve complexities, we propose an online framework in which only a single trajectory of decisions is needed. This turns the infinite-dimensional optimization problem in the policy space into a tractable problem with finite-dimensional outputs. Moreover, we show that the decomposability of the MDSA updates enables efficient online implementation, leading to algorithms whose oracle complexities scale linearly in $T$ at fixed per-stage accuracy, under uniformly bounded stagewise problem parameters and sufficiently small sampling bias. This statement concerns an $O(T\epsilon)$ total expected gap with $\epsilon=\Theta(1)$; the dependence on inverse accuracy remains exponential.

The online updating mechanism we propose highlights great potential in coupling the algorithms with the underlying stochastic process to achieve improved efficiency. More broadly, our results suggest that, even in high- or infinite-dimensional settings, solutions for appropriately structured subsets of decision variables can be computed efficiently. We believe this is a promising direction for future research.

\section*{Acknowledgments}
This research was partially supported by the ONR grant N00014-24-1-2470.
\par\bigskip
\begin{APPENDICES}
\renewcommand{\theHsection}{appendix.\Alph{section}}
\renewcommand{\theHsubsection}{\theHsection.\arabic{subsection}}
\renewcommand{\theHsubsubsection}{\theHsubsection.\arabic{subsubsection}}
\SingleSpacedXI 

\section{Discussions regarding the setup}\label{app:setup}

\subsection{Comparison with classical MSSP formulation}\label{sec:lagrangian}

Recall that in the classical MSSP formulation \citep{Shapiro2021Lectures}, consecutive decision variables are coupled through the constraints:
\begin{align}\label{eq:obj_constrained}
     &\inf_{X_t\in \mathcal P_t(\mathcal X_t),~t=1,\ldots,T} \E[\sum_{t=1}^T h_{t}(\mb x_t,\bx_t)],\\
     &s.t.~\mb x_t = X_t(\bx_{1:t}),\quad t=1,2,\ldots,T,\nonumber\\
     &\quad\quad  g_1(\mb x_1(w),\bx_1(w))\leq \mb 0,\quad  w\in \Omega,\nonumber\\
     &\quad \quad g_t(\mb x_{t-1}(w),\mb x_t(w),\bx_t(w))\leq \mb 0,\quad t = 2,\ldots,T,~  w\in \Omega,\nonumber
\end{align}
where $h_t:\R^{n_t}\times \Xi_t\to \R$ and each component of $g_t:\R^{n_{t-1}}\times \R^{n_t}\times \Xi_t\to \R^{m_t}$ ($n_0=0$) are convex in $x_t$ and $(x_{t-1},x_t)$ respectively. Under regularity conditions (such as Proposition 3.6 in \citet{Shapiro2021Lectures} for linear constraints), \eqref{eq:obj_constrained} can be reformulated as a saddle point problem \mssp, with $\phi_1(x_1,y_1,\xi_1) = h_1(x_1,\xi_1) + \langle y_1,g_1(x_1,\xi_1)\rangle $ and for $t = 2,\ldots,T$, $\phi_{t}(x_{t-1},x_t,y_t,\xi_t) = h_t(x_t,\xi_t) + \langle y_t,g_t(x_{t-1},x_t,\xi_t)\rangle$, with the caveat that $\mathcal Y_t = \R_{\geq 0}^{m_t}$ is unbounded. 

We work directly with adapted policies. Write
\begin{equation*}
    H(X):=\E\left[\sum_{t=1}^T h_t(\mb x_t,\bx_t)\right],\qquad
    \mathcal L(X,Y):=H(X)+\sum_{t=1}^T\E[\langle\mb y_t,g_t(\mb x_{t-1},\mb x_t,\bx_t)\rangle],
\end{equation*}
where $\mb x_t=X_t(\bx_{1:t})$, $\mb y_t=Y_t(\bx_{1:t})$, and the stage-$1$ constraint does not depend on $\mb x_0$. Assume that the displayed expectations are finite. Let $X^*$ be an optimal feasible adapted policy for \eqref{eq:obj_constrained}. For a dual cap $C>0$, let $\mathcal Y_t=[0,C]^{m_t}$ and define the expected total positive constraint residual
\begin{equation*}
    V(X):=\sum_{t=1}^T\sum_{i=1}^{m_t}\E\left[\left[g_{t,i}(\mb x_{t-1},\mb x_t,\bx_t)\right]_+\right],
    \qquad [a]_+:=\max(a,0).
\end{equation*}
The corrected gap for this compact-dual problem is
\begin{equation*}
    \gaps(X,Y)=\sup_{Y'}\mathcal L(X,Y')-\inf_{X'}\mathcal L(X',Y),
\end{equation*}
where $X'_t\in\mathcal P_t(\mathcal X_t)$ and $Y'_t\in\mathcal P_t([0,C]^{m_t})$. Since each constraint residual is measurable with respect to $\mathcal F_t$, the choice $\mb y'_{t,i}=C\mathbf 1_{\{g_{t,i}(\mb x_{t-1},\mb x_t,\bx_t)>0\}}$ is an admissible adapted policy. Consequently,
\begin{equation*}
    \sup_{Y'}\mathcal L(X,Y')=H(X)+CV(X).
\end{equation*}
Feasibility of $X^*$ and nonnegativity of $Y$ imply
$\inf_{X'}\mathcal L(X',Y)\leq\mathcal L(X^*,Y)\leq H(X^*)$.
Combining these observations gives the general objective-plus-penalty guarantee
\begin{equation}\label{eq:combined-penalty-gap}
    H(X)-H(X^*)+CV(X)\leq\gaps(X,Y).
\end{equation}
Thus $\gaps(X,Y)\leq\epsilon$ implies $H(X)-H(X^*)\leq\epsilon$. It does not, by itself, imply $V(X)\leq\epsilon/C$: an infeasible policy may have $H(X)<H(X^*)$.

A separate feasibility guarantee follows under an additional multiplier condition. Assume that the original constrained problem admits an optimal adapted nonnegative multiplier policy $Y^*$ satisfying
\begin{equation*}
    \|\mb y_t^*\|_\infty\leq C_*<C\quad\text{almost surely for every }t,
    \qquad
    H(X^*)=\mathcal L(X^*,Y^*)\leq\mathcal L(X,Y^*)
\end{equation*}
for every adapted primal policy $X$. This is a policy-level saddle relation for the original problem; it does not assert scenario-wise optimality of $X^*$. It implies
\begin{equation*}
    H(X)-H(X^*)\geq
    -\sum_{t=1}^T\E[\langle\mb y_t^*,g_t(\mb x_{t-1},\mb x_t,\bx_t)\rangle]
    \geq-C_*V(X).
\end{equation*}
Together with \eqref{eq:combined-penalty-gap}, this yields
\begin{equation}\label{eq:feasibility-multiplier-gap}
    (C-C_*)V(X)\leq\gaps(X,Y).
\end{equation}
In particular, $\gaps(X,Y)\leq\epsilon$ implies $V(X)\leq\epsilon/(C-C_*)$. For randomized policies, the same inequalities hold for each fixed realization of the internal randomness and hence after taking its expectation.

The complexity results in the main text apply to the compact-dual problem. When using \eqref{eq:feasibility-multiplier-gap} to obtain a prescribed feasibility tolerance for the original constrained problem, the dependence of the dual diameter, gradient bounds, and Bregman radius on $C$ must also be included. We do not claim bounds uniform in $C$, or a uniformly bounded multiplier policy for general constrained MSSP.

\subsection{Comparison with metrics for the first stage decision variable}\label{sec:compare-stage-one} 

Since the policy spaces are infinite-dimensional, it is impossible to even write down all solutions. One way to circumvent this is to look at first stage decision only: since $\Xi_1 = \{\xi_1\}$, the policy for stage $1$ is determined by $X_1(\xi_1)=x_1\in \R^{n_1}$, a finite dimensional vector. To measure how good a first stage decision is, past works (as represented by the dynamic stochastic approximation algorithm) have looked at the suboptimality w.r.t. the objective $f_1(x_1,\xi_1) + V_1(x_1)$, where for \eqref{eq:obj_unconstrained-online}
\begin{equation*}
    V_1(x_1)=\inf_{X_t\in \mathcal P_t(\mathcal X_t),t=2,\ldots,T}\E[\sum_{t=2}^Tf_t(X_{t-1}(\bx_{1:(t-1)}),X_t(\bx_{1:t}),\bx_t)], \quad X_1(\xi_1) = x_1, 
\end{equation*}
is the cost-to-go function. Our metric $\gapu$ for all stages $X_{1:T}$ is stronger, in the sense that given any policy $X$ for \eqref{eq:obj_unconstrained-online} with $X_1(\xi_1) = x_1$, we have  
\begin{equation*}
    f_1(x_1,\xi_1) + V_1(x_1) \leq f_1(x_1,\xi_1) + \E[\sum_{t=2}^Tf_t(X_{t-1}(\bx_{1:(t-1)}),X_t(\bx_{1:t}),\bx_t)] = \E[f(X(\bx),\bx)].
\end{equation*}
Thus, $f_1(x_1,\xi_1) + V_1(x_1) - \E[f(X^*(\bx),\bx)]\leq \gapu(X)$, or, the suboptimality guarantee for $X$ automatically becomes the guarantee for $X_1(\xi_1)$. Similar relation holds for \eqref{eq:minmax-online}.

\section{Accelerated MDSA}\label{appendix:acc}

\subsection{Algorithm and assumptions}\label{sec:acc_md}
With additional assumptions such as smoothness and/or strong convexity of the objective functions, classical convex optimization can be accelerated, with rates of convergence better than $1/\sqrt{l}$. More precisely, assume that $\R^{n}$ is equipped with the Euclidean inner product $\langle\cdot,\cdot\rangle$ and the \textit{induced norm} $\|\cdot\|$. For the problem \eqref{eq:obj_unconstrained-online}, consider the following updates.

\begin{algorithm}
\caption{Hypothetical accelerated mirror descent update for \eqref{eq:obj_unconstrained-online}}\label{alg:hypo-ms-acc}
\begin{algorithmic}
\Require Realizations of the random variables $\{\bz_k(w_{\internal})= \zeta_k,~k\in \mathcal K\}$; initialization $X^{(0)}_t\in \mathcal P_t(\mathcal X^o_t)$ for all $t$; parameters $\gamma\geq 0$, $\mu \geq0$, $L_2>0$, $\theta\in [0,1]$; oracles $\overline{G}^{(l)}_t:\Xi_{1:t}\times \Gamma_{S'(t,l)}\to \R^{n_t}$ for all $t$ and $l=0,\ldots,L$. 
\Ensure $\overline{X}_t(\cdot,\zeta_{S(t)}) = X_{t+}^{(L)}\in \mathcal P_t(\mathcal X_t)$ for all $t$.
\State Set $\alpha_0 = 1$, $A_0 = 1$. 
\For{$l=0,1,\ldots,L$}
\State Find $\alpha_{l+1}$ satisfying $(1+\gamma)L_2 + (1-\theta)\mu A_l = \frac{(1+\gamma)L_2\alpha_{l+1}^2}{A_{l}+\alpha_{l+1}}$.
\State Set $A_{l+1} = A_l + \alpha_{l+1}$ and $\tau_l = \alpha_{l+1}/A_{l+1}$. 
\For{$\xi\in \Xi$, $t=1,\ldots,T$}
\State Set $G_t^{(l)}(\xi_{1:t}) = \overline{G}_t^{(l)}(\xi_{1:t},\zeta_{S'(t,l)})$.
\State Compute $\overline{X}_t^{(l+1)}(\xi_{1:t},\zeta_{S'(t,l)}) = X_{t}^{(l+1)}(\xi_{1:t})$ using
\begin{align}\label{eq:update_acc_inexact}
    & X_{t+}^{(l)}(\xi_{1:t}) = \argmin_{x_t\in \mathcal X_t}\langle G_{t}^{(l)}(\xi_{1:t}),x_t-X_{t}^{(l)}(\xi_{1:t})\rangle + \frac{(1+\gamma)L_2}{2}\|x_t-X_{t}^{(l)}(\xi_{1:t})\|^2,\nonumber\\
    & X_{t-}^{(l)}(\xi_{1:t}) = \argmin_{x_t\in \mathcal X_t} (1+\gamma) L_2D_{v_t}(x_t,X_{t}^{(0)}(\xi_{1:t})) \nonumber\\
    &\qquad\qquad\qquad + \sum_{l'=0}^{l} \alpha_{l'}(\langle G_{t}^{(l')}(\xi_{1:t}),x_t-X_{t}^{(l')}(\xi_{1:t})\rangle + \frac{(1-\theta)\mu}{2}\|x_t-X_{t}^{(l')}(\xi_{1:t})\|^2),\nonumber\\
    &X_{t}^{(l+1)}(\xi_{1:t}) = \tau_l X_{t-}^{(l)}(\xi_{1:t}) + (1-\tau_l)X_{t+}^{(l)}(\xi_{1:t}).
\end{align}
\EndFor
\EndFor
\end{algorithmic}
\end{algorithm}

\begin{lemma}\label{lm:acc-inexact-multi-stage}
Under Assumptions \ref{assu:convex} and \ref{assumption:measurable}, and assume that $f(\cdot,\xi)$ satisfies the following condition for all $\xi\in \Xi$
\begin{equation}\label{eq:condition-convex-smooth-f}
    \frac{\mu}{2}\|x - x'\|^2\leq f(x',\xi) - f(x,\xi) - \langle \nabla f(x,\xi),x'-x\rangle \leq \frac{L_2}{2}\|x - x'\|^2,\quad \forall x,x'\in  \mathcal X,
\end{equation}
 then $X_t^{(l)}, X_{t\pm}^{(l)}\in \mathcal P_t(\mathcal X_t)$ for all $t,l$. In addition, with $\gamma = 1$ and $\theta = 1/2$, the following holds
\begin{equation*}
\gapu(X_{+}^{(l)}) \leq \E[A_l^{-1}(2L_2D_v(\mb x^*,\mb x^{(0)}) + \sum_{l'=0}^{l}\overline{\Delta}^{(l')})],
\end{equation*}
where for $t = 1,\ldots,T$ and $l = 0,1\ldots,L$, $\Delta^{(l)}_t = \mb g_t^{(l)} - \E[\frac{\partial}{\partial x_t} f(\mb x^{(l)},\bx)|\mathcal F_t]$, and denoting $A_{-1} = 0$,
\begin{equation*}
    \overline{\Delta}^{(l')} =\begin{cases}A_{l'}\frac{\|\Delta^{(l')}\|^2}{2 L_2} +  \langle \Delta^{(l')},\alpha_{l'}(\mb x^*-\mb x^{(l')}) + A_{l'-1}(\mb x_+^{(l'-1)}-\mb x^{(l')})\rangle&\mu=0\\
(\frac{\alpha_{l'}}{\mu} + \frac{A_{l'}}{2\mu}+\frac{A_{l'}}{2L_2})\|\Delta^{(l')}\|^2 &\mu >0
\end{cases}.
\end{equation*}

\end{lemma}
Notice that for $\mu = 0$, $\theta$ does not affect the trajectory of \eqref{eq:update_acc_inexact}, and so the bound in Lemma \ref{lm:acc-inexact-multi-stage} holds under all $\theta$. The proof of Lemma \ref{lm:acc-inexact-multi-stage} is deferred to Section \ref{app:acc-proofs}.

\textbf{Remark.} Our updates in Algorithm \ref{alg:hypo-ms-acc} are slightly different from the accelerated MDSA proposed in \citet{lan_optimal_2012} (for instance, we do not use restart when $\mu> 0$). However, the deterministic $O(1/l^2)$ and linear convergence terms for $\mu= 0$ and $\mu>0$, respectively, as will be shown in Theorem \ref{thm:acc-mu}, match the performance in \citet{lan_optimal_2012}.

\subsection{Stochastic convergence guarantees}\label{app:acc-guarantees}
For the accelerated mirror descent \eqref{eq:update_acc_inexact}, a similar argument as above applied to Lemma \ref{lm:acc-inexact-multi-stage} gives the following results.
\begin{theorem}\label{thm:acc-mu}
    Assume that Assumptions \ref{assu:convex} and \ref{assumption:measurable} and the conditions in Lemma \ref{lm:acc-inexact-multi-stage} hold. Further assume that $\|\frac{\partial}{\partial x}f(x,\xi)\|_*\leq \tilde{L}$ for all $x\in \mathcal X,\xi \in \Xi$, $\E[D_{v}(\mb x^*,\mb x^{(0)})]\leq \tilde{D}^2$, and $\overline{G}_t^{(l)}$ is chosen according to Lemma \ref{lm:MDSA-1}. Denote $b^2 = \sum_{t=1}^T b_t^2$, $\sigma^2 = \sum_{t=1}^T \sigma_t^2$, and set $\gamma = 1$, $\theta=1/2$. 

    If $\mu =0$ and $\|x-x'\|\leq D$ for all $x,x'\in \mathcal X$. Then we have 
\begin{equation}\label{eq:acc-1}
    \E_{\internal}[\gapu(\overline{X})]\leq \frac{8L_2 \tilde{D}^2}{(L+1)(L+2)} + \frac{2}{3}(L+3) (\frac{b^2+\sigma^2}{2L_2}+ bD).
\end{equation}
If $\mu>0$, then for $\rho =  (1+\frac{1}{4}\sqrt{\frac{\mu}{L_2}})^{-2}\leq 1-\frac{3}{16}\sqrt{\frac{\mu}{L_2}}$, 
\begin{equation}\label{eq:acc-2}
    \E_{\internal}[\gapu(\overline{X})]\leq \rho^L\cdot 2L_2 \tilde{D}^2 + \frac{b^2+\sigma^2}{2}\cdot (\frac{3}{\mu} + \frac{1}{L_2})\cdot\frac{1}{1-\rho}.
\end{equation}
\end{theorem}

The proof is provided in Appendix \ref{app:proof:thm:acc-mu}.

The first terms in \eqref{eq:acc-1} and \eqref{eq:acc-2} are deterministic optimization terms: they decay as $O(L^{-2})$ when $\mu=0$ and as $\exp(-\Omega(L\sqrt{\mu/L_2}))$ when $\mu>0$. These improve on the corresponding deterministic gradient-descent rates $O(L^{-1})$ and $\exp(-\Omega(L\mu/L_2))$. The remaining terms account for sampling noise and bias. At fixed nonzero $b$ or $\sigma$, \eqref{eq:acc-1} contains a contribution growing with $L$, whereas \eqref{eq:acc-2} contains a nonvanishing error floor; neither bound guarantees convergence to zero. Thus, the accelerated rates describe the total error only while those additional contributions are sufficiently small relative to the deterministic term.

\subsection{Online implementation and complexity comparison}\label{app:acc-implementation}
For Algorithm \ref{alg:hypo-ms-acc} for the accelerated MDSA, we can first rewrite the update \eqref{eq:update_acc_inexact} as the following,
\begin{align}\label{eq:tree_acc_inexact}
    & X_{t+}^{(l)}(\xi_{1:t}) = \argmin_{x_t\in \mathcal X_t}\langle G_{t}^{(l)}(\xi_{1:t}),x_t\rangle + \frac{(1+\gamma)L_2}{2}\|x_t-X_{t}^{(l)}(\xi_{1:t})\|^2\nonumber\\
    & G_{t+}^{(l)}(\xi_{1:t}) = G_{t+}^{(l-1)}(\xi_{1:t}) + \alpha_l (G_t^{(l)}(\xi_{1:t}) - (1-\theta)\mu X_t^{(l)}(\xi_{1:t}))\nonumber\\
    & X_{t-}^{(l)}(\xi_{1:t}) = \argmin_{x_t\in \mathcal X_t}  (1+\gamma)L_2v_t(x_t) + \frac{A_l(1-\theta)\mu}{2}\|x_t\|^2 + \langle G_{t+}^{(l)}(\xi_{1:t}),x_t\rangle\nonumber\\
    &X_{t}^{(l+1)}(\xi_{1:t}) = \tau_l X_{t-}^{(l)}(\xi_{1:t}) + (1-\tau_l)X_{t+}^{(l)}(\xi_{1:t}).
\end{align}
Above, $G_{t+}^{(l)} $ is the cumulative gradient, 
\begin{equation}\label{eq:cum_grad}
    G_{t+}^{(l)} = \sum_{l'=0}^{l} \alpha_{l'}( G_{t}^{(l')}-(1-\theta)\mu X_{t}^{(l')})- (1+\gamma)L_2 \nabla v(X^{(0)}_{t}).
\end{equation}
For the procedure Algorithm \ref{alg:update-procedure}, instead of evaluating (and keeping in memory) $X_t^{(l)}$, we evaluate and store $(X_{t}^{(l)}, G_{t+}^{(l-1)})$.

Table \ref{tab:acceleration-comparison} compares the proved MDSA bounds for unbiased conditional-gradient estimates ($b=0$), allowing nonzero variance. Write $v:=\sigma^2/T$, $d_0^2:=\widetilde D^2/T$, and $a:=8L_2d_0^2$. Dividing \eqref{eq:acc-1} by $T$ gives
\begin{equation*}
 \frac{\E_{\internal}[\gapu(\overline X)]}{T}
 \leq \frac{a}{(L+1)(L+2)}+\frac{(L+3)v}{3L_2}.
\end{equation*}
Thus, for $L_c:=\max\{1,\lceil\sqrt{2a/\epsilon}\rceil\}$, a sufficient condition for per-stage error at most $\epsilon$ is
\begin{equation*}
 \frac{(L_c+3)v}{3L_2}\leq\frac{\epsilon}{2}.
\end{equation*}
For fixed positive $a$ and uniformly controlled $L_2$, this permits nonzero $v=O(\epsilon^{3/2})$. In the strongly convex case, define
\begin{equation*}
 a_s:=2L_2d_0^2,\qquad
 V_s:=\frac{v}{2}\left(\frac{3}{\mu}+\frac{1}{L_2}\right)\frac{1}{1-\rho},\qquad
 L_s:=\max\left\{1,\left\lceil\frac{\log(1+2a_s/\epsilon)}{-\log\rho}\right\rceil\right\}.
\end{equation*}
Then \eqref{eq:acc-2} gives per-stage error at most $\epsilon$ whenever $V_s\leq\epsilon/2$. Here $L_s=O(\sqrt\kappa\log(1/\epsilon))$ for $\kappa=L_2/\mu$ and uniformly controlled initial-error constants. These conditions identify finite-accuracy regimes where acceleration reduces the required iteration count despite nonzero sampling variance. At fixed variance, the convex bound restricts the useful iteration range, and the strongly convex bound has a variance floor.

Define $\mathcal R(q):=T\min\{2^q,q^T\}$. Substituting the iteration counts into this recursive evaluation envelope gives the last column of Table \ref{tab:acceleration-comparison}.
\begin{table}[htbp]
\centering
\footnotesize
\setlength{\tabcolsep}{4pt}
\renewcommand{\arraystretch}{1.45}
\begin{tabular}{l|l|l|l}
\hline
Setting & Method & Sufficient iterations & Recursive oracle bound \\
\hline
Smooth convex & MDSA & $L_m=O(\epsilon^{-2})$ & $O(\mathcal R(L_m))$ \\
Smooth convex & Accelerated MDSA & $L_c=O(\epsilon^{-1/2})$ & $O(\mathcal R(L_c))$ \\
\hline
Smooth, strongly convex & MDSA & $L_m=O(\epsilon^{-2})$ & $O(\mathcal R(L_m))$ \\
Smooth, strongly convex & Accelerated MDSA & $L_s=O(\sqrt\kappa\log(1/\epsilon))$ & $O(\mathcal R(L_s))$ \\
\hline
\end{tabular}
\caption{Unbiased sampling ($b=0$), with nonzero variance allowed, for total expected error at most $T\epsilon$, $0<\epsilon<1$. The MDSA rows use the general stochastic guarantee in Theorem \ref{thm:MDSA-1}: $L_m=\max\{1,\lceil2\widetilde D^2\widetilde L^2/(T^2\epsilon^2)\rceil\}$. This bound does not exploit smoothness or strong convexity and is not asserted to be optimal. Accelerated rows use Theorem \ref{thm:acc-mu}, with $L_c,L_s$ and their respective variance conditions specified above. Parameters other than the displayed condition number are uniformly controlled. The recursive bounds count conditional-gradient and proximal operations as primitive calls under the access model of Section \ref{sec:eval_oracle}; they are not arithmetic-cost bounds. Acceleration improves these sufficient bounds when its variance condition holds and its required iteration count is smaller than $L_m$. No improvement at arbitrarily small target error with fixed nonzero variance is asserted.}
\label{tab:acceleration-comparison}
\end{table}

\subsection{Proofs}\label{app:acc-proofs}
Recall that $\mb x_t^{(l)}(w) = X_t^{(l)}(\bx_{1:t}(w))$, $\mb x_{t+}^{(l)}(w) = X_{t+}^{(l)}(\bx_{1:t}(w))$, $\mb x_{t-}^{(l)}(w) = X_{t-}^{(l)}(\bx_{1:t}(w))$, and $\mb g_t^{(l)}(w) = G_t^{(l)}(\bx_{1:t}(w))$. To abbreviate notation, we use $\mb x^{(l)} = (\mb x^{(l)}_1,\ldots,\mb x_T^{(l)})$ and similarly for $\mb x^{(l)}_{\pm}$ and $\mb g^{(l)}$. We denote $\mathcal X = \prod_{t=1}^T \mathcal X_t$. In addition, we define the following mapping $\psi_l:\{\Omega\to \R^{\sum_{t=1}^T n_t}\}\to \{\Omega\to \R\}$,
\begin{align*}
    \psi_{l}(\mb x)(w)&:= (1+\gamma)L_2v(\mb x(w)
    ) + \sum_{l'=0}^{l} \alpha_{l'}(f(\mb x^{(l')}(w),\bx(w)) \\
    &\quad +\langle \mb g^{(l')}(w),\mb x(w)-\mb x^{(l')}(w)\rangle + \frac{(1-\theta)\mu}{2}\|\mb x(w)-\mb x^{(l')}(w)\|^2).
\end{align*}

Below, for two random variables $W_1,W_2:\Omega\to \R$, by $W_1\geq W_2$, we mean $W_1(w)\geq W_2(w)$ for all $w\in \Omega$. 

\begin{lemma}\label{lm:general_inexact}
Under the assumptions in Lemma \ref{lm:acc-inexact-multi-stage}, for any $
\mb x,\mb x':\Omega\to \mathcal X$, 
\begin{equation}\label{eq:condition-convex-smooth}
    \frac{\mu}{2}\E[\|\mb x - \mb x'\|^2]\leq \E[f(\mb x',\bx) - f(\mb x,\bx) - \langle \nabla f(\mb x,\bx),\mb x'-\mb x\rangle] \leq \frac{L_2}{2}\E[\|\mb x -\mb x'\|^2]. 
\end{equation}
Assume that $v(\mb x^{(0)}(w)) = 0$ and $\nabla v(\mb x^{(0)}(w)) = \mb 0$ for all $w\in\Omega$, then for all $l\geq 0$, we have $A_lf(\mb x_+^{(l)},\bx)\leq \psi_l(\mb x_-^{(l)})  + E_l,$
where $E_l = \sum_{l'=0}^l A_{l'}\delta_{l'} $. 
\begin{equation*}
    \delta_0 =\langle \mb g^{(0)} - \nabla f(\mb x^{(0)},\bx),\mb x^{(0)}-\mb x_+^{(0)}\rangle - \frac{\gamma L_2}{2}\|\mb x_+^{(0)}- \mb x^{(0)}\|^2,
\end{equation*}
and for $l\geq 0$, $\delta_{l+1} = \hat{\delta}_{l+1} + (1-\tau_l) \tilde{\delta}_{l+1}$ where
\begin{equation*}
    \hat{\delta}_{l+1} = \langle \nabla f(\mb x^{(l+1)},\bx)- \mb g^{(l+1)},\mb x_+^{(l+1)}-\mb x^{(l+1)}\rangle-\frac{\gamma L_2}{2}\|\mb x_+^{(l+1)}-\mb x^{(l+1)}\|^2,
\end{equation*}
\begin{equation*}
    \tilde{\delta}_{l+1} = \langle \mb g^{(l+1)} - \nabla f(\mb x^{(l+1)},\bx) ,\mb x_+^{(l)}-\mb x^{(l+1)}\rangle-\frac{\mu}{2}\|\mb x_+^{(l)} - \mb x^{(l+1)}\|^2.
\end{equation*}
\end{lemma}
The proof of Lemma \ref{lm:general_inexact} follows closely \citep{DEVOLDER2013_strong}. However, \citep{DEVOLDER2013_strong} considers the case where the gradient inexactness is upper bounded by a constant $\delta$, independent of the query point and the iteration number. In the proof below, we explicitly track the accumulation of the error at each stage.

\begin{proof}{Proof of Lemma \ref{lm:general_inexact}}
\eqref{eq:condition-convex-smooth} follows from \eqref{eq:condition-convex-smooth-f}.

For $l = 0$, first, notice that since $\alpha_0 = 1$ and $v$ is $1$-strongly convex, we have
\begin{align*}
    \psi_0(\mb x) &=  (1+\gamma)L_2v(\mb x) + f(\mb x^{(0)},\bx) + \langle \mb g^{(0)},\mb x-\mb x^{(0)}\rangle + \frac{(1-\theta)\mu}{2} \|\mb x- \mb x^{(0)}\|^2\\
    &\geq \frac{(1+\gamma)L_2}{2}\|\mb x - \mb x^{(0)}\|^2 + f(\mb x^{(0)},\bx) + \langle \mb g^{(0)},\mb x-\mb x^{(0)}\rangle 
\end{align*}
Thus, we have for any $w\in \Omega$
\begin{align*}
    &\quad \psi_0(\mb x_-^{(0)})(w) = \min_{x\in \mathcal X} \psi_0(x)(w) \\  &\geq \min_{x\in \mathcal X} \frac{(1+\gamma)L_2}{2}\|x - \mb x^{(0)}(w)\|^2 + f(\mb x^{(0)}(w),\bx(w)) + \langle \mb g^{(0)}(w),x-\mb x^{(0)}(w)\rangle \\
    & = \frac{(1+\gamma)L_2}{2}\|\mb x_+^{(0)}(w)- \mb x^{(0)}(w)\|^2 + f(\mb x^{(0)}(w),\bx(w)) + \langle \mb g^{(0)}(w),\mb x_+^{(0)}(w)-\mb x^{(0)}(w)\rangle\\
    & = \frac{L_2}{2}\|\mb x_+^{(0)}(w)- \mb x^{(0)}(w)\|^2 + f(\mb x^{(0)}(w),\bx(w)) \\
    &\quad + \langle \nabla f(\mb x^{(0)}(w),\bx(w)),\mb x_+^{(0)}(w)-\mb x^{(0)}(w)\rangle - \delta_0(w)
\end{align*}
where $\delta_0 =\langle \mb g^{(0)} - \nabla f(\mb x^{(0)},\bx),\mb x^{(0)}-\mb x_+^{(0)}\rangle - \frac{\gamma L_2}{2}\|\mb x_+^{(0)}- \mb x^{(0)}\|^2 $. Assume now that the statement holds for some $l\geq 0$. By the update of $\mb x_-^{(l)}$, we have for any $\mb x = \mb x_{1:T}$ where $\mb x_t\in \mathcal X_t$ is measurable w.r.t. $\mathcal F_t$ for all $t$
\begin{equation*}
    \langle (1+\gamma)L_2\nabla v(\mb x_-^{(l)})+ \sum_{l'=0}^{l}\alpha_{l'}\left[\mb g^{(l')}+(1-\theta)\mu(\mb x_-^{(l)} - \mb x^{(l')})\right],\mb x-\mb x_-^{(l)}\rangle \geq 0. 
\end{equation*}
Then, the strong convexity of $v$ implies that
\begin{align*}
    L_2v(\mb x) &\geq L_2 v(\mb x_-^{(l)}) + \langle  L_2 \nabla v(\mb x_-^{(l)}), \mb x - \mb x_-^{(l)} \rangle + \frac{ L_2}{2}\|\mb x-\mb x_-^{(l)}\|^2\\
    & \geq   L_2 v(\mb x_-^{(l)}) + \frac{ L_2}{2}\|\mb x-\mb x_-^{(l)}\|^2 \\
    &\quad - (1+\gamma)^{-1}\langle \sum_{l'=0}^{l}\alpha_{l'}\left[\mb g^{(l')}+(1-\theta)\mu(\mb x_-^{(l)} - \mb x^{(l')})\right],\mb x-\mb x_-^{(l)}\rangle
\end{align*}
Thus we have
\begin{align*}
    \psi_{l+1}(\mb x) &\geq (1+\gamma)L_2 v(\mb x_-^{(l)}) + \frac{(1+\gamma)L_2}{2}\|\mb x-\mb x_-^{(l)}\|^2 \\
    &\quad - \langle \sum_{l'=0}^{l}\alpha_{l'}\left[\mb g^{(l')}+(1-\theta)\mu(\mb x_-^{(l)} - \mb x^{(l')})\right],\mb x-\mb x_-^{(l)}\rangle \\
    &\quad+ \sum_{l'=0}^{l}\alpha_{l'}(f(\mb x^{(l')},\bx) + \langle \mb g^{(l')},\mb x- \mb x^{(l')}\rangle+\frac{(1-\theta)\mu}{2} \|\mb x- \mb x^{(l')}\|^2)\\
    &\quad + \alpha_{l+1}(f(\mb x^{(l+1)},\bx) +\langle \mb g^{(l+1)},\mb x-\mb x^{(l+1)}\rangle + \frac{(1-\theta)\mu}{2}\|\mb x - \mb x^{(l+1)}\|^2)
\end{align*}
Using 
\begin{equation*}
    \langle \mb x_-^{(l)} - \mb x^{(l')} , \mb x_-^{(l)} - \mb x\rangle  = \frac{1}{2}\|\mb x_-^{(l)} - \mb x^{(l')}\|^2 + \frac{1}{2}\|\mb x_-^{(l)} - \mb x\|^2 - \frac{1}{2}\|\mb x - \mb x^{(l')}\|^2,
\end{equation*}
we have
\begin{align*}
    \psi_{l+1}(\mb x) &\geq (1+\gamma)L_2 v(\mb x_-^{(l)}) +  \frac{(1+\gamma)L_2+A_l(1-\theta)\mu}{2}\|\mb x-\mb x_-^{(l)}\|^2 \\
    &\quad+\sum_{l'=0}^{l}\alpha_{l'}(f(\mb x^{(l')},\bx) + \langle \mb g^{(l')},\mb x_-^{(l)}- \mb x^{(l')}\rangle+\frac{(1-\theta)\mu}{2} \|\mb x_-^{(l)}- \mb x^{(l')}\|^2)\\
    &\quad + \alpha_{l+1}(f(\mb x^{(l+1)},\bx) +\langle \mb g^{(l+1)},\mb x-\mb x^{(l+1)}\rangle + \frac{(1-\theta)\mu}{2}\|\mb x - \mb x^{(l+1)}\|^2)\\
    & = \psi_l(\mb x_-^{(l)}) + \frac{(1+\gamma)L_2+A_l(1-\theta)\mu}{2}\|\mb x-\mb x_-^{(l)}\|^2\nonumber\\
    &\quad +\alpha_{l+1}(f(\mb x^{(l+1)},\bx) +\langle \mb g^{(l+1)},\mb x-\mb x^{(l+1)}\rangle + \frac{(1-\theta)\mu}{2}\|\mb x - \mb x^{(l+1)}\|^2).
\end{align*}
By induction hypothesis, 
\begin{equation*}
    A_l^{-1}(\psi_l(\mb x_-^{(l)}) +E_l) \geq f(\mb x_+^{(l)},\bx) \geq f(\mb x^{(l+1)},\bx) + \langle \mb g^{(l+1)},\mb x_+^{(l)}-\mb x^{(l+1)}\rangle  - \tilde{\delta}_{l+1}
\end{equation*}
where
\begin{equation*}
    \tilde{\delta}_{l+1} = \langle \mb g^{(l+1)} - \nabla f(\mb x^{(l+1)},\bx) ,\mb x_+^{(l)}-\mb x^{(l+1)}\rangle-\frac{\mu}{2}\|\mb x_+^{(l)} - \mb x^{(l+1)}\|^2.
\end{equation*}
Thus we have
\begin{align*}
    \psi_{l+1}(\mb x) &\geq  A_l(f(\mb x^{(l+1)},\bx) + \langle \mb g^{(l+1)},\mb x_+^{(l)}-\mb x^{(l+1)}\rangle - \tilde{\delta}_{l+1}) - E_l \\
    &\quad + \frac{(1+\gamma)L_2+A_l(1-\theta)\mu}{2}\|\mb x-\mb x_-^{(l)}\|^2\\
    &\quad +\alpha_{l+1}(f(\mb x^{(l+1)},\bx) +\langle \mb g^{(l+1)},\mb x-\mb x^{(l+1)}\rangle + \frac{(1-\theta)\mu}{2}\|\mb x - \mb x^{(l+1)}\|^2)\\
    &  \geq A_{l+1}f(\mb x^{(l+1)},\bx) + \alpha_{l+1}\langle \mb g^{(l+1)},\mb x-\mb x_-^{(l)}\rangle  - A_l\tilde{\delta}_{l+1} - E_l \\
    &\quad + \frac{(1+\gamma)L_2+A_l(1-\theta)\mu}{2}\|\mb x-\mb x_-^{(l)}\|^2
\end{align*}
where the last equality is by noticing that
\begin{equation*}
    A_l(\mb x_+^{(l)}-\mb x^{(l+1)}) + \alpha_{l+1}(\mb x-\mb x^{(l+1)}) = \alpha_{l+1}(\mb x-\mb x_-^{(l)}).
\end{equation*}
Therefore, we have
\begin{align*}
    \psi_{l+1}(\mb x_-^{(l+1)})(w) \geq  & A_{l+1}f(\mb x^{(l+1)},\bx) (w)  - A_l\tilde{\delta}_{l+1}(w) - E_l(w) \\
    &+ A_{l+1}\cdot \min_{x\in \mathcal X}  \tau_{l}\langle \mb g^{(l+1)}(w),x-\mb x_-^{(l)}(w)\rangle+ \frac{(1+\gamma)L_2\tau_l^2}{2}\|x-\mb x_-^{(l)}(w)\|^2.
\end{align*}
Notice that defining $y = \tau_l x + (1-\tau_l)\mb x_+^{(l)}(w)$, we get $y-\mb x^{(l+1)}(w) = \tau_l(x-\mb x_-^{(l)}(w))$ and defining $\mathcal X' = \tau_l \mathcal X+ (1-\tau_l)\mb x_+^{(l)}(w) \subset \mathcal X$, we have
\begin{align*}
      &\quad \min_{x\in \mathcal X}  \tau_{l}\langle \mb g^{(l+1)}(w),x-\mb x_-^{(l)}(w)\rangle+ \frac{(1+\gamma)L_2\tau_l^2}{2}\|x-\mb x_-^{(l)}(w)\|^2 \\
    & = \min_{y\in \mathcal X'}  \langle \mb g^{(l+1)}(w),y-\mb x^{(l+1)}(w)\rangle+ \frac{(1+\gamma)L_2}{2}\|y-\mb x^{(l+1)}(w)\|^2\\
    &\geq \min_{y\in \mathcal X}   \langle \mb g^{(l+1)}(w),y-\mb x^{(l+1)}(w)\rangle+ \frac{(1+\gamma)L_2}{2}\|y-\mb x^{(l+1)}(w)\|^2.
\end{align*}
Thus, we have
\begin{align*}
    &\quad \psi_{l+1}(\mb x_-^{(l+1)})(w) +E_l(w)\\ &\geq   A_{l+1}\min_{y\in \mathcal X}  (\langle \mb g^{(l+1)}(w),y-\mb x^{(l+1)}(w)\rangle\\
    &\quad + \frac{(1+\gamma)L_2}{2}\|y-\mb x^{(l+1)}(w)\|^2 +f(\mb x^{(l+1)},\bx)  )(w) - A_l\tilde{\delta}_{l+1}(w) \\
    & = A_{l+1} (\langle \mb g^{(l+1)}(w),\mb x_+^{(l+1)}(w)-\mb x^{(l+1)}(w)\rangle\\
    &\quad + \frac{(1+\gamma)L_2}{2}\|\mb x_+^{(l+1)}(w)-\mb x^{(l+1)}(w)\|^2 +f(\mb x^{(l+1)},\bx)  )(w) - A_l\tilde{\delta}_{l+1} \\
    & \geq A_{l+1}f(\mb x_+^{(l+1)},\bx)(w) - A_l\tilde{\delta}_{l+1}(w) -A_{l+1}\hat{\delta}_{l+1}(w)
\end{align*}
where 
\begin{equation*}
    \hat{\delta}_{l+1} = \langle \nabla f(\mb x^{(l+1)},\bx)- \mb g^{(l+1)},\mb x_+^{(l+1)}-\mb x^{(l+1)}\rangle-\frac{\gamma L_2}{2}\|\mb x_+^{(l+1)}-\mb x^{(l+1)}\|^2.
\end{equation*}
Taking $\delta_{l+1} = \hat{\delta}_{l+1} + (1-\tau_l) \tilde{\delta}_{l+1}$
proves the result for $l+1$. 
\end{proof}

\begin{lemma}\label{lm:ineact-convergence}
Under the assumptions of Lemma \ref{lm:general_inexact}, for any $l\geq 0$,  
\begin{align*}
    f(\mb x_+^{(l)},\bx)
    &\leq f(\mb x^*,\bx) + (1+\gamma)L_2A_l^{-1}v(\mb x^*) +  A_l^{-1}E_l\\
    &\quad + A_l^{-1}\sum_{l'=0}^l \alpha_{l'}(\langle \mb g^{(l')}-\nabla f(\mb x^{(l')},\bx),\mb x^*-\mb x^{(l')}\rangle-\frac{\theta \mu}{2}\|\mb x^* - \mb x^{(l')}\|^2) .
\end{align*}
where $E_l = \sum_{l'=0}^{l} A_{l'}\delta_{l'}$ is as defined in Lemma \ref{lm:general_inexact}. 

\end{lemma}

\begin{proof}{Proof of Lemma \ref{lm:ineact-convergence}}
Notice that from the definition of $\psi_l$ and $\mb x_-^{(l)}$, we have
\begin{align*}
    &\quad \psi_l(\mb x_-^{(l)})-(1+\gamma)L_2v(\mb x^*)  \leq \psi_l(\mb x^*)-(1+\gamma)L_2v(\mb x^*)\\
    & =  \sum_{l'=0}^l \alpha_{l'}(f(\mb x^{(l')},\bx) +\langle \mb g^{(l')},\mb x^*-\mb x^{(l')}\rangle + \frac{(1-\theta)\mu}{2}\|\mb x^* - \mb x^{(l')}\|^2)\\
    &\leq A_lf(\mb x^*,\bx) + \sum_{l'=0}^l \alpha_{l'}(\langle \mb g^{(l')}-\nabla f(\mb x^{(l')}),\mb x^*-\mb x^{(l')}\rangle-\frac{\theta \mu}{2}\|\mb x^* - \mb x^{(l')}\|^2).
\end{align*}
In addition, from Lemma \ref{lm:general_inexact}, we have
\begin{align*}
    f(\mb x_+^{(l)},\bx)&\leq A_l^{-1}(\psi_l(\mb x_-^{(l)}) + E_l)\\
    & \leq f(\mb x^*,\bx) + (1+\gamma)L_2A_l^{-1}v(\mb x^*) +  A_l^{-1}E_l
     \\
     &\quad + A_l^{-1}\sum_{l'=0}^l \alpha_{l'}(\langle \mb g^{(l')}-\nabla f(\mb x^{(l')},\bx),\mb x^*-\mb x^{(l')}\rangle-\frac{\theta \mu}{2}\|\mb x^* - \mb x^{(l')}\|^2) .
\end{align*}
\end{proof}

\begin{lemma}\label{lm:para-seq}
If $\mu>0$, the sequence $A_l$ satisfies that 
    \begin{equation*}
        (1+\frac{1}{2}\sqrt{\frac{(1-\theta)\mu}{(1+\gamma)L_2}})^2A_l\leq A_{l+1},\quad l=0,1,\ldots.
    \end{equation*}
If $\mu =0$, the sequence $\alpha_l$ satisfies that 
\begin{equation*}
    \frac{1}{2}(l+1)\leq \alpha_l \leq l+1,\quad l=0,1,\ldots.
\end{equation*}
\end{lemma}

\begin{proof}{Proof of Lemma \ref{lm:para-seq}}
The result for the case when $\mu>0$ is from Lemma 4 in \citet{DEVOLDER2013_strong}. For $\mu = 0$, notice that $ A_l + \alpha_{l+1}= \alpha_{l+1}^2$.
The claim is true for $l = 0$ since $\alpha_0 =1$. Suppose the statement is true for $l$, thus 
\begin{equation*}
    \frac{(l+1)(l+2)}{4}= \frac{1}{2}\sum_{i=0}^l (i+1)\leq A_l \leq \sum_{i=0}^l (i+1) = \frac{(l+1)(l+2)}{2}.
\end{equation*}
Thus, for $l+1$, notice that $ \alpha_{l+1} = \frac{1}{2} + \sqrt{A_l + \frac{1}{4}}$, we have
\begin{equation*}
    \alpha_{l+1} \geq \frac{1}{2} +\sqrt{A_l} \geq \frac{1}{2} + \frac{l+1}{2} = \frac{l+2}{2}.
\end{equation*}
The upper bound follows from the following
\begin{equation*}
    A_l+\frac{1}{4} \leq  \frac{(l+1)(l+2)}{2} + \frac{1}{4} = \frac{2l^2+6l+5}{4} \leq \frac{4l^2+12l +9}{4} =(l+\frac{3}{2})^2.
\end{equation*}
\end{proof}

\begin{proof}{Proof of Lemma \ref{lm:acc-inexact-multi-stage}}

First, it's easy to check that Lemma \ref{lm:ineact-convergence} still holds when $v_t$ is a random function, defined as $\widehat{v}_t(x_t,w) = v_t(x_t) - v_t(\mb x_t^{(0)}(w)) - \langle \nabla v_t(\mb x_t^{(0)}(w)),x_t-\mb x_t^{(0)}(w)\rangle = D_{v_t}(x_t,\mb x_t^{(0)}(w))$. Then $\widehat{v}_t(\mb x_t^{(0)}(w),w)=0$ and $\nabla \widehat{v}_t(\mb x_t^{(0)}(w),w)=\mb 0$ for all $w\in\Omega$. Also, $D_{v_t}=D_{\widehat{v}_t}$, since adding an affine function does not change the induced Bregman divergence. In the remainder of this proof, write $v(\mb x)(w)=\sum_{t=1}^T\widehat{v}_t(\mb x_t(w),w)$, so $v(\mb x^*)=D_v(\mb x^*,\mb x^{(0)})$ for the original DGF. Following a similar argument as in the proof of Lemma \ref{lm:MD1}, we can show that $X_t^{(l)}, X_{t\pm}^{(l)}\in \mathcal P_t(\mathcal X_t)$ for all $t,l$.

With $\gamma=1$,
\begin{align*}
    \E[\delta_0] &= \E[\langle \mb g^{(0)} - \nabla f(\mb x^{(0)},\bx),\mb x^{(0)}-\mb x_+^{(0)}\rangle- \frac{\gamma L_2}{2}\|\mb x_+^{(0)}- \mb x^{(0)}\|^2]\\
    &=\E[\langle \mb g^{(0)} - \left(\E[\frac{\partial}{\partial x_t}f(\mb x^{(0)},\bx)\mid\mathcal F_t]\right)_{t=1}^T,\mb x^{(0)}-\mb x_+^{(0)}\rangle - \frac{\gamma L_2}{2}\|\mb x_+^{(0)}- \mb x^{(0)}\|^2]\\
    &= \E[\langle \Delta^{(0)},\mb x^{(0)}-\mb x_+^{(0)}\rangle - \frac{\gamma L_2}{2}\|\mb x_+^{(0)}- \mb x^{(0)}\|^2]\leq \frac{\E[\|\Delta^{(0)}\|^2]}{2\gamma L_2} = \frac{\E[\|\Delta^{(0)}\|^2]}{2 L_2},
\end{align*}
where the second $=$ is because $\mb g_t,\mb x_t^{(0)}, \mb x_+^{(0)}$ are all measurable w.r.t. $\mathcal F_t$. Similarly $\E[\hat{\delta}_{l+1}] \leq \frac{\E[\|\Delta^{(l+1)}\|^2]}{2 L_2}$. Thus, for any $\mu\geq 0$, 
\begin{align*}
    \E[f(\mb x_+^{(l)},\bx)-f(\mb x^*,\bx)]&\leq  \E[2L_2A_l^{-1}v(\mb x^*) +  A_l^{-1}\sum_{l'=0}^l A_{l'}\frac{\|\Delta^{(l')}\|^2}{2 L_2}+A_l^{-1} \langle\Delta^{(0)},\mb x^*-\mb x^{(0)}\rangle\nonumber\\
    & \quad + A_l^{-1}\sum_{l'=1}^l \langle \Delta^{(l')},\alpha_{l'}(\mb x^*-\mb x^{(l')}) + A_{l'-1}(\mb x_+^{(l'-1)}-\mb x^{(l')})\rangle ]
\end{align*}
Thus, the result follows from taking for $l'\geq 0$, $A_{-1} = 0$, 
\begin{equation*}
    \overline{\Delta}^{(l')} =A_{l'}\frac{\|\Delta^{(l')}\|^2}{2 L_2} +  \langle \Delta^{(l')},\alpha_{l'}(\mb x^*-\mb x^{(l')}) + A_{l'-1}(\mb x_+^{(l'-1)}-\mb x^{(l')})\rangle.
\end{equation*}

Further assuming that $\mu>0$ and taking $\theta = 1/2$, we get
\begin{equation*}
    \E[\langle \mb g^{(l')}-\nabla f(\mb x^{(l')},\bx),\mb x^*-\mb x^{(l')}\rangle-\frac{\theta \mu}{2}\|\mb x^* - \mb x^{(l')}\|^2] \leq \frac{\E[\|\Delta^{(l')}\|^2]}{2\theta\mu}=\frac{\E[\|\Delta^{(l')}\|^2]}{\mu}.
\end{equation*}
\begin{equation*}
    \E[\tilde{\delta}_{l+1}] = \E[\langle \mb g^{(l+1)} - \nabla f(\mb x^{(l+1)},\bx) ,\mb x_+^{(l)}-\mb x^{(l+1)}\rangle-\frac{\mu}{2}\|\mb x_+^{(l)} - \mb x^{(l+1)}\|^2]\leq \frac{\E[\|\Delta^{(l+1)}\|^2]}{2\mu}.
\end{equation*}
Thus, for all $l\geq 0$, $\E[\delta_l ]\leq \E[\frac{\|\Delta^{(l)}\|^2}{2\mu} + \frac{\|\Delta^{(l)}\|^2}{2L_2}]$. 
Thus, 
\begin{align*}
    \E[f(\mb x_+^{(l)},\bx) - f(\mb x^*,\bx)]\leq2L_2A_l^{-1}\E[v(\mb x^*)]+ A_l^{-1}\sum_{l'=0}^l (\frac{\alpha_{l'}}{\mu} + \frac{A_{l'}}{2\mu}+\frac{A_{l'}}{2L_2})\E[\|\Delta^{(l')}\|^2 ].
\end{align*}
The result follows from taking $\overline{\Delta}^{(l')} =(\frac{\alpha_{l'}}{\mu} + \frac{A_{l'}}{2\mu}+\frac{A_{l'}}{2L_2})\|\Delta^{(l')}\|^2 $. 
\end{proof}

\subsubsection{Proof of Theorem~\ref{thm:acc-mu}}\label{app:proof:thm:acc-mu}
\begin{proof}{Proof.}
    For the first claim, from Lemma \ref{lm:acc-inexact-multi-stage}, the following holds for any $w\in \Omega$
\begin{equation*}
    \overline{\Delta}^{(l',w_{\internal})} =A_{l'}\frac{\|\Delta^{(l',w_{\internal})}\|^2}{2 L_2} +  \langle \Delta^{(l',w_{\internal})},\alpha_{l'}(\mb x^*-\overline{\mb x}^{(l',w_{\internal})}) + A_{l'-1}(\mb x_+^{(l'-1,w_{\internal})}-\mb x^{(l',w_{\internal})})\rangle.
\end{equation*}
Thus, from Lemma \ref{lm:prop-oracle}, 
\begin{equation*}
    \E_{\internal}[\overline{\Delta}^{(l')}] \leq \frac{A_{l'}(b^2+\sigma^2)}{2L_2} + A_{l'}bD. 
\end{equation*}
The result follows from $\frac{(l+1)(l+2)}{4}\leq A_l\leq \frac{(l+1)(l+2)}{2}$ for all $l$ (Lemma \ref{lm:para-seq}) and Lemma \ref{lm:acc-inexact-multi-stage}.

For the second claim, by Lemma \ref{lm:acc-inexact-multi-stage}, with $\gamma = 1,\mu>0,\theta=1/2$,
    \begin{equation*}
        (\frac{\alpha_{l'}}{\mu} + \frac{A_{l'}}{2\mu}+\frac{A_{l'}}{2L_2})^{-1}\E_{\internal}[\overline{\Delta}^{(l')}] =\E_{\internal}[\|\Delta^{(l')}\|^2]\leq b^2+\sigma^2.
    \end{equation*}
By convexity of $s\to (1+s)^{-2}$ on $[0,\infty)$, we have $(1+s)^{-2}\leq 1-3s/4$ for all $0\leq s\leq 1$, and so $\rho =  (1+\frac{1}{4}\sqrt{\frac{\mu}{L_2}})^{-2}\leq 1-\frac{3}{16}\sqrt{\frac{\mu}{L_2}}$. The result follows from Lemma \ref{lm:acc-inexact-multi-stage} and Lemma \ref{lm:para-seq}: we have $A_l\geq  (1+\frac{1}{4}\sqrt{\frac{\mu}{L_2}})^{2l} = \rho^{-l}$ and
\begin{equation*}
    \sum_{l'=0}^l \frac{A_{l'}}{A_l}\leq \sum_{l'=0}^{l} \rho^{l-l'}\leq \sum_{l'=0}^{\infty}  \rho^{l'}\leq \frac{1}{1-\rho}.
\end{equation*}
\end{proof}

\section{Proofs of main-text results}\label{app:main-proofs}

\subsection{Proof of Lemma~\ref{lm:MD1}}\label{app:proof:lm:MD1}
\begin{proof}{Proof.} For the first claim, notice that $G_t^{(l)}$ is measurable by Assumption \ref{assumption:measurable}, and assuming $X_t^{(l)}$ is also measurable, then $(\xi_{1:t},x_t) \to \langle \gamma G_t^{(l)}(\xi_{1:t}),x_t\rangle + D_{v_t}(x_t,X_t^{(l)}(\xi_{1:t}))$ is a normal integrand (since the function $(g,x_t)\to \langle g,x_t\rangle$ is continuous, it is a normal integrand by Example 14.30 in \citet{rockafellarVariational}; taking $g = \gamma G_t^{(l)} -\nabla v_t(X_t^{(l)})$ and using Proposition 14.45 \citep{rockafellarVariational} gives the result). Thus, $X_t^{(l+1)}(\xi_{1:t}) = \argmin_{x_t\in \mathcal X_t} \langle\gamma G_t^{(l)}(\xi_{1:t}),x_t\rangle + D_{v_t}(x_t,X_t^{(l)}(\xi_{1:t}))$ is measurable (Theorem 14.37 \citep{rockafellarVariational}), thereby is in $\mathcal P_t(\mathcal X_t^{(o)})$. The first claim then follows from induction. 

By Lemma \ref{lm:MD-prop-lemma}, for any $x_t\in \mathcal X_t$, and $w\in \Omega$
\begin{equation*}
    \langle \gamma_l \mb g^{(l)}_t(w),\mb x^{(l)}_t(w) - x_t\rangle\leq D_{v_t}(x_t,\mb x^{(l)}_t(w)) - D_{v_t}(x_t,\mb x^{(l+1)}_t(w)) + \frac{1}{2}\gamma_l^2 \|\mb g^{(l)}_t(w)\|_*^2.
\end{equation*}
Thus, for any $\mb x_t':\Omega\to \mathcal X_t$ which is measurable w.r.t. $\mathcal F_t$ and any $w\in \Omega$,
\begin{align*}
    &\quad \E[\langle \gamma_l \frac{\partial}{\partial x_t}f(\mb x^{(l)},\bx),\mb x_t^{(l)}- \mb x'_t \rangle|\mathcal F_t](w)\\
    &=\langle \gamma_l \E[\frac{\partial}{\partial x_t}f(\mb x^{(l)},\bx)|\mathcal F_t],\mb x_t^{(l)}- \mb x'_t \rangle(w)=\langle \gamma_l (\mb g_t^{(l)} -\Delta_t^{(l)}),\mb x_t^{(l)}- \mb x'_t \rangle(w)\\
    &\leq D_{v_t}(\mb x'_t(w),\mb x^{(l)}_t(w)) - D_{v_t}(\mb x'_t(w),\mb x^{(l+1)}_t(w)) + \frac{1}{2}\gamma_l^2 \|\mb g^{(l)}_t(w)\|_*^2-\gamma_l  \langle \Delta_t^{(l)},\mb x^{(l)}_t - \mb x'_t\rangle(w).
\end{align*}
Taking expectation and summing over $l$ and $t$, and using $D_{v_t}\geq 0$, we get 
\begin{align}\label{eq:md-cum}
     \sum_{l=0}^{L-1} \gamma_l \E[\langle \frac{\partial}{\partial x}f(\mb x^{(l)},\bx),\mb x^{(l)}- \mb x'\rangle]\leq &\E[D_{v}(\mb x',\mb x^{(0)})]+  \frac{1}{2}\sum_{l=0}^{L-1}\gamma_l^2 \E[\|\mb g^{(l)}\|_*^2]\\
     &-\sum_{l=0}^{L-1}\gamma_l \E[ \langle \Delta^{(l)},\mb x^{(l)} - \mb x'\rangle].\nonumber
\end{align}

By convexity of $f(\cdot,\xi)$ we have for any $w\in \Omega$, 
\begin{equation*}
   \sum_{l=0}^{L-1}\gamma_l\langle \frac{\partial}{\partial x}f(\mb x^{(l)}(w),\bx(w)),\mb x^{(l)}(w)- \mb x'(w) \rangle\geq  (\sum_{l=0}^{L-1}\gamma_l)(f(\widetilde{\mb x}^{(L)}(w),\bx(w))-f(\mb x'(w),\bx(w))).
\end{equation*}
Finally, taking $X'_t = X^*_t$, and taking expectation w.r.t. $w$ and dividing by $ (\sum_{l=0}^{L-1}\gamma_l)$ gives the result. 
\end{proof}

\subsection{Proof of Lemma~\ref{lm:MD-saddle}}\label{app:proof:lm:MD-saddle}
\begin{proof}{Proof.} $Z_t^{(l)}\in \mathcal P_t(\mathcal X_t^o \times \mathcal Y_t^o)$ follows from the same argument as in the proof of Lemma \ref{lm:MD1}. In addition, following the same argument for \eqref{eq:md-cum}, we get for any $\mb z_t:\Omega\to \mathcal Z_t$ measurable w.r.t. $\mathcal F_t$, 
\begin{equation*}
    \sum_{l=0}^{L-1}\gamma_l\E[\langle \mb z_{t}^{(l)} - \mb z_t,\frac{\tilde{\partial}}{\partial z_t}\phi(\mb z^{(l)},\bx)\rangle]\leq \E[D_{w_t}(\mb z_t,\mb z^{(0)}_t)]+
   \sum_{l=0}^{L-1} \left[\frac{\gamma_l^2}{2}\E[\|G^{(l)}_t\|_*^2] -\gamma_l\E[\langle \mb z_{t}^{(l)} - \mb z_t,\Delta_t^{(l)}\rangle]\right].
\end{equation*}
Applying Lemma \ref{lm:MD-prop-lemma} to the sequence $\widehat{\mb z}_t^{(l)}$, we get 
\begin{equation*}
    -\gamma_l \E[\langle\widehat{\mb z}_{t}^{(l)} - \mb z_t,\Delta_t^{(l)}\rangle]\leq  \frac{\gamma_l^2}{2}\E[\|\Delta_t^{(l)}\|_*^2] + \E[D_{w_t}(\mb z_t,\widehat{\mb z}^{(l)}_t) - D_{w_t}(\mb z_t,\widehat{\mb z}^{(l+1)}_t)].
\end{equation*}
Thus we get
\begin{align}\label{eq:md-saddle-cum}
    &\quad \sum_{l=0}^{L-1}\gamma_l\E[\langle \mb z_{t}^{(l)} - \mb z_t,\frac{\tilde{\partial}}{\partial z_t}\phi(\mb z^{(l)},\bx)\rangle]\nonumber\\
    &\leq \E[2D_{w_t}(\mb z_t,\mb z^{(0)}_t)+
   \sum_{l=0}^{L-1} \left[\frac{\gamma_l^2}{2}(\|\mb g^{(l)}_t\|_*^2 +\|\Delta_t^{(l)}\|_*^2 )-\gamma_l\langle \mb z_{t}^{(l)} - \widehat{\mb z}_t^{(l)},\Delta_t^{(l)}\rangle\right]].
\end{align}
     
Since $\phi(x_{1:T},y_{1:T},\xi)$ is convex in $x_{1:T}$ and concave in $y_{1:T}$, we have for any $\mb z_t:\Omega\to \R^{m_t+n_t}$ for $t=1,\ldots,T$,
   
    \begin{equation}\label{eq:gap_ub}
        \E[\phi(\widetilde{\mb x}^{(L)},\mb y,\bx) -\phi(\mb x,\widetilde{\mb y}^{(L)},\bx) ]\leq \E[\left(\sum_{l=0}^{L-1} \gamma_l\right)^{-1}\left(\sum_{l=0}^{L-1}\gamma_l  \langle 
        \mb z^{(l)} - \mb z,\frac{\tilde{\partial}}{\partial z}\phi(\mb z^{(l)},\bx)\rangle\right)]. 
    \end{equation}

Finally, summing \eqref{eq:md-saddle-cum} over $t$ and together with \eqref{eq:gap_ub} give the result.
\end{proof}

\subsection{Proof of Lemma~\ref{lm:prop-oracle}}\label{app:proof:lm:prop-oracle}
\begin{proof}{Proof.}
The measurability of $G_t$ follows from the measurability of $\frac{\partial}{\partial x_t} f$, $X_{s}$ for $t-1\leq s\leq t+1$, and $\widehat{\bx}_{t+1}$.

    For \eqref{eq:grad-unbias}, notice that the only difference between $G_t^{*}(\xi_{1:t})$ and $\E[G_t(\xi_{1:t},U)]$ is the distribution over $\bx_{t+1}$, which differs by $\beta_{t+1}$ in TV distance. Using $\|\frac{\partial}{\partial x_t}f\|_*\leq \lips_t$ gives the result. For \eqref{eq:grad-var}, we use $\|G_t\|_*\leq \lips_t$. 
\end{proof}

\subsection{Proof of Lemma~\ref{lm:MDSA-1}}\label{app:proof:lm:MDSA-1}
\begin{proof}{Proof.}
The first two claims can be proved by induction. The case $l=0$ certainly holds. Now suppose claim 1 and 2 holds for $l$. For $l+1$, by the update $X_t^{(l+1)}$ in \eqref{eq:MD1}, $X_t^{(l+1)}$ depends on $G_t^{(l)}$ and $X_t^{(l)}$ and thus on $S'(t,l)$ only. $G_t^{(l+1)}$ depends on the available neighboring iterates $X_s^{(l+1)}$ and the fresh seed $\zeta_{t,l+1}$, and thus on $\left(\bigcup_{s=\max(1,t-1)}^{\min(T,t+1)}S'(s,l)\right)\cup\{(t,l+1)\}=S'(t,l+1)$ only. The first two claims follow from induction. 

For the third claim, notice that $\overline{X}_t^{(L)}$ depends on $X_t^{(0:(L-1))}$ and thus on $\cup_{l=0}^{L-1}S'(t,l)=S'(t,L-1)$ only. The claim on $S'(t,l)$ can be proved by induction. 
\end{proof}

\subsection{Proof of Theorem~\ref{thm:MDSA-1}}\label{app:proof:thm:MDSA-1}
\begin{proof}{Proof.}
By the Lipschitz assumption and our choice for $\overline{G}_t^{(l)}$, we have $\|\overline{G}_t^{(l)}\|_*\leq \lips_t$. Thus, from Lemma \ref{lm:MD1} 
we have for each $w_{\internal}\in \Omega_{\internal}$, the following holds for $\overline{X}_t^{(w_{\internal})}:= \overline{X}_t(\cdot,\bz_{S(t)}(w_{\internal}))$:
\begin{equation}\label{eq:main-pf-1}
        \gapu(\overline{X}^{(w_{\internal})}) \leq \frac{\widetilde{D}^2+ \frac{\tilde{L}^2}{2}\sum_{l=0}^{L-1} \gamma_l^2 -\sum_{l=0}^{L-1}\gamma_l \E[\langle \overline{\Delta}^{(l,w_{\internal})},\overline{\mb x}^{(l,w_{\internal})} - \mb x^* \rangle]}{\sum_{l=0}^{L-1} \gamma_l},
   \end{equation}

    where $\overline{\mb x}_t^{(l,w_{\internal})} (w)= \overline{X}_t^{(l)}(\bx_{1:t}(w),\bz_{S'(t,l-1)}(w_{\internal}))$, $\overline{\mb g}_t^{(l,w_{\internal})} (w)= \overline{G}_t^{(l)}(\bx_{1:t}(w),\bz_{S'(t,l)}(w_{\internal}))$, and $\overline{\Delta}^{(l,w_{\internal})}_t(w)=\overline{\mb g}^{(l,w_{\internal})}_t(w)-\E[\frac{\partial}{\partial x_t}f(\overline{\mb x}^{(l,w_{\internal})},\bx)|\bx_{1:t}](w)$.

    For convenience, for each $w\in \Omega$ and $w_{\internal}\in \Omega_{\internal}$, viewing $\overline{\mb g}_t^{(l)} (w)$ as a random variable in $\Omega_{\internal}$, due to Lemmas \ref{lm:prop-oracle} and \ref{lm:MDSA-1}: for any $w_{\internal}\in \Omega_{\internal},w\in \Omega$,
    \begin{equation}\label{eq:delta-prop1}
        \|\E_{\internal}[\overline{\Delta}^{(l)}_t(w)|\bz_{1:T,0:(l-1)}](w_{\internal})  \|_*\leq b_t
    \end{equation}

Thus, we have
\begin{align*}
    &\quad |\E_{\internal}[\langle \overline{\Delta}_{t}^{(l)}(w),\overline{\mb x}_{t}^{(l)}(w)-\mb x^*_{t}(w)\rangle|\bz_{1:T,0:(l-1)}]| \\
    &= |\langle \E_{\internal}[\overline{\Delta}_{t}^{(l)}(w)|\bz_{1:T,0:(l-1)}],\overline{\mb x}_{t}^{(l)}(w)-\mb x^*_{t}(w)\rangle|\\
    &\leq  \|\E_{\internal}[\overline{\Delta}_{t}^{(l)}(w)|\bz_{1:T,0:(l-1)}]\|_*\cdot \|\overline{\mb x}_{t}^{(l)}(w)-\mb x^*_{t}(w)\|\leq b_t\cdot \|\overline{\mb x}_{t}^{(l)}(w)-\mb x^*_{t}(w)\|.
\end{align*}
Thus, the following holds for all $w_{\internal}\in \Omega_{\internal}$:
\begin{equation*}
    |\E_{\internal}[\langle \overline{\Delta}^{(l)}(w),\overline{\mb x}^{(l)}(w)-\mb x^*(w)\rangle|\bz_{1:T,0:(l-1)}] |\leq \sum_{t=1}^Tb_t\cdot \|\overline{\mb x}_{t}^{(l)}(w)-\mb x^*_{t}(w)\| \leq bD. 
\end{equation*}
Taking expectation of \eqref{eq:main-pf-1}, we get 
\begin{equation*}
        \E_{\internal}[\gapu(\overline{X})] \leq \frac{\tilde{D}^2 + \frac{\tilde{L}^2}{2}\sum_{l=0}^{L-1} \gamma_l^2  }{\sum_{l=0}^{L-1} \gamma_l}+bD,
   \end{equation*}
which proves \eqref{eq:convergence-in-exp}.

\end{proof}

\subsection{Proof of Lemma~\ref{lm:procedure-prop}}\label{app:proof:lm:procedure-prop}
\begin{proof}{Proof.}
For the first claim, if $t=  T$, then the procedure iteratively computes $X_T^{(1)}(\xi_{1:t}),\ldots,X_T^{(l)}(\xi_{1:t})$, and indeed, for $l'=1,\ldots,l$, since by assumption, $\memo$ contains $X_{T-1}^{(l'-1)}(\xi_{1:(t-1)})$, and $X_T^{(l'-1)}(\xi_{1:t})$ has been computed and stored in $\memo$ in the previous iteration, $X_T^{(l')}(\xi_{1:t})$ can be computed. If $l>0$ and $t<T$, then $X_{t+1}^{(l'-1)}(\xi_{1:t},\widehat{\bx}_{t+1}(\xi_{1:t},\zeta_{t,l'-1}))$ is evaluated, and indeed, $X_{t}^{(0:(l'-2))}(\xi_{1:t})$ have been computed in the previous iterations and are stored in $\memo$. For $l'=1,\ldots,l$, when computing $X_t^{(l')}(\xi_{1:t})$, ${X}_t^{(l'-1)}(\xi_{1:t})$ has been computed in previous iteration, $X_{t+1}^{(l'-1)}(\xi_{1:t},\widehat{\bx}_{t+1}(\xi_{1:t},\zeta_{t,l'-1}))$ has been evaluated, and if $t\geq 2$, $X_{t-1}^{(l'-1)}(\xi_{1:(t-1)})$ is in $\memo$ by assumption.

For the second claim, first, notice that for the procedure, the gradient oracle calls involved are the following: $l$ calls to $\frac{\partial}{\partial x_t}f$, and if $l>0,t<T$, then all oracle calls when evaluating $X_{t+1}^{(l'-1)}(\xi_{1:t},\widehat{\bx}_{t+1}(\xi_{1:t},\zeta_{t,l'-1}))$ for $l'=1,\ldots,l$. Denoting the number of total oracle calls when evaluating $X_{t'}^{(l')}(\xi'_{1:t})$ by $a_{t',l'}$. Then we can take $a_{t',0}=0$ and $a_{T+1,l'}=0$, and the following holds 
\begin{equation*}
    a_{t,l} = l + \sum_{l'=1}^{l-1} a_{t+1,l'}.
\end{equation*}
Simple induction then shows that $a_{t,l}\leq \min(2^l-1,l^{T-t+1})$ for all $t,l$.

For the third claim, the first part can be proven by induction on $(l,t)$ where the base case is $l=0$ and $t = T$: if $l=0$ then $\overline{\memo}_0 = \memo$; if $t=T$, then only $lB$ bits of information (i.e. $X_t^{(1:l)}(\xi_{1:t})$) is written to $\memo$. If $l>0,t<T$, since by inductive hypothesis, evaluating $X_{t+1}^{(l'-1)}(\xi_{1:t},\widehat{\bx}_{t+1}(\xi_{1:t},\zeta_{t,l'-1}))$ only adds $X_{t+1}^{(1:(l'-1))}(\xi_{1:t},\widehat{\bx}_{t+1}(\xi_{1:t},\zeta_{t,l'-1}))$ to $\memo$, which is deleted by end of iteration $l'$, and so after the $l'$-th iteration only $X_t^{(l')}(\xi_{1:t})$ is added to $\memo$. For the second part, the statement is true for $l=0$ and $t = T$. For $l>0,t<T$, notice that $X_{t+1}^{(l'-1)}(\xi_{1:t},\widehat{\bx}_{t+1}(\xi_{1:t},\zeta_{t,l'-1}))$ are evaluated sequentially, and at the end of iteration $l'$, all intermediate results are deleted. Denoting the maximum (over all $t'$) additional space needed by when evaluating $X_{t+1}^{(l'-1)}(\xi_{1:t},\widehat{\bx}_{t+1}(\xi_{1:t},\zeta_{t,l'-1}))$ by $b_{l'}$: we can take $b_0 =0$, $b_1 = B$, then we have 
\begin{equation*}
    b_{l}\leq l B + \max_{l'=0,\ldots, l-1} b_{l'}.
\end{equation*}
Then induction gives $b_l\leq Bl(l+1)/2 = O(l^2B)$. 
\end{proof}

\subsection{Proof of Theorem~\ref{thm:main-async}}\label{app:proof:thm:main-async}
\begin{proof}{Proof.}
For the first claim, when $t=1$, the evaluation procedure can be called since $\memo^{(init)}\subset \memo$. By the end of stage $1$, \eqref{eq:memo_update} follows from Lemma \ref{lm:procedure-prop}. Now suppose the first claim is true for stage $1,\ldots,t-1$, then at the beginning of stage $t$, the memory space is $\overline{\memo}_{t-1}$, and so the $X_t^{(L)}(\xi_{1:t})$ can be evaluated, which adds $X_t^{(1:L)}(\xi_{1:t})$ to the memory (Lemma \ref{lm:procedure-prop}). The delete operation removes all evaluated variables for $t-1$. Thus, the first claim holds for $t$. By induction, it holds for all $t=1,\ldots,T$.

For the second claim, from Lemma \ref{lm:procedure-prop}, each $X_t(\xi_{1:t})$ requires at most $\min(2^L-1,L^{T-t+1})$ oracles. Thus, the total number of oracles is upper bounded by $O(T \cdot \min(2^L,L^T))$. Since each gradient oracle is used in one proximal update, the number of updates is also $O(T \cdot \min(2^L,L^T))$. 

For the third claim,  By \eqref{eq:memo_update}, $\overline{\memo}_t\setminus \overline{\memo}_0$ requires at most $O(LB) $ bits. Notice that at stage $t$, $X_t^{(L)}(\xi_{1:t})$ is evaluated, which requires at most $O(L^2B)$ bits of additional memory by Lemma \ref{lm:procedure-prop}, we have $\memo\setminus \overline{\memo}_t$ is no more than $O(LB +L^2B)$ throughout stage $t$. Thus, $\memo\setminus \overline{\memo}_0$ requires no more than $O(LB +L^2B)$ bits. 
\end{proof}

\section{Numerical experiments}\label{sec:num_exp}

We apply our mirror descent stochastic approximation \eqref{eq:MD1} (denoted as MDSA) and \eqref{eq:update_acc_inexact} (denoted as A-MDSA) to a smoothed online convex optimization problem, and \eqref{eq:MD2} (denoted as MDSA) to a revenue management problem. 

Our experiment results demonstrate that the proposed mirror descent stochastic approximation algorithms converge in a variety of settings: convex optimization with and without strong convexity, and saddle point problems. In addition, they admit the following advantages: applicable even when the randomness across stages is \textit{correlated}, robustness against \textit{misspecified} sampling distribution in constructing the gradients, efficiency for large $T$.

All experiments are implemented using Python and run on MacBook Air with the M3 chip.

\subsection{Smoothed online convex optimization}\label{sec:exp-soco-gen}
We consider a smoothed online convex optimization problem, modeled as \eqref{eq:obj_unconstrained-online} with
\begin{equation}\label{eq:soco}
     f_t(x_{t-1},x_t,\xi_t): = h(\|x_t - (\theta_t+\overline{\xi}_t)\|_2) + \frac{1}{2} \|x_t - x_{t-1}\|_2^2,\quad \xi_t = (\overline{\xi}_t,\epsilon_t).
\end{equation}
The objective function \eqref{eq:soco} appears in tracking problems \citep{LiLi2020}, where the goal is to decide the positions $x_1,\ldots,x_T$ in order to minimize the distance from the moving target ($\theta_t + \overline{\xi}_t$ at stage $t$) and the moving cost (functions of $\|x_{t}-x_{t-1}\|_2$). 

In the experiment, we take $x_0 = \mb 0$  and $h:[0,\infty)\to \R$ is either the strongly convex quadratic cost $h_{\text{quad}}(s) = s^2/2$, or the huber cost $h_{\text{huber}}(s) = \begin{cases}
    s^2/2~ &|s|\leq 1\\
    |s| - 1/2~ &|s|>1 
\end{cases}$, which is convex but not strongly convex. $n_t=10$ and $\mathcal X_t = \{x_t\in \R^{n_t}, ~\|x_t\|_2\leq 10\}$ is the ball with radius $10$ for all $t$. For $i=1,\ldots,10$, $\theta_{t,i} = 7.5\sin(2\pi \cdot (1+\frac{i-1}{100})t)$ is a known sequence. 

For the stochastic process $\bx_{1:t}$, we assume that the $\overline{\bx}_{1:T}$ component is an auto-regressive process with discount factor $\rho = 0.8$, such that $\overline{\bx}_1 = \boldsymbol{\epsilon}_1$, and $\overline{\bx}_{t} = \rho \overline{\bx}_{t-1} + \boldsymbol{\epsilon}_t$ for $t\geq 2$. For the stochastic process $\boldsymbol{\epsilon}_{1:T}$, we take the (true) distribution of  $\boldsymbol{\epsilon}_t$ conditioning on $\bx_{1:(t-1)}=\xi_{1:(t-1)}$ to be a uniform distribution over a (finite) set $E_t(\xi_{1:(t-1)})$. These sets $E_t(\xi_{1:(t-1)})$ will be specified in Sections \ref{sec:soco-exp} and \ref{sec:exp_soco_T} below. For convenience, we denote this uniform distribution as $\pi_t(\xi_{1:(t-1)})$.

\textbf{Misspecified distributions.} To test robustness of our algorithms against biases in the conditional scenario samplers $\widehat{\bx}_t$, given $\xi_{1:(t-1)}$, we sample $\boldsymbol{\epsilon}_t$ using the perturbed distribution $\hat{\pi}_t(\xi_{1:(t-1)}) = (1-\delta)\pi_t(\xi_{1:(t-1)}) + \delta d_t(\xi_{1:(t-1)})$ where $d_t(\xi_{1:(t-1)})$ is a probability distribution and $\delta\in [0,1]$. Thus, $d_t$ represents the misspecification in the distribution of $\boldsymbol{\epsilon}_t$. We consider $2$ types of perturbation $d_t$: for $d^{(0)}_t$, for each $\epsilon_t\in E_{t}(\xi_{1:(t-1)})$, we generate a uniform random variable in $[0,1]$, independently. Then, we normalize them such that their sum is $1$. For $d^{(1)}_t$, we pick one element in $E_{t}(\xi_{1:(t-1)})$ uniformly at random and set the corresponding component in $d_t$ to be $1$, and all rests are set to $0$. Thus, type $d^{(1)}_t$ can be viewed as a more adversarial perturbation. We generate one $d_t$ for each $\xi_{1:(t-1)}$ independently.

\textbf{Algorithms setup.} We use $v_t(x_t) = \frac{1}{2}\|x_t\|_2^2$ as the distance generating functions. For the accelerated updates \eqref{eq:update_acc_inexact}, we use the $\alpha_l,A_l$ as defined in \eqref{eq:update_acc_inexact} with $\gamma = 1$ and $\theta = 1/2$. For $h= h_{\text{quad}}$, we take $\mu = 1$ and $L_2 = 3$. For $h=h_{\text{huber}}$, we take $\mu = 0$ and $L_2 = 3$. 

Additional experiments show that the performance of (A-)MD(SA) algorithms is similar to that presented below, under a variety of parameter settings for $\delta$, $\rho$, and $\alpha$ when $h= \alpha h_{\text{quad}}$ and $h= \alpha h_{\text{huber}}$. We omit these results due to space constraints. 

\subsubsection{Case 1: $T = 5$}\label{sec:soco-exp}

We generate $E_t(\xi_{1:(t-1)})$ such that $|E_t(\xi_{1:(t-1)})| = 10$ for $t=2,\ldots,T$ in the following manner: we first sample a vector $\epsilon_1$ under $\mathcal N(\mb 0,16 I)$ and set $E_1 = \{\epsilon_1\}$;then we generate $10$ $\epsilon_2\sim_{i.i.d.} \mathcal N(\mb 0,16 I)$ independent of $\epsilon_1$, and $E_2((\epsilon_1,\epsilon_1))$ is the set of these $10$ random vectors (Recall that $\xi_1 = (\overline{\xi}_1,\epsilon_1) = (\epsilon_1,\epsilon_1)$); then for each $(\epsilon_1,\epsilon_2)$, we generate $10$ $\epsilon_3\sim_{i.i.d.} \mathcal N(\mb 0,16 I)$ independent of $(\epsilon_1,\epsilon_2)$, and set $E_3(\xi_{1:2})$ to be these $10$ vectors, and so on for $t=4,5$. 

For the updates \eqref{eq:MD1}, we take the step size $\gamma_l = 1/\sqrt{L}$. We test two types of initialization: $X_t^{(0)} = \mb 0$ and $X_t^{(0)} = \operatorname{Proj}_{\mathcal X_t}(\theta_{t-1})$, with $\theta_0=\mb 0$. We generate $5$ realizations of the random seeds $\zeta^{(1)},\ldots,\zeta^{(5)}$, and in Figure \ref{fig:soco-exp}, we present the averages of $\{\E[f(\overline{X}^{(l)}_{1:T}(\bx_{1:T},\zeta^{(i)}),\bx_{1:T})],~i=1,\ldots,5\}$ as ``MDSA (mean)'' (initialization is $X_t^{(0)} = \mb 0$) and ``MDSA (mean), init'' (initialization is $X_t^{(0)} = \operatorname{Proj}_{\mathcal X_t}(\theta_{t-1})$). Here $\overline{X}^{(l)}$ is the weighted average of $X^{(0:l)}_{1:T}$ using weights $\gamma_{0:l}$ in Algorithm \ref{alg:hypo-ms-u}. In addition, we also run \eqref{eq:MD1} with the \textit{exact} gradient computed under $\hat{\pi}_t$ ((MD) and (MD, init)) and under ${\pi}_t$ ((MD (true dist)) and (MD (true dist), init)). We present the results for the accelerated updates, with the same setup. We point out that the first output of the accelerated updates is given by \eqref{eq:update_acc_inexact}, and so is not necessarily $X_t^{(0)}$.
\begin{figure}[htbp]
    \centering
    \begin{subfigure}[t]{0.4\textwidth}
        \centering
        \includegraphics[width = \textwidth]{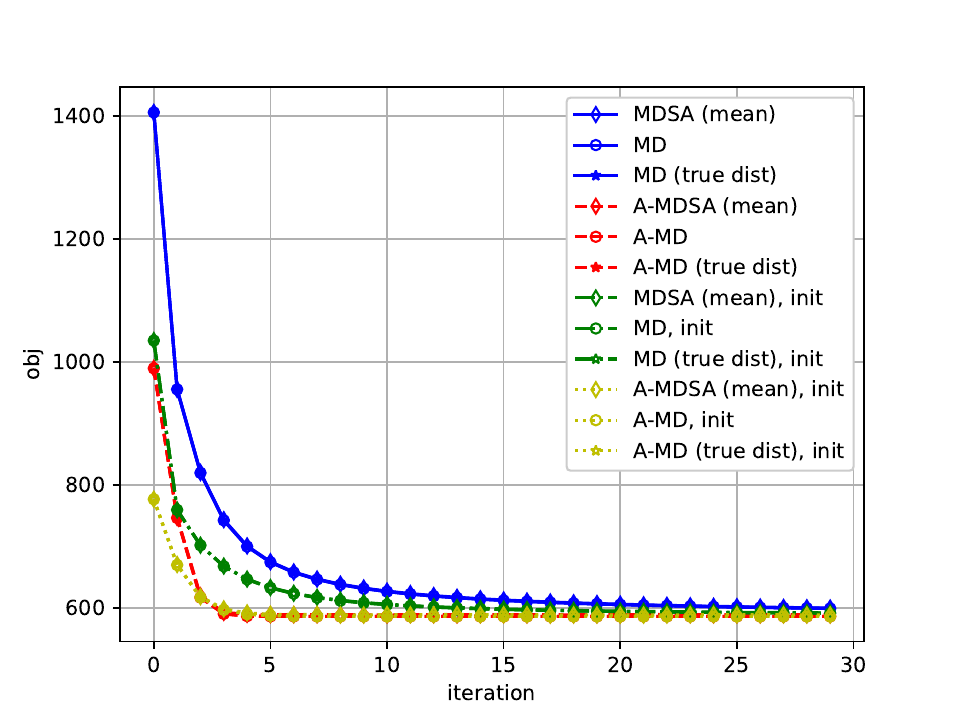}
        \caption{$(0.1,d^{(0)}_t)$, $ h_{\text{quad}}$}
    \end{subfigure}
    \begin{subfigure}[t]{0.4\textwidth}
        \centering
        \includegraphics[width = \textwidth]{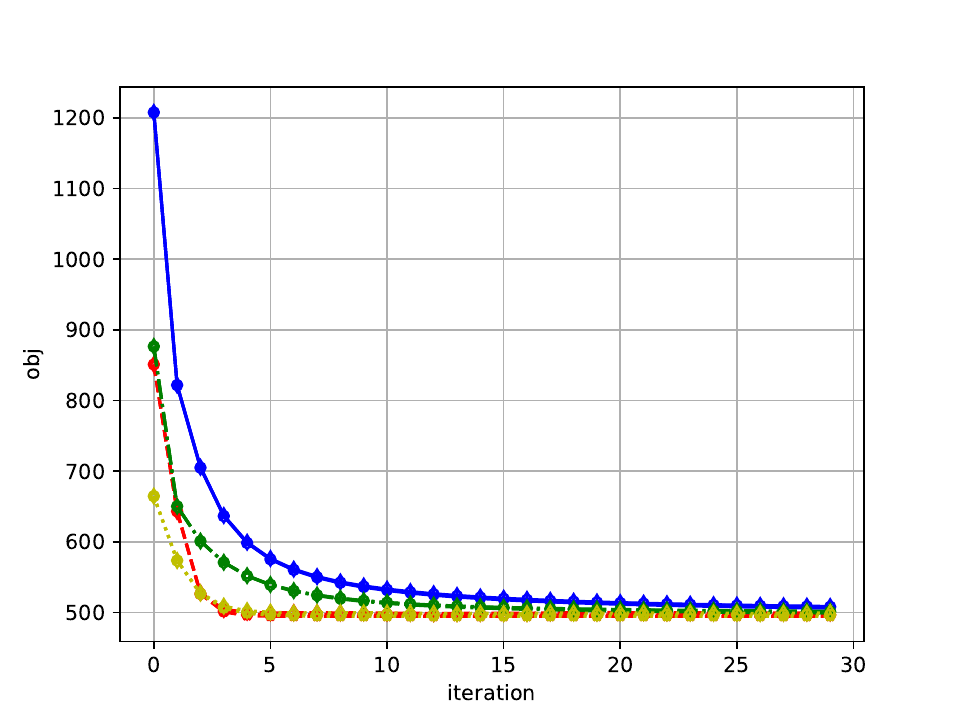}
        \caption{$(0.5,d^{(1)}_t)$, $ h_{\text{quad}}$}
    \end{subfigure}
    
    \begin{subfigure}[t]{0.4\textwidth}
        \centering
        \includegraphics[width = \textwidth]{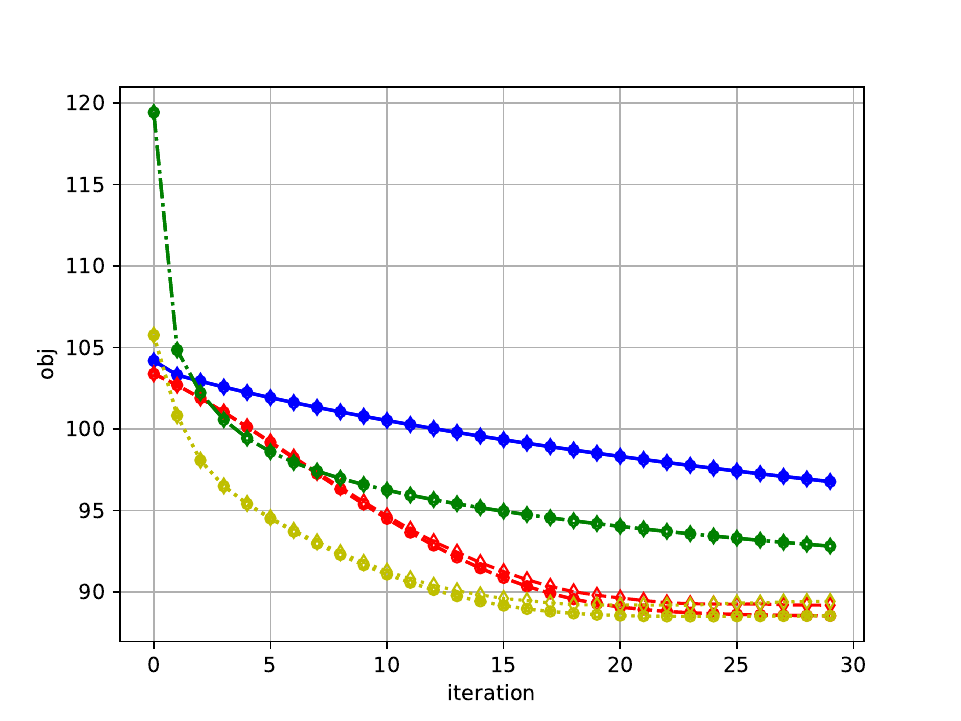}
        \caption{$(0.1,d^{(0)}_t)$, $ h_{\text{huber}}$}
    \end{subfigure}
    \begin{subfigure}[t]{0.4\textwidth}
        \centering
        \includegraphics[width = \textwidth]{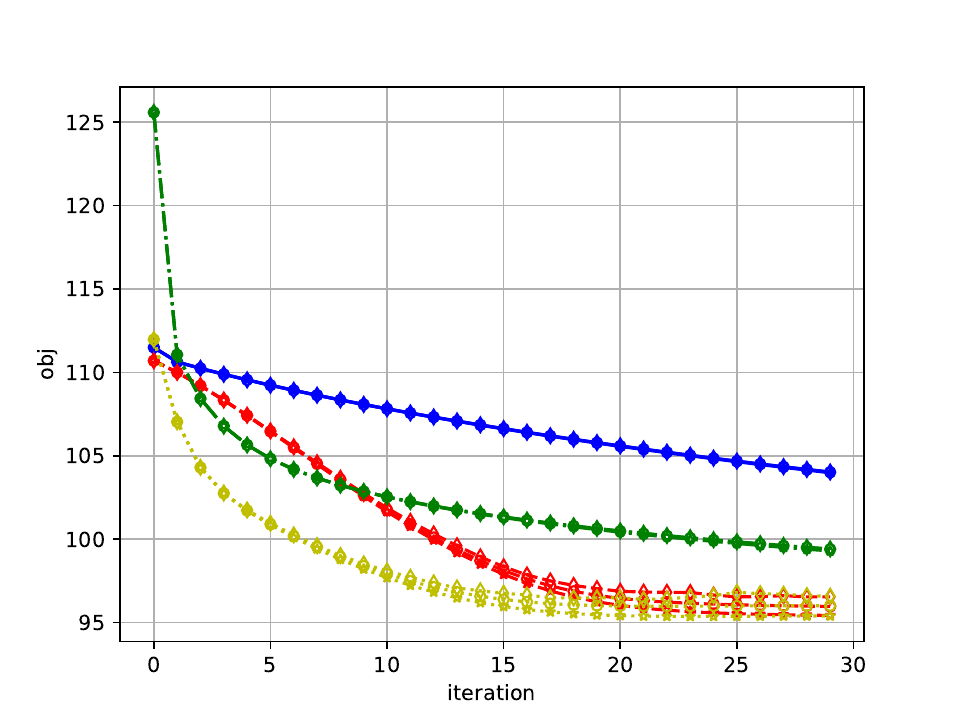}
        \caption{$(0.5,d^{(1)}_t)$, $ h_{\text{huber}}$}
    \end{subfigure}
        \caption{Objective values $\E[f(X_{1:T}(\bx_{1:T}),\bx_{1:T})]$ for the smoothed online convex optimization problems corresponding to $h$, under different settings of $\{(\delta,d_t),h\}$. MDSA (Algorithm \ref{alg:hypo-ms-u}) and A-MDSA (Algorithm \ref{alg:hypo-ms-acc}) are applied with inexact sampling distributions $\hat{\pi}_t = (1-\delta) \pi_t +\delta d_t$. MD and MD (true dist) use exact gradients computed under $\hat{\pi}_t$ and $\pi_t$ respectively, and similarly for A-MD and A-MD (true dist). The ``init'' means the initialization is at $\theta_t$ for $X_t$; otherwise the initialization is $\mb 0$. Legends for (b,c,d) are the same as the legend for (a). 
        }\label{fig:soco-exp}
\end{figure}

\textbf{Effects of stochastic gradients.} 
As expected, with the stochastic gradient oracle (MDSA and A-MDSA), the total running time is approximately $ 80\sim 90$ seconds, while with the exact gradient ((A-)MD, (A-)MD (true dist)), the total running time is approximately $140\sim 150$ seconds. Thus, the running time is significantly reduced if the stochastic gradient is used, without compromising the convergence by too much. In terms of the convergence of objective values, notice that one key difference between stochastic gradients and full gradients is the noise term $\sigma_t$, which is $0$ for full gradients, and which depends on the variance of the gradient for stochastic gradients. This is reflected by Figure \ref{fig:soco-exp} (c) and (d), where A-MDSA is converging to a \textit{slightly} larger value than A-MD and A-MD (true dist) for both initialization. 

\textbf{Correlation between randomness.} For both the quadratic loss and the huber loss, we test our algorithms for $\rho = 0.8$. With this correlated sequence $\bx$, as shown by all setups in Figure \ref{fig:soco-exp}, all of our algorithms converge within $30$ iterations, or show a trend of convergence (MD(SA) for huber loss). Indeed, our theoretical results do not assume independence between randomness in different stages.

\textbf{Robustness against misspecification.} Comparing the results for A-MD and A-MD (true dist) for both initializations in Figure \ref{fig:soco-exp} (d), we see that A-MD is converging to a \textit{slightly} larger value than A-MD (true dist). However, in all other settings, the bias in the sampling distribution does not have a noticeable effect on the convergence.

\textbf{Strong convexity.} Comparing the results for $h_{\text{quad}}$ and $h_{\text{huber}}$, we see that for both the accelerated and the non-accelerated updates, strongly convex objectives converge faster for all our algorithms. For the update \eqref{eq:update_acc_inexact}, this agrees with our theoretical results; for the update \eqref{eq:MD1}, this suggests that the suboptimality in Theorem \ref{thm:MDSA-1} could be loose, and tighter bounds could be attained with strong convexity.

\textbf{Acceleration.} As expected, in all our settings, when the initializations are the same, A-MD(SA) converges faster than MD(SA): for $h_{\text{quad}}$, A-MD(SA) converges in $\sim 5$ iterations, MD(SA) converges in $\sim 15$ iterations; for $h_{\text{huber}}$, A-MD(SA) converges in $\sim 20$ iterations, MD(SA) does not reach convergence in $30$ iterations.

\subsubsection{Case 2: $T=30,50$}\label{sec:exp_soco_T}

We test our MDSA and A-MDSA using the efficient online updates in Algorithm \ref{alg:async} for the smoothed online convex optimization problem with $T = 30,50$. For $t=2,\ldots,T$, we choose $E_t(\xi_{1:(t-1)})=E$ to be the same, where $|E|=50$, and $E$ consists of $50$ i.i.d. random vector with distribution $\mathcal N(\mb 0, 16 I)$. We fix one element $\epsilon_1\in E$ as the first-stage innovation and set $E_1=\{\epsilon_1\}$.

We take $\delta = 0.2$ and $d_t^{(0)}$ as the perturbation to the sampling distribution. In addition, if the procedure Algorithm \ref{alg:update-procedure} is applied to the same scenario $\xi_{1:t}$ more than once, the sampling distributions are perturbed by two independent random $d_t^{(0)}$'s. For \eqref{eq:MD1}, we use $\gamma = 3/\sqrt{L}$, and we use $X_t^{(0)} = \mb 0$ as initialization.  

\begin{figure}[htbp]
    \centering
    \begin{subfigure}[t]{0.4\textwidth}
        \centering
        \includegraphics[width = \textwidth]{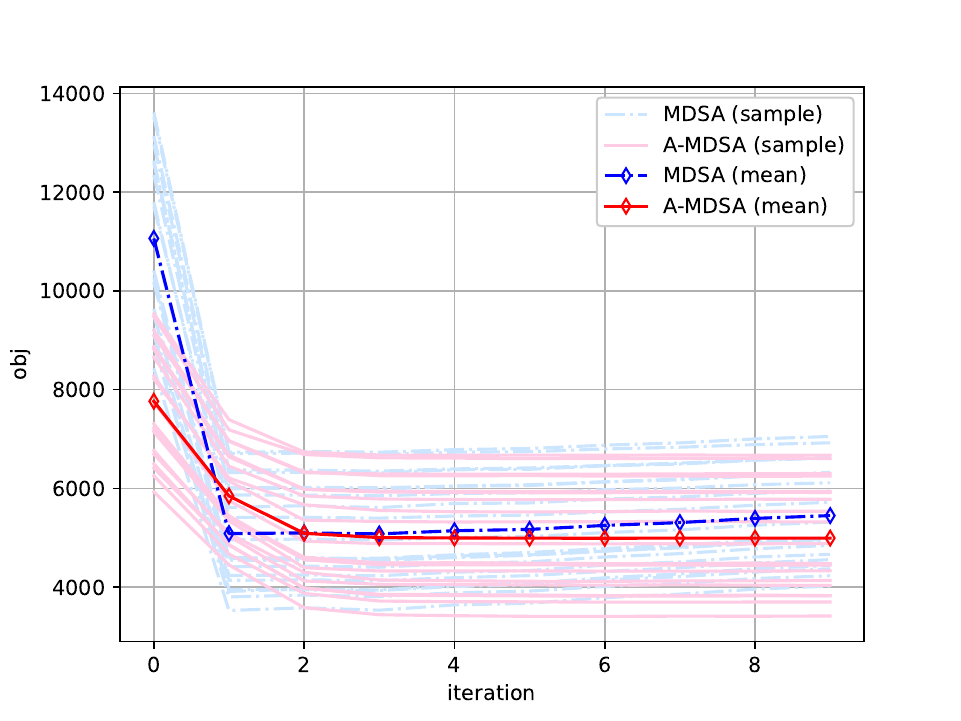}
        \caption{$T = 30$, $ h_{\text{quad}}$. }
    \end{subfigure}
    \begin{subfigure}[t]{0.4\textwidth}
        \centering
        \includegraphics[width = \textwidth]{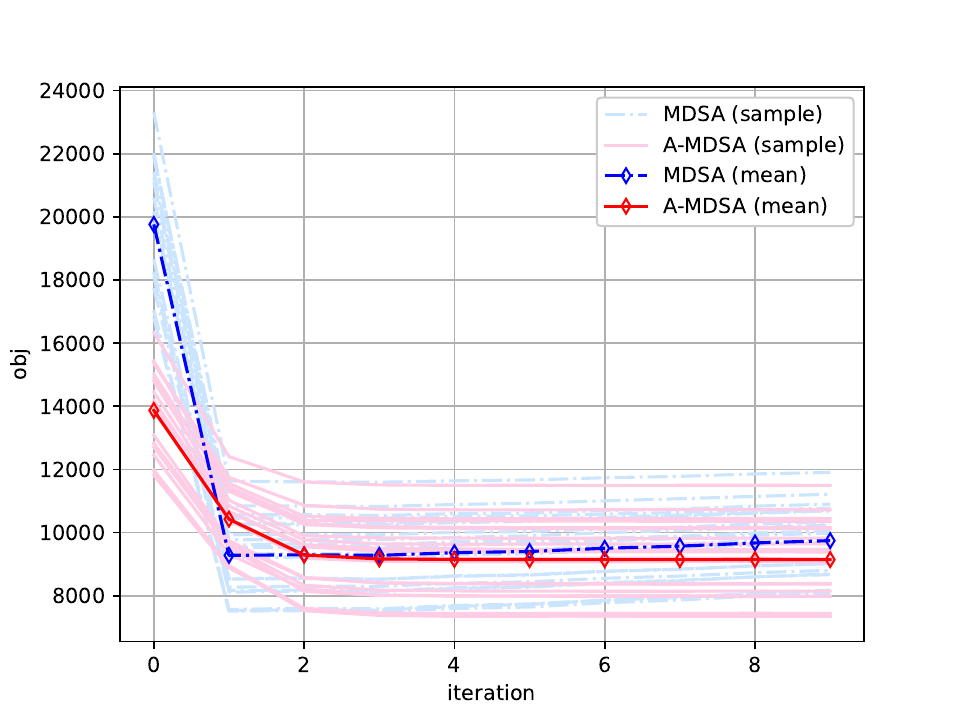}
        \caption{$T = 50$, $ h_{\text{quad}}$. }
    \end{subfigure}   
    
    \begin{subfigure}[t]{0.4\textwidth}
        \centering
        \includegraphics[width = \textwidth]{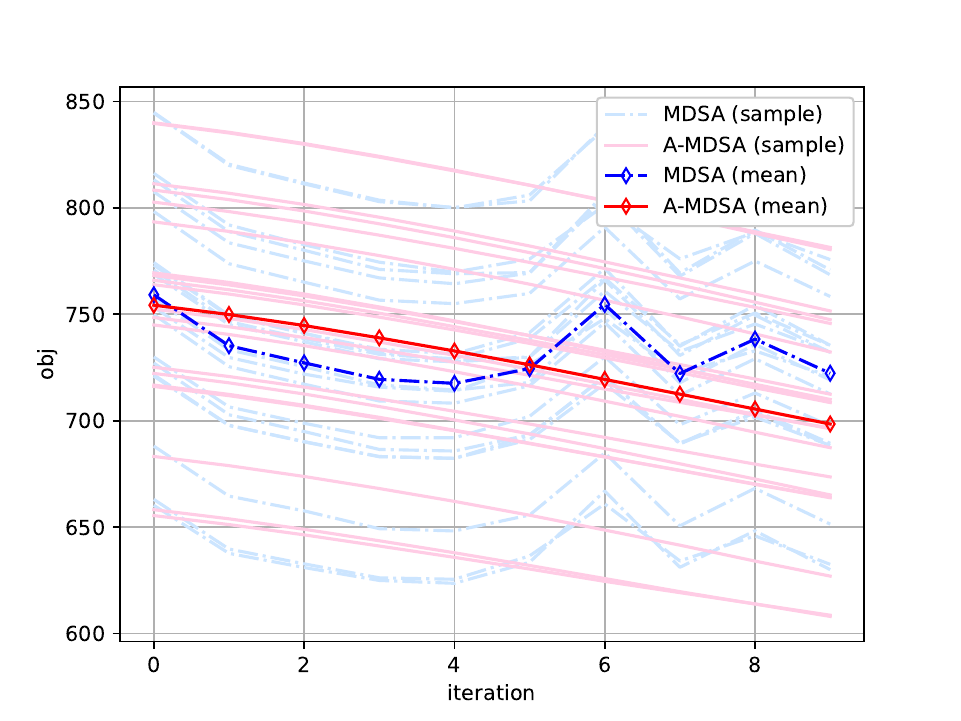}
        \caption{$T = 30$, $h_{\text{huber}}$. }
    \end{subfigure}
    \begin{subfigure}[t]{0.4\textwidth}
        \centering
        \includegraphics[width = \textwidth]{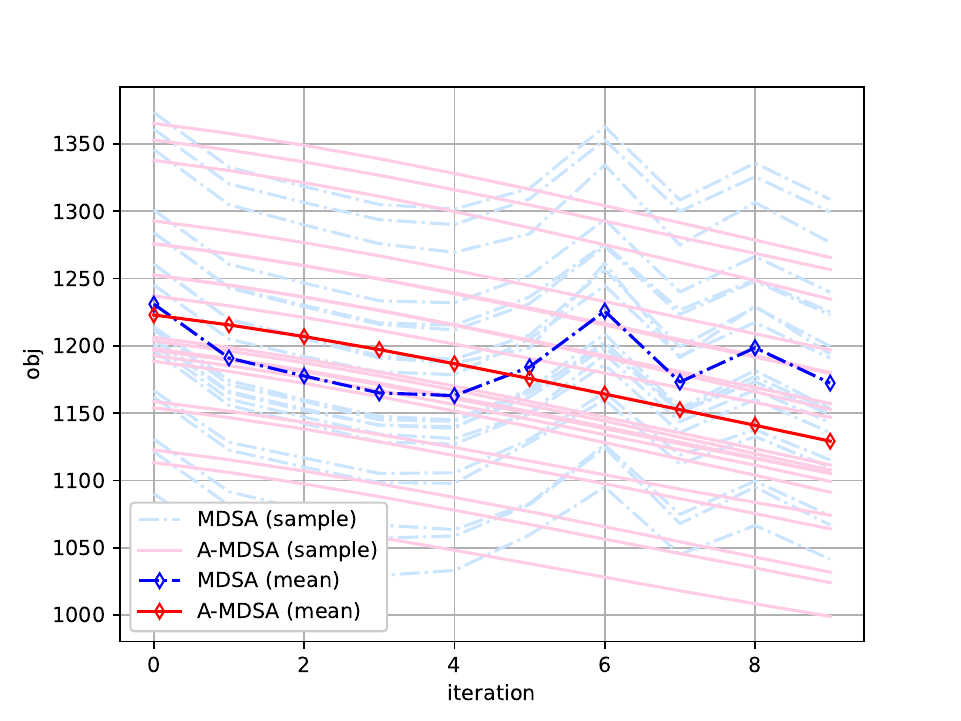}
        \caption{$T = 50$, $h_{\text{huber}}$. }
    \end{subfigure}
    \caption{ MDSA and A-MDSA using the online Algorithm \ref{alg:async}, applied to the smoothed online convex optimization problem corresponding to $h$, with inexact sampling distribution $\hat{\pi}_t = 0.8\pi_t +0.2 d_t^{(0)}$. The light blue lines represent the objective value $f(\overline{X}^{(l)}(\xi,\zeta),\xi)$ for $20$ generated sequence $\xi$ when running MDSA, and the dark blue line is their average. Similarly, the light red lines represent $f(X_+^{(l)}(\xi,\zeta),\xi)$ for $20$ generated sequence $\xi$ when using A-MDSA, and the dark red lines are their averages.}\label{fig:soco-exp-long-T}
\end{figure}

Our experiments show that even when $T$ and the number of child nodes are large, our algorithms are very efficient: the running time for the entire $T$ stages is approximately $ 20$ seconds for $T = 30$ , and approximately $45$ seconds for $T = 50$. 

In Figure \ref{fig:soco-exp-long-T}, the light blue lines represent the objective value $f(\overline{X}^{(l)}(\xi^{(i)},\zeta),\xi^{(i)})$ for $20$ generated sequence $\{\xi^{(i)},~i=1,\ldots,20\}$ when running MDSA. Here $\overline{X}^{(l)}$ is the weighted average of $X^{(0:l)}(\xi,\zeta)$ using weights $\gamma_{0:l}$ in Algorithm \ref{alg:hypo-ms-s}. We also plot their averages in dark blue. Similarly, the light red lines represent $f(X_+^{(l)}(\xi^{(i)},\zeta),\xi^{(i)})$ for $20$ generated sequence $\xi$ when using A-MDSA, and the dark red lines are their averages. 

As revealed in Figure \ref{fig:soco-exp-long-T}, the quadratic loss converges within $5$ iterations, while the huber loss does not reach convergence within $10$ iterations, but the objective values show a decreasing trend. These agree with our theoretical results in Corollary \ref{cor:overall}.

\subsection{Revenue management}\label{sec:exp_rm}
We consider the following revenue management problem:
\begin{equation*}
    \max_{x_1\in [0,1]}\cdots \max_{x_T\in [0,1]} \sum_{t=1}^T c_tx_t,\quad s.t. \sum_{t=1}^T\mb a_tx_t\leq \mb b_0.
\end{equation*}
Here $\mb b_0 \in (0,\infty)^M$ denotes the budget, and $c_t\geq 0$, $\mb a_t\geq \mb 0$ denote the revenue and the resource consumed when $x_t  =1$. If all pairs $(c_t,\mb a_t)$ are known, then the problem is a linear programming problem. However, when $(c_t,\mb a_t)$ are random variables which are revealed at stage $t$, the problem becomes a sequential decision making process with stochasticity, which we model using the framework of \eqref{eq:minmax-online}.

More precisely, we consider a problem with $T = 5$, $M = 10$, and $\mb b_0= 10\cdot \mb 1\in \R^M$.  We sample $\xi_1$ once uniformly from $[0,2]\times[1,5]^M$ and fix $E_1=\{\xi_1\}$. For $t=2,\ldots,T$, each $\bx_t = (c_t,\mb a_t)$ (here $c_t,\mb a_t$ are random variables/vectors), conditioned on $\bx_{1:(t-1)} = \xi_{1:(t-1)}$, is sampled uniformly in $E_t(\xi_{1:(t-1)})$, where $E_t(\xi_{1:(t-1)})$ consists of $10$ i.i.d. vectors uniformly sampled in $[0,2]\times [1,5]^M$. 

Let $\mathcal X_t=[0,1]\times \{\tilde{\mb b}\in \R^M,~\mb 0\leq \tilde{\mb b} \leq \mb b_0\}$, then we consider the following problem: 
\begin{align}\label{eq:rm_exp}
    &\quad \max_{(X_1,B_1)\in \mathcal P_1(\mathcal X_1),\ldots,(X_T,B_T)\in \mathcal P_T(\mathcal X_T)}\sum_{t=1}^T\E[ \mb c_tX_t(\bx_{1:t})]\\
    & s.t.~\mb a_t X_t(\bx_{1:t}) + B_t(\bx_{1:t}) \leq B_{t-1}(\bx_{1:(t-1)}),~ t=1,\ldots,T .\nonumber
\end{align}

Following the discussion in Section \ref{sec:lagrangian}, we consider the following saddle point reformulation, with $\mathcal Y_t = [0,5]^{M}$, 
\begin{equation}\label{eq:rm_exp_saddle}
    \phi_t(x_{t-1},b_{t-1},x_t, b_t,\mb y_t,\xi_t) := -c_tx_t  + \langle \mb y_t,\mb a_t x_t + b_t- b_{t-1}\rangle,\quad \xi_t = (c_t,\mb a_t)
\end{equation}

\textbf{Algorithms setup.} We use $\gamma = 5$, and MDSA is applied with $\gamma_l = \gamma/\sqrt{L}$ for $L=100$. The sampling distributions are perturbed by $d_t^{(0)}$, with $\delta = 0.2$. We use $0$ and $\mb 0$ as the initialization for all $X_t(\xi_{1:t}), Y_t(\xi_{1:t})$. For $B_t$, we consider two types of initialization: $B_t^{(0)} = \mb 0$, and $B_t^{(0)} = 2T(1-t/T)\cdot \mb 1$. To differentiate them, we add ``init" in the legends of the second type.

\textbf{Evaluation.} We use $\overline{Z}^*_{1:T}$, the $100$ iteration averaged output of MD under the true sampling distribution (with $\gamma_l = \gamma/\sqrt{100}$), as an approximation to the solution to the saddle point problem. The algorithms are evaluated through the (approximate) gaps (The approximate optimal solution $\overline{Z}^*_{1:T}$ used is the one corresponding to the same initialization. That is, the blue and green curves in Figure \ref{fig:rm_exp2}(a) are evaluated using the $100$-th iteration output initialized at $B_t^{(0)} = \mb 0$ and at $B_t^{(0)} = 2T(1-t/T)\cdot \mb 1$, respectively), objective function values, and the total budgets spent. 

We test MDSA and (MDSA, init) for $5$ runs. For MDSA, the running time is $\sim 220$ seconds, while for MD, the running time is $\sim 390$ seconds. In Figure \ref{fig:rm_exp2}, we present the mean of the evaluation metrics for $5$ runs using blue and green lines, representing initialization at $B_t^{(0)} = \mb 0$ and $B_t^{(0)} = 2T(1-t/T)\cdot \mb 1$ respectively.

\begin{figure}[htbp]
    \centering
    \begin{subfigure}[t]{0.32\textwidth}
        \centering
        \includegraphics[width = \textwidth]{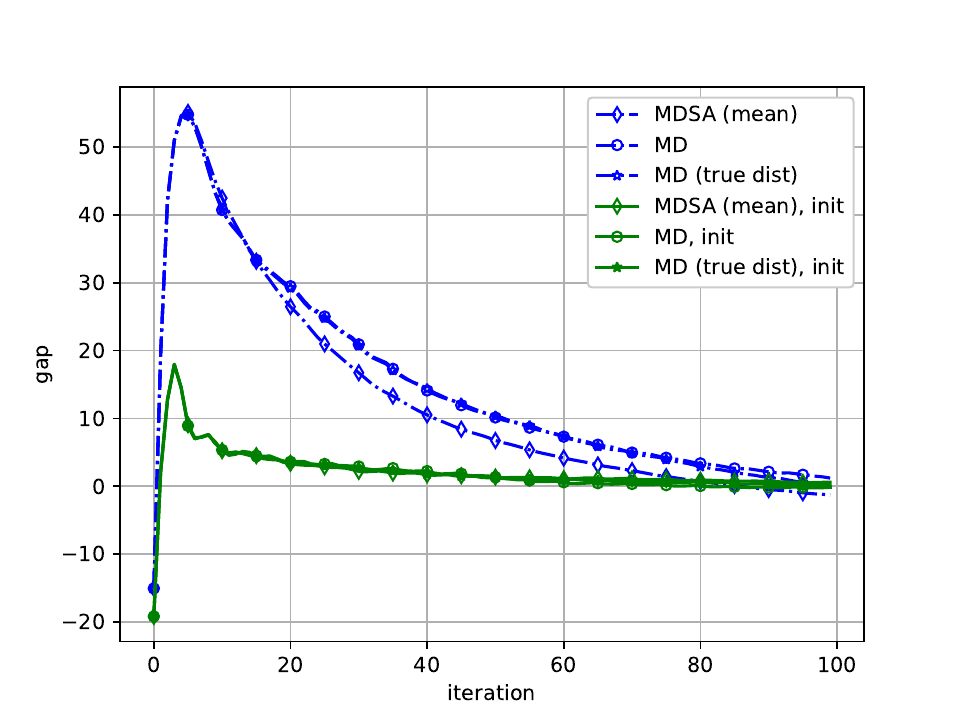}
        \caption{gap}
    \end{subfigure}
    \begin{subfigure}[t]{0.32\textwidth}
        \centering
        \includegraphics[width = \textwidth]{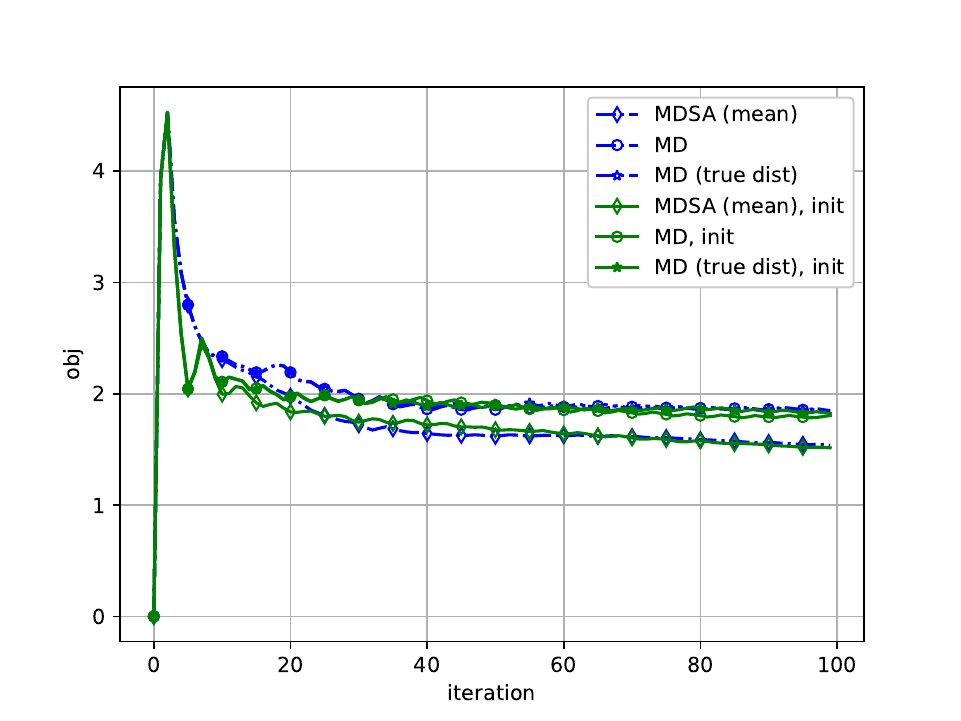}
        \caption{objective}
    \end{subfigure}
    \begin{subfigure}[t]{0.32\textwidth}
        \centering
        \includegraphics[width = \textwidth]{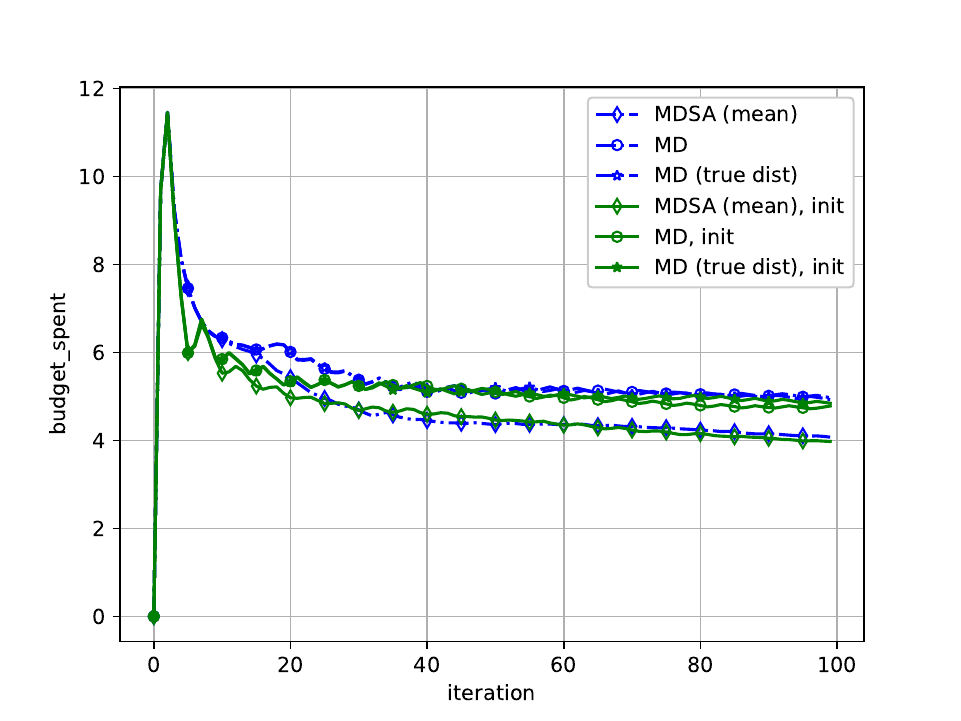}
        \caption{budget spent}
    \end{subfigure}
     \caption{Revenue management using MDSA Algorithm \ref{alg:hypo-ms-s}. Averages over $5$ runs for the gap: $\E[\phi(\overline{X}_{1:T}^{(l)},\overline{Y}_{1:T}^*)] - \E[\phi(\overline{X}_{1:T}^{*},\overline{Y}_{1:T}^{(l)})]$, objective: $\E[\sum_{t=1}^T c_t \overline{X}_{t}^{(l)}]$, budget spent: $\|\E[\sum_{t=1}^T \mb a_t \overline{X}_{t}^{(l)}]\|_{\infty}$. The ``init'' means the initialization $B_t^{(0)} = 2T(1-t/T)\cdot \mb 1$ is used, otherwise $B_t^{(0)} = \mb 0$.}\label{fig:rm_exp2}
\end{figure}

\textbf{Convergence of the (approximate) gap $\E[\phi(\overline{X}_{1:T}^{(l)}(\bx),\overline{Y}_{1:T}^*(\bx))] - \E[\phi(\overline{X}_{1:T}^{*}(\bx),\overline{Y}_{1:T}^{(l)}(\bx))]$.} In Figure \ref{fig:rm_exp2}(a), for both initializations, the term $\E[\phi(\overline{X}_{1:T}^{(l)}(\bx),\overline{Y}_{1:T}^*(\bx))] - \E[\phi(\overline{X}_{1:T}^{*}(\bx),\overline{Y}_{1:T}^{(l)}(\bx))]$ converges to $0$ when the exact gradients are used (MD and MD( true dist)), and to a slightly negative value when stochastic gradients are used (MDSA). To understand why the (approximate) gap can be negative, notice that 
\begin{align*}
&\E[\phi(\overline{X}^{(l)}(\bx),\overline{Y}^*(\bx))] - \E[\phi(\overline{X}^{*}(\bx),\overline{Y}^{(l)}(\bx))] =A_1-A_2,\\
& A_1 = \E[\phi(\overline{X}^{(l)}(\bx),\overline{Y}^*(\bx))] - \min_{X}\E[\phi(X(\bx),\overline{Y}^*(\bx))] \\
&\quad \quad \quad+ \max_{Y}\E[\phi(\overline{X}^{*}(\bx),{Y}(\bx))] -\E[\phi(\overline{X}^{*}(\bx),\overline{Y}^{(l)}(\bx))],\\
&A_2 =  \max_{Y}\E[\phi(\overline{X}^{*}(\bx),{Y}(\bx))] -\min_{X}\E[\phi(X_{1:T}(\bx),\overline{Y}^*(\bx))].
\end{align*}
where we omit all subscripts $1:T$ (e.g. $\overline{X}_{1:T}^{(l)}= \overline{X}^{(l)}$), and $\min_X$ is over $X_t\in \mathcal P_t(\mathcal X_t)$ for $t=1,\ldots,T$ and similarly for $\max_Y$. Here $A_2$ is the nonnegative primal--dual gap of the reference pair. The plotted difference can therefore be negative when $A_2>A_1$. Under the assumptions of Theorem \ref{thm:MDSA-2}, a reference generated after $N_{\rm ref}$ iterations with unbiased conditional gradients and step size $\gamma/\sqrt{N_{\rm ref}}$ satisfies $\E_{\internal}[A_2]=O(N_{\rm ref}^{-1/2})$; here $N_{\rm ref}=100$.

\textbf{Objective values $\E[\sum_{t=1}^T \mb c_t \overline{X}_{t}^{(l)}(\bx_{1:t})]$ and budget spent $\|\E[\sum_{t=1}^T \mb a_t \overline{X}_{t}^{(l)}(\bx_{1:t})]\|_{\infty}$.} From Figure \ref{fig:rm_exp2} (b) and (c), we see that for all settings, during the first 2 iterations, the objective values are increasing, and at iteration $2$, the budget spent exceeds the budget $\mb b_0 = 10 \cdot \mb 1$, thereby the constraints are violated (The violation of the constraints \textit{in expectation} implies that there exists at least one scenario such that the constraints are violated). After that, objective values and the constraints are converging to values that do not depend on the initialization or the bias in the gradient sampling distribution. However, the values depend on whether the stochastic gradients are used. This suggests that randomness in stochastic approximations may slightly compromise solution quality.

\end{APPENDICES}

\bibliographystyle{informs2014}
\bibliography{myBib}
\end{document}